\documentclass[a4paper,11pt,leqno]{article}
\usepackage{amssymb, amsmath, amsthm, graphicx}

\newtheorem{thm}{Theorem}[section]
\newtheorem{lem}[thm]{Lemma}
\newtheorem{cor}[thm]{Corollary}
\newtheorem{prop}[thm]{Proposition}
\newtheorem{defn}{Definition}[section]
\newtheorem{assump}{Assumption}[section]
\newtheorem{rem}{Remark}[section]

\numberwithin{equation}{section}

\renewcommand{\a}{\alpha}
\renewcommand{\b}{\beta}
\newcommand{\e}{\varepsilon}
\newcommand{\de}{\delta}
\newcommand{\fa}{\varphi}
\newcommand{\ga}{\gamma}
\renewcommand{\k}{\kappa}
\newcommand{\la}{\lambda}
\renewcommand{\th}{\theta}
\newcommand{\si}{\sigma}

\renewcommand{\t}{\tau}
\newcommand{\om}{\omega}
\newcommand{\De}{\Delta}
\newcommand{\Ga}{\Gamma}
\newcommand{\La}{\Lambda}
\newcommand{\Om}{\Omega}
\newcommand{\lan}{\langle}
\newcommand{\ran}{\rangle}
\newcommand{\na}{\nabla}

\def\R{{\mathbb{R}}}
\def\N{{\mathbb{N}}}
\def\Z{{\mathbb{Z}}}

\allowdisplaybreaks

\newcommand{\CerrH}{C_{err}G_\e}

\title{Sharp interface limit for stochastically perturbed \\
mass conserving Allen-Cahn equation}
\author{Tadahisa Funaki$^{1)}$ and Satoshi Yokoyama$^{2)}$}
\date{\today}

\begin{document}
\maketitle

\begin{abstract}
\noindent
This paper studies the sharp interface limit for a mass conserving Allen-Cahn 
equation added an external noise and derives a stochastically 
perturbed mass conserving mean curvature flow in the limit.
The stochastic term destroys the precise conservation law, instead the 
total mass changes like a Brownian motion in time.  For our equation,
the comparison argument does not work, so that to study the limit
we adopt the asymptotic expansion method, which extends
that for deterministic equations used in Chen et al.\ \cite{CHL}.  
Differently from the deterministic case, each term except the leading term
appearing in the expansion of the solution in a small 
parameter $\e$ diverges as  $\e$ tends to $0$, since our equation
contains the noise which converges to a white noise  and the products or the
powers of the white noise diverge.  To derive the error estimate for our
asymptotic expansion, we need to establish the Schauder estimate for
a diffusion operator with coefficients determined from higher order derivatives
of the noise and their powers.  We show that one can choose the noise
sufficiently mild in such
a manner that it converges to the white noise and at the same time its 
diverging speed is slow enough for establishing a necessary error estimate.
\footnote{
\hskip -6mm
$^{1),2)}$
Graduate School of Mathematical Sciences,
The University of Tokyo, Komaba, Tokyo 153-8914, Japan.
e-mails: funaki@ms.u-tokyo.ac.jp, satoshi2@ms.u-tokyo.ac.jp}
\footnote{
\hskip -6mm
\textit{Abbreviated title $($running head$)$: Sharp interface limit 
for stochastic mass conserving Allen-Cahn equation}}
\footnote{
\hskip -6mm
\textit{2010MSC: primary 60H15, secondary 35K93, 74A50.}}
\footnote{
\hskip -6mm
\textit{The author$^{1)}$ is supported in part by the JSPS 
KAKENHI Grant Numbers $($B$)$ 26287014 and 26610019.}}
\end{abstract}

\section{Introduction and main results}

First, we introduce the stochastic mass conserving Allen-Cahn equation.
Our goal is to discuss its sharp interface limit.  The mass conservation law
forces us to introduce a scaling for the noise different from the stochastic 
Allen-Cahn equation and prevents the comparison argument, which was 
useful for the stochastic Allen-Cahn equation.  Our method is the 
asymptotic expansion, in which powers of derivatives of the noise
appear repeatedly, and our main effort is devoted to controlling
these diverging terms.  The evolutional law of the limit hypersurface
is described by the mass conserving mean curvature flow with
a nonlocal multiplicative white noise term.

\subsection{Stochastic mass conserving Allen-Cahn equation and
its background}

We consider the solution $u=u^\e(t,x)$ of the following stochastic 
partial differential equation \eqref{eq:1}
in a bounded domain $D$ in $\R^n$ having a smooth boundary $\partial D$:
\begin{equation}\label{eq:1}
\left\{
\begin{aligned}
& \frac{\partial u^\e}{\partial t} = \De u^\e + \e^{-2} \left( f(u^\e) - 
- \hspace{-3.9mm}\int_D f(u^\e) \right) + \a \dot{w}^\e(t),
\quad \text{ in } D \times \R_+, \\
& \frac{\partial u^\e}{\partial \nu} = 0, 
\hspace{66mm}\quad \text{ on } \partial D \times \R_+, \\
& u^\e(0,\cdot) = g^\e(\cdot),\hspace{56mm}\quad \text{ in } D,
\end{aligned}
\right.
\end{equation}
where $\e>0$ is a small parameter,
$\a>0$, $\nu$ is the inward normal vector on $\partial D$,
$\R_+ = [0,\infty)$,
\begin{align*}
- \hspace{-3.9mm}\int_D f(u^\e) 
= \frac1{|D|} \int_D f(u^\e(t,x))dx,
\end{align*}
$g^\e$ are continuous functions having the property \eqref{eq:g} stated below,
or more precisely, satisfying the conditions 
\eqref{eq:TH assump 1}--\eqref{eq:TH assump 3}.  The noise
$\dot{w}^\e(t)$ is the derivative of $w^\e(t)\equiv w^\e(t,\om) \in C^\infty(\R_+)$
in $t$ defined on a certain probability space
$(\Om,\mathcal{F},P)$ such that $w^\e(t)$ converges to 
a one-dimensional  standard Brownian motion $w(t)$ as $\e\downarrow 0$
in a suitable sense; see \eqref{eq:definition 2 of g} and Assumption
\ref{ass:1} below.  We assume that the reaction term $f\in C^\infty(\R)$ is 
bistable and satisfies the following three conditions:
\begin{align*}
\text{(i)}&\quad f(\pm 1)=0, \, f'(\pm 1)<0, \, \int_{-1}^1 f(u) du =0,
\\
\text{(ii)}&\quad f \text{ has only three zeros }\pm1 \text{ and one another between }\pm1,
\\
\text{(iii)}&\quad \text{there exists } \bar{c}_1 >0 \, \text{ such that } 
f^\prime(u) \leq \bar{c}_1 \text{ for every }u\in\R.
\end{align*}
The last equality in (i) is called the balance condition.
A typical example is $f(u)=u-u^3$. 
From (i) and (ii), it follows that there exists a unique increasing solution $m$: 
$\R\longrightarrow (-1,1)$ of
\begin{align}\label{eq:m standing wave}
m^{\prime \prime}+f(m)=0\quad\text{on }\R,\quad m(\pm\infty)=\pm1,\quad m(0)=0,
\end{align}
and $m$ is called the traveling wave (or standing wave) solution.
The function $m=m(\rho), \rho\in \R,$ satisfies 
\begin{align}\label{eq:decay of m}
\partial_\rho^k (m(\rho)\mp1)=O(e^{-\zeta|\rho|}),\quad \text{as }
\rho\rightarrow\pm \infty,
\end{align}
for $k=0,1,2,\cdots,$ where $\zeta=\min\{\sqrt{-f^\prime(-1)},\sqrt{-f^\prime(1)}\}>0$.

For a mass conserving Allen-Cahn equation without noise, that is \eqref{eq:1} with $\a=0$,
the existence and uniqueness results are established by \cite{EG}, \cite{ES}, \cite{H} and
its sharp interface limit as $\e\downarrow0$ is studied by Chen et al.\ \cite{CHL}.
On the other hand, for a stochastic Allen-Cahn equation, that is \eqref{eq:1} 
without the averaged reaction term, the sharp interface limit is studied by 
Funaki \cite{F95}, \cite{F99}, Lions and Souganidis \cite{LS98-2} and 
Weber \cite{We}; see also \cite{F03}, \cite{F-SB} which give a brief survey.
The study of the sharp interface limit of the stochastic mass conserving 
Allen-Cahn equation \eqref{eq:1} is a natural extension 
and combination of these two problems.
Note that the noise term considered in \cite{F99}, \cite{We} was
scaled as $\e^{-1}\a \dot{w}^\e(t)$.  This is very different from
$\a \dot{w}^\e(t)$ in \eqref{eq:1}.
In other words, our dynamics is more sensitive to the noise.
We will give a heuristic explanation for this difference in Section \ref{sec:2.1}.
Note also that the comparison argument used in \cite{F99}, \cite{We} does not
work for the equation \eqref{eq:1}.  This is the main technical difference between
our equation and the stochastic Allen-Cahn equation.

\subsection{Limit dynamics, conservation law and related results}

Our goal is to show that the solution $u^\e(t,x)$ of \eqref{eq:1} converges as $\e\downarrow 0$
to $\chi_{\ga_t}(x)$ with certain hypersurface $\ga_t$ in $D$,
where $\chi_{\ga}(x) = +1$ or $-1$ according to the outside or inside of the hypersurface $\ga$,
if this holds for the initial data $g^\e$ with a certain $\ga_0$, 
and the time evolution of $\ga_t$ is governed by 
\begin{equation}\label{eq:2}
V = \k - - \hspace{-3.9mm}\int_{\ga_t} \k + \frac{\a|D|}{2|\ga_t|}\circ\dot{w}(t),\quad t \in [0,\sigma],
\end{equation}
up to a certain stopping time $\sigma>0$ (a.s.), 
where $V$ is the inward normal velocity of $\ga_t$, 
$\k$ represents the mean curvature of $\ga_t$ multiplied by $n-1$,
$- \hspace{-3.9mm}\int_{\ga_t} \k = \frac1{|\ga_t|} \int_{\ga_t} \k d \bar s$,
$\dot{w}(t)$ is the white noise process and $\circ$ means 
the Stratonovich stochastic integral.
When $\a=0$, the equation \eqref{eq:2} coincides with the limit of the mass conserving 
Allen-Cahn equation studied in \cite{CHL}.  Note that, in \cite{CHL}, the sign
of $V$ is taken opposite.  When $\a=0$, the mass of the solution $u^\e$ of \eqref{eq:1} 
is conserved, namely,
\begin{align}\label{eq:mass cons 1}
&\frac{1}{|D|}\int_D u^\e(t,x) dx = C, 
\end{align}
holds for some constant $C\in\R$.
On the other hand, in the case where the fluctuation caused by $\a w^\e(t)$ is added, 
the rigid mass conservation law is destroyed and in place of \eqref{eq:mass cons 1},
we have the conservation law in a stochastic sense
\begin{align}\label{eq:mass cons 2}
&\frac{1}{|D|}\int_D u^\e(t,x) dx = C+ \a w^\e(t), \quad t\in \R_+,
\end{align}
which can be derived by integrating \eqref{eq:1} over $D$:
\begin{align}\label{eq:mass cons 3}
\frac1{|D|}\int_D \Bigl(\partial_tu^\e(t,x) - \a \dot{w}^\e(t) \Bigr)dx = 0.
\end{align}
The equation \eqref{eq:mass cons 2} implies that the total mass per volume 
behaves like a Brownian motion multiplied by $\a$ as $\e$ tends to $0$.

The equation \eqref{eq:1} with $\a=0$ and without the averaged reaction
term is called the Allen-Cahn equation.  It is well-known that the mean curvature
flow $V=\k$ appears in the limit for this equation; cf., \cite{B13}, \cite{F-SB}.
For the stochastic Allen-Cahn equation, i.e., \eqref{eq:1} without the 
averaged reaction term and with the noise differently scaled as we 
mentioned above, the limit dynamics is given by $V=\k + \bar\si \a \dot{w}(t)$,
where $\bar\si$ is the inverse surface tension defined below, see \cite{F99}, (1.5).
In this case, a simple additive noise appears in the limit, while in our case
the limit dynamics \eqref{eq:2} has a multiplicative noise and its coefficient contains a
nonlocal term $|\ga_t|^{-1}$.  This is due to the effect of the conservation law.

\subsection{Formulation of main results}

We take an integer $K$ satisfying $K>\max(n+2,6)$ and fix it throughout the paper.
This will be necessary for the proof of Theorem \ref{thm:err} later.
Let $w^\e=w^\e(t)\equiv w^\e(t,\omega)$, $0<\e\le 1$, $t\in\R_+$, 
$\omega\in\Omega$ be a family of $(\mathcal{F}_t)$-adapted stochastic processes
defined on a probability space $(\Omega,\mathcal{F},P)$ equipped with 
the filtration $(\mathcal{F}_t)_{t\ge 0}$, which satisfy that $w^\e(0)=0$,
$w^\e(\cdot)\in C^\infty(\R_+)$ in $t$ a.s.\ $\om$ and
\begin{align}\label{eq:definition 2 of g}
&\lim_{\e \downarrow 0}\|w^\e-w\|_{C^{\th}([0,T])}=0,\quad \text{a.s.},
\end{align}
for every $T>0$ and some $\th\in(0,\frac12)$, where
$w(t)$ is an $(\mathcal{F}_t)$-Brownian motion satisfying $w(0)=0$ and
\begin{align}
&\|u\|_{C^{\th}([0,T])}=\sup_{t\in[0,T]}|u(t)|+
\sup_{\substack{0\leq s,t \leq T \\ s\neq t}}\frac{|u(t)-u(s)|}{|t-s|^{\th}}.
\end{align}

To prove our main theorem, we need two assumptions formulated as follows:
\begin{assump}  \label{ass:1}
For every $T > 0$, there exists $H_\e \ge 1, 0<\e\le 1,$ such that
\begin{align}\label{eq:definition of g}
&\sup_{t\in[0,T],\,\omega\in\Omega}|\frac{d^k}{dt^k}w^\e(t,\omega)|\leq H_\e,
\quad k=1,2,\cdots,n_1(K)+1,\\
&\lim _{\e\downarrow 0}H_\e=\infty,\quad 
\lim _{\e\downarrow 0} \frac{H_\e^{2n_1(K)}}{\log\log|\log\e|}=0,
\label{eq:definition of g(2)}
\end{align}
where $n_1(K)\in\N$ is the number determined from $K$ by Proposition
\ref{prop:Lemma order of e} below.
\end{assump}
Two examples of our mild noise $w^\e$ satisfying \eqref{eq:definition 2 of g} 
and Assumption \ref{ass:1} will be given in Section 4.1. 

\begin{assump}  \label{ass:2}
There exist stopping times $\sigma^\e$ and $\sigma$ such that
$V^{\a\dot{w}^\e}$ {\rm(}resp.\ $V${\rm)},
the solution of \eqref{eq:2-e} below with $v=\a\dot{w}^\e$
{\rm(}resp.\ \eqref{eq:2}{\rm)}, exists uniquely in $[0,\sigma^\e]$
 {\rm(}resp.\ $[0,\sigma]${\rm)}.
In addition, $\sigma^\e>0$ and $\sigma>0$ hold {\rm a.s.}
Furthermore, for every $T>0$ and $m\in \N$, the joint variable
$(\sigma^{\e}, d^\e(t\wedge\sigma^\e))$ 
$\in \R_+\times C([0,T],C^m(\mathcal{O}))$
converges in this space to $(\sigma, d(t\wedge\sigma))$
as $\e \downarrow 0$ in {\rm a.s.}-sense,
where $d^\e(t)$ {\rm(}resp.\ $d(t)${\rm)} is the signed distance determined by the 
hypersurface $\gamma^{\a\dot{w}^\e}_t$ {\rm(}resp.\ $\gamma_t${\rm)},
which is negative inside $\gamma^{\a\dot{w}^\e}_t$ {\rm(}resp.\ $\gamma_t${\rm)},
and $\mathcal{O}$ is an open neighborhood of $\ga_0$; see Theorem \ref{thm:Th1}
and \cite{We}.
\end{assump}

We will show in Section 5 that Assumption \ref{ass:2} holds in law sense
under a two-dimensional setting as long as the limit curve $\ga_t$ is convex. 
Applying Skorohod's representation theorem for joint variables
$(w^\e,\si^\e,d^\e(t\wedge \si^\e))$, by changing the probability space
$(\Om,\mathcal{F},P)$ if necessary, one can realize the convergence in a.s.-sense as in 
\eqref{eq:definition 2 of g} and Assumption 1.2.

Assumption 1.2 implies uniform bounds on spatial derivatives of the distance functions
in $\e>0$ locally in time, see Section 3.3.1.
We also need bounds on $t$-derivatives of the hypersurface $\ga_t^{\a\dot{w}^\e}$ 
by means of a certain norm of the noise $\dot{w}^\e$.
This will be formulated precisely as Assumption \ref{ass:3} in Section 3.3.2 and
shown under two-dimensional settings in Section 5.

The aim of this paper is to prove the following theorem.
\begin{thm}\label{thm:Th1}
Let $\gamma_0$ be a smooth hypersurface in $D$ without boundary
with finitely many connected components and it has the form 
$\gamma_0= \partial D_0$ with a smooth domain $D_0$ such that
$\overline{D}_0\subset D$.  Suppose that a local solution
$\Gamma = \cup _{0\leq t < \sigma}(\gamma_t \times \{t\})$ of \eqref{eq:2} 
up to the stopping time $\si >0$ {\rm(}a.s.{\rm)} satisfying
$\gamma_t \subset D$ for all $t\in[0,\sigma]$ uniquely exists {\rm(}a.s.{\rm)}.
Furthermore, let us assume three Assumptions \ref{ass:1}, \ref{ass:2} and \ref{ass:3}. 
Then, one can find a family of continuous functions $\{g^\e(\cdot)\}_{\e\in(0,1]}$ satisfying
\begin{align}\label{eq:g}
\lim_{\e\downarrow 0}g^\e(x)=\chi_{\ga_0},
\end{align}
such that $u^\e(t\wedge \sigma^\e\wedge \t,\cdot)$
converges to $\chi_{\gamma_{t\wedge \sigma\wedge \t}}(\cdot)$ 
in $C(\R_+,L^2(D))$ as $\e \downarrow 0$
in {\rm a.s.}-sense, where $u^\e$ is the
solution of \eqref{eq:1} with initial value $g^\e$ and $\t=\t(\om)>0$ is that
given below Assumption 3.1.
\end{thm}

If Assumption 1.2 holds in law sense, by the observation we gave above,
Theorem \ref{thm:Th1} holds also in law sense.  More precise conditions
for $g^\e$ are formulated in \eqref{eq:TH assump 1}--\eqref{eq:TH assump 3}.

The right hand side of the first equation in \eqref{eq:1} contains the averaged reaction term, 
and hence, as we already pointed out,
one cannot directly apply the comparison principle to estimate its solution. 
This is a difference from the stochastic Allen-Cahn equation treated in \cite{F99}, \cite{We}.
Our method for the proof of Theorem \ref{thm:Th1} is an extension of
the asymptotic expansion used in \cite{CHL}.   Recall that $w(t)$ is not in $C^1$-class 
in $t$ and therefore for any smooth sequence $\{w^\e(t)\}_{0<\e\le 1}$ converging toward $w(t)$,
the products or the powers of its time derivative $\{\dot{w}^\e(t)\}_{0<\e\le 1}$ 
and higher order time derivatives always diverge as 
$\e\downarrow0$.  Due to this, each term in inner and outer solutions constructed by 
the asymptotic expansion except its leading term explodes as $\e\downarrow0$ 
because it contains some powers of the time derivatives of $w^\e$.
However, by choosing the sequence $\dot{w}^\e(t)$ in a suitable manner that
its divergent speed is slow enough as in Assumption \ref{ass:1}, one can control the diverging terms. 
Indeed, the speed of divergence of each $k$th 
inner and outer solutions having prefactor $\e^k$ can be controlled
once we can make the divergent speed of the powers of time derivatives of  $w^\e(t)$
slower than $\e^{-k}$.  This is one of the key points in the proof of Theorem \ref{thm:Th1}.

The paper is organized as follows.
In Section 2, we first give a heuristic derivation of the evolutional law \eqref{eq:2} of $\ga_t$,
as a result of the combination of \eqref{eq:lambda-0} and \eqref{eq:V + K}.
Then, we introduce the asymptotic expansion of the solution $u^\e$ of 
the equation \eqref{eq:1} in $\e$ in details.  It turns out that one needs to analyze 
the asymptotic expansion up to the $K$th order term with $K>\max(n+2,6)$; 
cf.\ Lemma 2 of \cite{CHL} and Theorem \ref{thm:err} below.
In Section 3, we define an approximate solution and show the estimates on each term
in the asymptotic expansion.  This is accomplished by carefully studying the
Schauder estimate for a diffusion operator with diverging coefficients.
Then in Section 4, we give the proof of Theorem \ref{thm:Th1}.
In Section 5, we discuss the stochastic partial differential equation (SPDE)
corresponding to \eqref{eq:2} under the situation that $n=2$ and the interfaces
stay convex and show that the SPDE corresponding to \eqref{eq:2} has a 
unique local solution in such case.  Assumption 1.2 in law sense and Assumption 3.1
are shown in this situation.  

\begin{rem}
Weber \cite{We} established the sharp interface 
limit for the stochastic Allen-Cahn equation in a.s.-sense with 
the choice of $\dot{w}^\e(t)$ such as the first example given in 
Section 4.1.  In his argument, it was essential that the SPDE 
describing the dynamics of the limit hypersurface
$\ga_t$ in terms of the signed distance $d$ has an additive noise.
This is not the case for our limit dynamics \eqref{eq:2}, so that
his argument doesn't work for showing Assumption 1.2 under 
a more general setting than we discuss in Section 5.
\end{rem}

\section{Asymptotic expansion of the solution of \eqref{eq:1}}

\subsection{Heuristic argument for the derivation of \eqref{eq:2}}  \label{sec:2.1}

Before starting the proof of Theorem \ref{thm:Th1},
it might be worthy to give a heuristic derivation of 
the evolutional law \eqref{eq:2} of the limit hypersurface $\ga_t$
from the stochastically perturbed mass-conserving
Allen-Cahn equation \eqref{eq:1}.

Note that the scaling for the noise term in \eqref{eq:1}
is essentially different from 
that for the stochastic Allen-Cahn equation.  In fact, without the
term $- - \hspace{-4mm}\int_D f(u^\e)$ in \eqref{eq:1}, the
proper scaling for the noise term was $\e^{-1} \a \dot{w}^\e(t)$
rather than $\a \dot{w}^\e(t)$, see \cite{F99}.  We give a heuristic
argument to explain the reason for this difference.  In particular,
we will see that the averaged reaction term behaves as
$- \hspace{-3.9mm}\int_D f(u^\e) = O(\e)$ as $\e\downarrow 0$ so that
$\e^{-2} - \hspace{-3.9mm}\int_D f(u^\e) \approx \e^{-1} \la_0(t)$
with a certain $\la_0(t)$, and the evolution of $\la_0(t)$ is governed
by the noise $\a \dot{w}^\e(t)$ of $O(1)$.

Our basic ansatz is that $u^\e(t,x) \approx \pm 1$ as $\e\downarrow 0$.  
We actually assume this at $t=0$ as in \eqref{eq:g}.  This implies 
$f(u^\e)\approx 0$ so that $a^\e(t) := - \hspace{-3.9mm}\int_D f(u^\e)$ should be small;
see Remark \ref{rem:2-a} below.
Conversely, if $a^\e(t)$ is small, the main term of the reaction term becomes
$\e^{-2}f(u^\e)$ so that the solution $u^\e$ is pushed toward $\pm 1$ and we can expect
our ansatz should be true.  Anyway, this observation suggests that, instead of $m$ 
defined by \eqref{eq:m standing wave}, it might be better to consider
the perturbed traveling wave solutions with the reaction term $f$ replaced by 
$f-a^\e(t)+\e^2 v(t)$ and $v(t) = \a \dot{w}^\e(t)$.  We denote 
 $v(t)$ for $\a \dot{w}^\e(t)$, since this term could be regarded as $O(1)$ 
as $\e\downarrow 0$; see \eqref{eq:1-v} below.  More precisely,
for $a\in \R$ with small $|a|$, define the traveling wave 
solution $m=m(\rho;a), \rho\in \R$ and its speed $c=c(a)$ by
\begin{equation}\label{eq:d}
\left\{\begin{aligned}
m'' + c m' + \{f(m)-a\} = 0,& \quad  \rho \in \R, \\
m(\pm \infty) = m_\pm^*, \quad m(0)& =0,   \end{aligned}
\right. 
\end{equation}
where $m_\pm^* \equiv m_\pm^*(a) = \pm 1 + O(a) \, 
(a\to 0) \,$ are solutions of $f(m_\pm^*) - a =0$.  
It is easy to see that $c_0:= c'(0) = 2/\int_\R m'(\rho)^2 d\rho$,
see below.  Another expression of $c_0$ is also known: $c_0= 
\frac{\sqrt{2}}{\int_{-1}^1 \sqrt{V(u)} du}$
where $V(u) = \int_u^1 f(v)dv$, see \cite{F99}. 

Our guess for the behavior of the solution $u^\e$ of \eqref{eq:1} is the following:
\begin{equation} \label{eq:assume}
u^\e(t,x) = m\Big( \frac{d(t,x)}\e; a^\e(t)-\e^2 v(t)\Big) + O(\e),
\end{equation}
as $\e\downarrow 0$, where $d(t,x)$ is the signed distance between $x\in D$ 
and the limit hypersurface $\ga_t$.
Then, denoting $m(\cdot; a^\e(t)-\e^2 v(t))$ simply by $m$, we have
\begin{align*}
0 &= \int_D \Big\{\Delta u^\e +\frac{1}{\e^2}\bigl(f(u^\e)-- \hspace{-3.9mm}\int_D f(u^\e)\bigr)\Big\}dx \\
&\approx \frac1{\e^2} \int_D \Big\{m''(\tfrac d\e)+\e m'(\tfrac d\e)\Delta d + f\Big(m(\tfrac d\e)\Big)
-a^\e(t)+ \e^2 v(t)\Big\}dx - |D| v(t) \\
&= \frac1{\e^2} \int_D \Big\{-c\big(a^\e(t)-\e^2 v(t)\big)m'(\tfrac d\e) + \e m'(\tfrac d\e)\Delta d \Big\} dx - |D| v(t) \\
&\approx \frac1{\e^2}\int_D m'(\tfrac d\e) \{ -c_0 a^\e(t) + \e\Delta d \} dx - |D| v(t)\\
&\approx\frac1{\e^2}\int_\R \int_{\ga_t} m'(\tfrac{r}{\e}) \{ -c_0 a^\e(t) + \e\k(t,\bar s)+\e rb(t,\bar s)\}dr d\bar s- |D| v(t) \\
&\approx\int_\R \int_{\ga_t} m'(\rho) 
\Big\{ -\frac{1}{\e}c_0 a^\e(t) + \k(t,\bar s) + \e \rho b(t,\bar s)\Big\} d\rho d\bar s- |D| v(t) \\
&= (m_+^*-m_-^*)  \int_{\ga_t}
\Big\{ -c_0 \frac{a^\e(t)}{\e} + \k(t,\bar s) \Big\}d\bar s - |D| v(t) + O(\e).
\end{align*}
Here, the first line is a consequence of $\int_D\De u^\e dx=0$ which holds under
the Neumann condition at $\partial D$, we apply \eqref{eq:assume} for the second line
recalling \eqref{eq:change variable 02} below for $\De u^\e$,
the third line follows from \eqref{eq:d} with $a= a^\e(t)-\e^2 v(t)$,
the fourth line by $c(a^\e-\e^2 v) = c_0 a^\e +O((a^\e)^2)+O(\e^2)$ and $c(0)=0$,
the fifth line from (30), (40) of \cite{CHL}, that is, 
$\De d = \k(t,\bar s)+rb(t,\bar s)+O(\e)$ with $b(t,\bar s)=-\na d\cdot\na\De d$ for $\bar s\in \ga_t$ denoting the volume element of $\ga_t$ by $d\bar s$,
and the sixth line follows by the change of variables $r=\e\rho$.
The above computation implies that $a^\e(t)$ should be of order $O(\e)$ and, defining $\la_0(t)$ as
\begin{equation} \label{eq:lambda}
\frac1\e a^\e(t) \equiv \frac1\e  - \hspace{-3.9mm}\int_D f(u^\e(t,\cdot)) 
= \la_0(t) + O(\e),
\end{equation}
since $m_\pm^*= \pm 1 + O(a^\e(t)-\e^2 v(t)) = \pm 1+ O(\e)$, 
we obtain
\begin{align} \label{eq:lambda-0}
2 c_0 \la_0(t) |\ga_t| = 2 \int_{\ga_t}\k d\bar s - |D| v(t),
\end{align}
in the limit $\e\downarrow 0$.

It will be clear that this condition is necessary for the first term of
\eqref{eq:eq of mass con 1} to vanish and used in \eqref{eq:lambda 0} to
determine $\la_0(t)$; note that
$c_0= \bar\si :=2 (\int_\R m'(\rho)^2 d\rho)^{-1}$ holds as we will see
below, where $\bar\si$ is called the inverse surface tension.
In particular, \eqref{eq:assume} and \eqref{eq:lambda} suggest
$$
u^\e(t,x) \underset{x=(\bar s,r), r\to\pm\infty}{\sim}
m_\pm^*(a^\e(t)-\e^2 v(t)) \sim \pm 1 + \e \frac{\la_0(t)}{f'(\pm 1)},
$$
since $0=f(m_\pm^*) - a \sim f'(\pm 1) (m_\pm^* \mp 1) - a$, which
implies $m_\pm^* \sim \pm 1 + \frac{a}{f'(\pm 1)}$.  This exactly coincides with
the formula \eqref{eq:u0 outer}
for the asymptotic behavior of the outer solutions.

Once \eqref{eq:lambda-0} is obtained, \eqref{eq:2} could be derived from
\begin{equation*}
V= \k -\bar\si \la_0(t) \quad \text{ on }\ga_t,
\end{equation*}
which is obtained as a solvability condition for $u_0$ appearing
in the expansion of $u^\e$, see \eqref{eq:V + K} below.
Later, we will consider the expansion of $u^\e$ based on $m(\frac{d(t,x)}\e;0)$
with $a=0$ rather than that introduced in \eqref{eq:assume},
since the leading orders are same.

We finally comment on the identity $c_0=\bar\si$.
The smoothness of $c(a)$ in $a$ is shown 
in the Appendix of \cite{FH}.  We compute for $m=m(\cdot;a)$ as
\begin{align*}
0&= - \int_\R \frac12\Big\{(m')^2\Big\}' d\rho
= - \int_\R m''\cdot m' d\rho \\
&= \int_\R \{ c(a) m' + f(m) -a\} m'd\rho \\
& = c(a) \int_\R (m')^2d\rho + \int_{m_-^*}^{m_+^*} f(s)ds 
- a (m_+^*-m_-^*).
\end{align*}
Take the derivative of both sides in $a$ and set $a=0$ noting that $m'=m'(\rho;a)$
and $m_\pm^* = m_\pm^*(a)$.  This leads to the identity $c_0=\bar\si$ since
$m_\pm^*(a) = \pm 1+ O(a)$ as $a\to 0$.

\begin{rem} \label{rem:2-a}
If the condition \eqref{eq:g} does not hold for the initial data $g^\e$,
$a^\e(0)$ is not small in general.  In this case, $a^\e(t)$ may not be 
small as well.  This means that the reaction term $f(u)-a^\e(t)$ does 
not satisfy the balance condition, i.e., the last equality in the 
condition {\rm (i)} for $f$.  Thus, the situation is more close to 
that G\"artner \cite{Ga83} and others discussed at least in the 
non-random case; see Section 4.1 of \cite{F-SB} and also Hilhorst 
et.\ al.\ \cite{HMNW}.  In particular, the proper time scale and 
the limit dynamics should be totally different from ours.
\end{rem}

\subsection{Signed distance from $\ga_t$ and parametrization of $\ga_t$}
\label{sec:2.2}
%
Let us start more precise discussions.
The expansion of the solution $u^\e(t,x)$ of \eqref{eq:1} in $\e$ will be given only 
in $\e$ appearing in the reaction term and not that in the noise term.
To make this clear, we consider the following equation with an external force 
$v(t)$, which is deterministic (non-random) such that $v \in C^\infty(\R_+)$:
\begin{equation}\label{eq:1-v}
\left\{
\begin{aligned}
& \frac{\partial u^\e}{\partial t} = \De u^\e + \e^{-2} \left( f(u^\e) - 
- \hspace{-3.9mm}\int_D f(u^\e) \right) + v(t),
\quad \text{ in } D \times \R_+, \\
& \frac{\partial u^\e}{\partial \nu} = 0,
\hspace{61mm}\quad \text{ on } \partial D \times \R_+, \\
& u^\e(\cdot,0) = g^\e(\cdot), \hspace{51mm}\quad \text{ in } D.
\end{aligned}
\right.
\end{equation}
Nevertheless to say, the solution of \eqref{eq:1} is the same as that of \eqref{eq:1-v} with 
$v=\a\dot{w}^\e$.
In addition, we consider the hypersurface $\{\ga_t^v\}$ whose evolution is governed by
\begin{equation}\label{eq:2-e}
V^v= \k - - \hspace{-3.9mm}\int_{\ga_t^v} \k + \frac{|D|}{2|\ga_t^v|}v(t),
\end{equation}
where $V^v$ is the inward normal velocity of $\ga^v_t$.
Suppose that \eqref{eq:2-e} has a unique solution for $t\leq T^v$ with some $T^v>0$.
Under these settings, we will first expand the solution $u^\e = u^{\e,v}$
of \eqref{eq:1-v} in $\e$ 
based on the solution $\ga_t = \ga_t^v$ of \eqref{eq:2-e}. 
Next, we will estimate each term appearing in the expansion by a suitable norm of $v$;
see Proposition \ref{prop:Lemma order of e} and Lemma \ref{lem:Lemma 2-1}.
Finally, in Sections 4.2--4.5, we will apply these lemmas taking $v(t)=\a\dot{w}^\e(t)$.

%
%
%
Let $d=d^v(t,x)$ be the signed distance of $x\in D$ to the hypersurface 
$\ga_t$, which is negative inside $\ga_t$; cf.\ Bellettini \cite{B13}, Chapter 1.
To parametrize the hypersurface $\ga_t$ on a fixed reference manifold,
we go along with \cite{B13}, Chapter 16 as follows.   Let $\mathcal{S}\subset\R^n$
be an  oriented compact $(n-1)$-dimensional submanifold without boundary and
with finitely many connected components being smoothly embedded in $\R^n$.
For each $s=(s^l)_{l=1}^n\in \mathcal{S}$, except some singular points,
$s^n$ is represented by other coordinates such that $s^n=s^n(s^1,\ldots,s^{n-1})$
and thus we can take $s=(s^l)_{l=1}^{n-1}$ as a local coordinate of $\mathcal{S}$.  
Near singular points for $s^n$, one may take other coordinates,
e.g., $(s^l)_{l=2}^n$, but we denote it by $s=(s^l)_{l=1}^{n-1}$ for simplicity.
Then, $\Big\{ \frac{\partial}{\partial s^l}\Big\}_{l=1}^{n-1}$ forms a basis of
the tangent space $T_s\mathcal{S}$ at $s\in\mathcal{S}$.
We parametrize $\ga_t, t \in [0,T]$ as $x=X_0(t,s)$ by $s=(s^l)_{l=1}^{n-1}\in \mathcal{S}$
such that $X_0\in C^\infty([0,T]\times \mathcal{S}, \R^n)$ and
the map $X_0(t,\cdot): \mathcal{S}\to \ga_t$ is homeomorphic 
for every $t\in [0,T]$.  In particular, $(\frac{\partial X_0(t,s)}
{\partial s^1},\ldots,\frac{\partial X_0(t,s)}{\partial s^{n-1}})$ forms a 
basis of the tangent space to $\ga_t$ at $x=X_0(t,s)$ for each $s\in \mathcal{S}$.

We denote by $\bold{n}(t,s)$ the unit outer normal vector on $\ga_t$
so that 
\begin{equation}  \label{eq:boldn}
\bold{n}(t,s)=\nabla d(t,X_0(t,s)).
\end{equation}
We define the Jacobian of the map $X_0(t,\cdot)$ as
\begin{equation}  \label{eq:2-J0}
J^0(t,s) = \text{det}\bigl[\bold{n}(t,s),\partial_{s^1}X_0(t,s),\ldots,\partial_{s^{n-1}}X_0(t,s)\bigr].
\end{equation}

\begin{rem}
Bellettini \cite{B13}, p.\ 251 introduces the parametrization $\fa_\e(t,s)$ of 
the hypersurface $\tilde\ga_t^\e$
determined by \eqref{eq:2.ga-tilde} below.  
Our parametrization $X_0(t,s)$ of $\ga_t$ by $\mathcal{S}$ 
is similar and corresponds to $\fa_0(t,s)$ in \cite{B13}. Chen et al.\
\cite{CHL}, (31) considers the parametrization by $U$ instead of
our $\mathcal{S}$ satisfying  
\begin{align} \label{eq: determinant}
J^0(t,s)=1,
\end{align}
but under this parametrization the set $U$ changes in $t$.  
To avoid this inconvenience, we follow \cite{B13}.
\end{rem}

Let $\delta>0$ be small enough such that the signed distance function 
$d(t,x)$ from $\gamma_t$ is smooth in the $3\delta$-neighborhood of $\ga_t$
and the distance between $\ga_t$ and $\partial D$ is larger than $3\delta$
for every $t\in[0,T^v]$.
A local coordinate $(r,s)\in (-3\delta,3\delta) \times \mathcal{S}$
of $x$ in a tubular neighborhood of $\ga_t$ is defined by
\begin{align}\label{eq:parameterize}
x=X_0(t,s)+r\bold{n}(t,s)=:X(t,r,s). 
\end{align}
Then its inverse function is given by
\begin{align*}
 r=d(t,x),\quad s=\bold{S}(t,x)=(S^{1}(t,x),\ldots, S^{n-1}(t,x)).
\end{align*}
In particular, since $r=d(t,X_0(t,s)+r \bold{n}(t,s))$, by differentiating this in $r$,
we have
$$
\nabla d(t,X_0(t,s)+r \bold{n}(t,s)) \cdot \bold{n}(t,s) = 1,
\quad r \in (-3\delta,3\delta).
$$
Since $|\nabla d(t,x)|=1$ for $x$ close to $\ga_t$, this implies
\begin{equation}  \label{eq:nablad}
\nabla d(t,X_0(t,s)+r \bold{n}(t,s)) = \bold{n}(t,s), 
\quad r \in (-3\delta,3\delta).
\end{equation}
Changing coordinates from $(t,x)$ to $(t,r,s)$ for a function $\phi=\phi(t,x)$,
we associate another function $\tilde{\phi}=\tilde{\phi}(t,r,s)$ as
\begin{align*}
\tilde{\phi}(t,r,s)=\phi(t,X_0(t,s)+r\bold{n}(t,s)), 
\end{align*}
or equivalently 
\begin{align*}
\phi(t,x)=\tilde{\phi}(t,d(t,x),\bold{S}(t,x)).
\end{align*}
Let $V(t,s)$ be the inward normal velocity of the interface 
$\ga_t$ at $X_0(t,s)$, 
$$
V(t,s) = -\partial_t X_0(t,s)\cdot \bold{n}(t,s).
$$
From \eqref{eq:boldn} and $d(t,X_0(t,s))=0$, we see that
\begin{align}\label{eq:2.V}
V(t,s) = \partial_t d(t,X_0(t,s)).
\end{align}

Then, we have
\begin{align*}
&\partial_t\phi(t,x)=(V\partial_r+\partial_t^{\Ga})\tilde{\phi}(t,d(t,x),\bold{S}(t,x)), 
\\
&\nabla\phi(t,x)=(\bold{n}(t,\bold{S}(t,x))\partial_r+\nabla^{\Ga})\tilde{\phi}(t,d(t,x),\bold{S}(t,x)), 
\\
&\Delta\phi(t,x)=(\partial_{r}^2+\Delta d(t,x)\partial_r+\Delta^{\Ga})\tilde{\phi}(t,d(t,x),\bold{S}(t,x)),
\end{align*}
where the superscripts $\Ga$ mean the derivatives tangential to the hypersurface $\ga^v$
seen under the coordinate $s\in \mathcal{S}$:
\begin{align*}
&\partial_t^{\Ga}\tilde{\phi}=(\partial_t+\sum_{i=1}^{n-1}S_t^i \partial_{s^i})\tilde{\phi}, 
\\
&\nabla^{\Ga}\tilde{\phi}=\left(\sum_{i=1}^{n-1}\partial_1S^i\partial_{s^i},\ldots,\sum_{i=1}^{n-1}\partial_n S^i\partial_{s^i}\right)\tilde{\phi},
\\
&\Delta^{\Ga}\tilde{\phi}=\left(\sum_{i=1}^{n-1}\Delta S^i \partial_{s^i}+\sum_{i,j=1}^{n-1}\nabla S^i\cdot\nabla S^j \partial_{s^is^j}^2\right)\tilde{\phi},
\end{align*}
see (33) in \cite{CHL} recalling that the sign of $V$ is opposite.
We denote by $\kappa_1,\cdots,\kappa_{n-1},0$ the eigenvalues 
of the Hessian $D_x^2 d(t,x)$ with corresponding normalized eigenvectors 
$\tau_1,\cdots,\tau_{n-1},$ $\nabla d$.  Set   
\begin{align}\label{eq:curvature}
\kappa(t,s):=(n-1) \bar\kappa_{\gamma_t}
=\sum_{i=1}^{n-1}\kappa_i=\Delta d(t,X_0(t,s)),
\end{align}
where $\bar\kappa_{\gamma_t}$ is the mean curvature of $\gamma_t$ 
at $x=X_0(t,s)$.  We denote 
\begin{align}\label{eq:b(t,s)}
b(t,s):=-\nabla d\cdot \nabla \Delta d(t,x)|_{x=X_0(t,s)}=
\sum_{i=1}^{n-1}\kappa_i^2.
\end{align}

\subsection{Formal expansion of the solution $u^\e$}

In this subsection, we briefly recall in our setting the method of construction of inner and outer solutions given in \cite{CHL}. 
%
%
The equation \eqref{eq:1-v} is expressed as
\begin{align}\label{eq:our equation}
& 0=f(u^\e(t,x))+\e^2(-\partial_t u^\e(t,x) + \Delta u^\e(t,x) + v(t)) - \e \lambda_\e(t),
\end{align}
where $\la_\e(t):= \e^{-1} a^\e(t)$ is given in \eqref{eq:lambda}, that is,
\begin{align}\label{eq:lambda-e}
\lambda_\e(t):=\e^{-1}- \hspace{-3.9mm}\int_D f(u^\e(t,\cdot)).
\end{align}
Note that \eqref{eq:our equation} combined with the conservation law 
\begin{align}\label{eq:2.10}
\partial_t\int_D u^\e(t,x)dx=|D|v(t)
\end{align}
implies \eqref{eq:lambda-e} and therefore \eqref{eq:1-v}.
We define $h_\e(t,s)$ by 
\begin{align}  \label{eq:2.ga-tilde}
& \tilde{\ga}_t^\e\equiv\{x\in D\,|\,u^\e(t,x)=0\}=\{ X(t,r,s)  \,| \, r = \e h_\e(t,s), s\in \mathcal{S} \},
\end{align}
which is the 0-level set of the solution $u^\e$ of \eqref{eq:1-v}.
We will expand $u^\e$, $\la_\e$ and $h_\e$ in $\e$, see \eqref{eq:expansion-i} and \eqref{eq:expansion-o} below.

Let us define the stretched variable $\rho=\rho^\e(t,x)$ as
\begin{align*}
& \rho^\e(t,x)=\frac{ d(t,x)-\e h_\e(t,\bold{S}(t,x))}{\e},
\end{align*}
which is the distance between $x$ and $\tilde{\ga}^\e_t$ divided by $\e>0$.
Then, the variables $(t,x)$ and $(t,\rho,s)$ are related by
\begin{align}\label{eq:related variables}
&x=\hat{X}(t,\rho,s)=X(t,\e(\rho+h_\e(t,s)),s)=X_0(t,s)+\e(\rho+h_\e(t,s))\bold{n}(t,s).
\end{align}
Furthermore, its Jacobian $J^\e(t,\rho,s)$ defined by $dx=\e J^\e(t,\rho,s)d\rho ds$ is written as 
\begin{align}\label{eq:Jacobian e}
  J^\e(t,\rho,s) = J^0(t,s) \left(
1+\Delta d(t,s)(\e(\rho+h_\e(t,s)))+\sum_{i=2}^{n-1}(\e(\rho+h_\e(t,s)))^ij_i(t,s)\right),
\end{align}
with some given functions $j_i$ depending on $\gamma_t$; see \cite{CHL},
 (44) p.~537 and also p.\ 538.
By the change of variable formula (see \cite{CHL}, (39)), we obtain
\begin{align} 
\partial_t u^\e(t,x) =& \bigl(V(t,\bold{S}(t,x)) \e^{-1}
  - \partial_t^{\Ga} h_{\e}(t,\bold{S}(t,x)) \bigr)\partial_{\rho} \tilde{u}^\e(t,\rho(t,x),\bold{S}(t,x))  \label{eq:change variable 01}
 \\
  &+ \partial_t^{\Ga} \tilde{u}^\e(t,\rho(t,x),\bold{S}(t,x)),
  \nonumber
\\
\Delta u^\e(t,x) = &\left(\e^{-2} + |\nabla^{\Ga}h_\e(t,\bold{S}(t,x))|^2\right)\partial_{\rho}^2 \tilde{u}^\e(t,\rho(t,x),\bold{S}(t,x))  \label{eq:change variable 02}\\
  &+\left(\Delta d(t,x)\e^{-1}-\Delta^{\Ga}h_\e(t,\bold{S}(t,x))\right)
   \partial_{\rho} \tilde{u}^\e(t,\rho(t,x),\bold{S}(t,x))
  \nonumber\\
  &-2\nabla^{\Ga}h_\e(t,\bold{S}(t,x))\cdot \nabla^{\Ga}\partial_\rho \tilde{u}^\e(t,\rho(t,x),\bold{S}(t,x))
  \nonumber\\
  &+\Delta^{\Ga} \tilde{u}^\e(t,\rho(t,x),\bold{S}(t,x)),
  \nonumber
\end{align}
where $\tilde{u}^\e=\tilde{u}^\e(t,\rho,s)$ is the function $u^\e=u^\e(t,x)$ 
viewed under the coordinate $(t,\rho,s)$ defined by \eqref{eq:related variables}.
In the following, we will denote $\tilde{u}^\e$ by $u$ for simplicity.
Therefore, from \eqref{eq:our equation}, we have 
\begin{align}\label{eq:(55)}
  0=&\bigl[\partial_{\rho}^2u+f(u)\bigr] + \e \bigl[ (-V(t,s)+\Delta d)\partial_\rho u-\lambda_\e(t) \bigr]
  \\
  &+\e^2\bigl[(\Delta^{\Ga}u-\partial_t^{\Ga}u)+(\partial_t^{\Ga}h_\e-\Delta^{\Ga}h_\e)\partial_\rho u\bigr]
  \nonumber
  \\
  &+\e^2\bigl[ |\nabla^{\Ga}h_\e|^2\partial_{\rho}^2u-2\nabla^{\Ga}h_\e \cdot \nabla^{\Ga}\partial_\rho u \bigl]+\e^2v(t).
  \nonumber  
\end{align}
Suppose that $u$ and $h_\e$ have the {\it inner} asymptotic expansions: 
\begin{equation}  \label{eq:expansion-i}
\begin{aligned}
  &u(t,\rho,s)=m(\rho)+\e u_0(t,\rho,s) + \e^2 u_1(t,\rho,s) + \e^3 u_2(t,\rho,s) + \cdots,
  \\
  &\e h_\e(t,s)=\e h_1(t,s) + \e^2 h_2(t,s) + \e^3 h_3(t,s) +  \cdots,
\end{aligned}
\end{equation}
for $(t,\rho,s) \in [0,T^v]\times \R\times \mathcal{S}$, respectively,
where $m$ is the standing wave solution determined by 
\eqref{eq:m standing wave}.  
On the other hand, assume that $\lambda_\e$ and $u^\pm$ have the
{\it outer} asymptotic expansions:
\begin{equation}  \label{eq:expansion-o}
\begin{aligned}
  &\lambda_\e(t)=\lambda_0(t) + \e \lambda_1(t) + \e^2 \lambda_2(t) + \e^3\lambda_3(t) + \cdots, 
  \\
  &u^\pm(t)=\pm1+\e u_0^\pm(t) + \e^2 u_1^\pm(t) + \e^3 u_2^\pm(t) + \cdots,
\end{aligned}
\end{equation}
for  $t \in[0,T^v]$, respectively.
Furthermore, let us note that $\Delta d(t,x)$ is expanded into 
\begin{align}\label{eq:Delta d}
&\Delta d(t,x)|_{x=X_0(t,s)+\e(\rho+h_\e(t,s))\bold{n}(t,s)} 
\\
= &\kappa(t,s)-\e(\rho+h_\e(t,s))b(t,s)+\sum_{i\geq 2}\e^ib_i(t,s)(\rho+h_\e(t,s))^i,
\nonumber
\end{align}
where $b_i(t,s)$, $i\geq 2$ are some functions depending only on $\gamma_t$; see (40) of \cite{CHL}.
\subsection{Inductive scheme to determine coefficients}
In this subsection, for a fixed $K>\max(n+2,6)$, we construct functions
$\{u_k\}_{k=0}^K$, $\{h_k\}_{k=0}^K$, $\{\la_k\}_{k=0}^{K}$ and $\{u_k^\pm\}_{k=0}^K$
appearing in the above expansions \eqref{eq:expansion-i} and \eqref{eq:expansion-o}
of $u$, $h_\e$, $\lambda_\e$ and $u^\pm$ defined for $t\in[0,T^v]$
in such a manner that all $k$-th order terms (that is, those of order $O(\e^k)$) vanish
when we substitute these formal expansions in \eqref{eq:(55)},
where we set $h_0=0$ for convenience.  In fact, setting
\begin{align}
  \nu_k=(u_k,h_k, \la_k,u^{\pm}_k),\quad k=0,1,\ldots,K,
\end{align}
$\nu_k$ will be inductively determined as follows:  For $k=0$,
$\lambda_0$ will be defined by \eqref{eq:lambda 0} below, $u_0$ by \eqref{eq:inner u_0}
and $u_0^\pm$ by \eqref{eq:u0 outer}, respectively.  This determines $\nu_0$.
For $k\ge 1$,  once we know $\{\nu_i\}_{i=0}^{k-1}$,
$h_k$ is determined by solving the equation \eqref{eq:equation for h_k}
and $\lambda_k$ by \eqref{eq:lambda_k representation} knowing $h_k$ additionally,
respectively.  Furthermore, $u_k$ is defined by \eqref{eq:inner u_k (odd)}  knowing
$h_k$ and $\la_k$, while $u_k^\pm$ are determined by \eqref{eq:uk outer}.

When we insert the expansion of $u$ in \eqref{eq:expansion-i} into \eqref{eq:(55)},
the term of order $O(1)$ needs to satisfy $m^{''}(\rho)+f(m(\rho))=0$ so that
we took the leading term of $u$ as $m(\rho)$.

Let us start the procedure to determine $\nu_0$.
For the term of order $O(\e)$ in \eqref{eq:(55)} to vanish, 
recalling \eqref{eq:Delta d},
we have
\begin{align}\label{eq:Order e1}
 \mathcal{L}u_0=(-V+\kappa)m'-\lambda_0(t), 
\end{align}
where 
$\mathcal{L}$ is the linearized operator of $-(\partial_\rho^2u+f(u))$ 
around $m$ defined as
\begin{align*}
\mathcal{L}u_0=- \partial_\rho^2 u_0-f'(m)u_0,\quad \rho\in\R.
\end{align*} 
Suggested by \eqref{eq:lambda-0}, we first define $\lambda_0(t)$ as 
\begin{align}\label{eq:lambda 0}
  \lambda_0(t)=\frac{1}{\bar\sigma|\ga_t^v|}\int_{\ga_t^v}
 \kappa(t,\bar s) d\bar s - \frac{|D|}{2\bar\sigma|\ga_t^v|}v(t).
\end{align}
Note that $\int_{\ga_t^v} \kappa(t,\bar s) d\bar s
=\int_{\mathcal{S}} \k(t,s)J^0(t,s)ds$ and 
$|\ga_t^v|=\int_{\mathcal{S}}J^0(t,s)ds$ by the change
of variables $\bar s=X_0(t,s)$.
Then, from the solvability condition $\int \mathcal{L}u_0 \, m'd\rho=0$
for \eqref{eq:Order e1}, $V$ needs to satisfy
\begin{align}\label{eq:V + K}
 -V(t,s)+\kappa(t,s)=\lambda_0(t){\bar\sigma},\quad (t,s)\in [0,T^v]\times \mathcal{S}, 
\end{align}
where $\bar\sigma=2(\int_\R m'(\rho)^2 d\rho)^{-1}$.
This combined with \eqref{eq:lambda 0}
leads to the evolutional law \eqref{eq:2-e} of $\ga_t^v$.

Next, to determine $u_0$, we note the following fact:
Since $m_1(\rho)=m'(\rho)$ and $m_2(\rho)=m'(\rho)\int_0^\rho \frac{1}{m'(y)^2}dy$ 
are linearly independent two solutions of $\mathcal{L}u(\rho)=0$
and $m_1^\prime m_2-m_1m_2^\prime=-1$,
the solution of the equation $\mathcal{L}u(\rho)=g(\rho)$ satisfying $u(0)=0$
is unique and given by
\begin{align}\label{eq:fundamental solution m}
u(\rho)=&\Bigl(-\int_0^\rho m'(y)g(y)dy\Bigr)m'(\rho)\int_0^\rho\frac{dy}{m'(y)^2} 
\\
&+ \Bigl(\int_0^\rho m'(y) (\int_0^y \frac{dz}{m'(z)^2}) g(y)dy \Bigr)m'(\rho).
\nonumber
\end{align}
Let us determine $u_0(t,\rho,s)$ and $u^{\pm}_0(t)$.
The function $u_0(t,\rho,s)$ satisfies \eqref{eq:Order e1}.
To solve this equation,
let $\theta_1$ be a unique solution of $\mathcal{L}\theta_1=1-m'\bar\sigma$, namely,
from \eqref{eq:fundamental solution m} with $g=1-m^\prime\bar\sigma$,
\begin{align*}
\theta_1(\rho)=&\Bigl(-\int_0^\rho m'(y)(1-m'(y)\bar\sigma)dy\Bigr)m'(\rho)\int_0^\rho\frac{1}{m'(y)^2}dy 
\nonumber
\\&+ \Bigl(\int_0^\rho m'(y)(\int_0^y \frac{1}{m'(z)^2}dz) (1-m'(y)\bar\sigma)dy \Bigr)m'(\rho).
\nonumber
\end{align*}
Then, noting \eqref{eq:V + K}, $u_0(t,\rho,s)$ is given by 
\begin{align}\label{eq:inner u_0}
u_0(t,\rho,s)=-\lambda_0(t)\theta_1(\rho).
\end{align}
Furthermore, since $\lim_{\rho\rightarrow\pm\infty}|g(t,\rho,s)+\lambda_0(t)|=O(e^{-\zeta|\rho|})$, $\zeta>0$
holds for the right hand side $g(t,\rho,s)$ of \eqref{eq:Order e1}, 
by Lemma 3 in \cite{CHL}, it follows that 
\begin{align}\label{eq:compat 0th}
&  \lim_{\rho\rightarrow\pm\infty}
   \partial_{\rho}^m \partial_{s}^\bold{n} \partial_{t}^l \Bigl(u_0(t,\rho,s)-\frac{\lambda_0(t)}{f'(\pm1)}\Bigr)=O(e^{-\zeta|\rho|}),
\end{align}
for all $(m,\bold{n},l)\in \Z_+\times\Z_+^{n-1}\times\Z_+$.  Therefore, we define 
\begin{align}\label{eq:u0 outer}
u_0^{\pm}(t):=\frac{\lambda_0(t)}{f'(\pm1)}.
\end{align}

Now, let us determine $\nu_k$ for $k\ge 1$ assuming that
$\{\nu_i\}_{i=0}^{k-1}$ are known.
For the term of order $O(\e^{k+1})$ in \eqref{eq:(55)} to vanish,
it is necessary to have
\begin{align}\label{eq:A_1}
 \mathcal{L}u_k=A^{k-1}
& +\bigl(\partial_t^{\Ga}-\Delta^{\Ga}-b(t,s)\bigr)h_k \, m'  
  - b(t,s) \rho \, m' 1_{\{k=1\}}  \\
& + (2 \cdot 1_{\{k\ge 2\}}+1_{\{k=1\}}) \nabla^\Ga h_1 \cdot \nabla^\Ga h_k \, m''
-\lambda_k(t)+v(t)1_{\{k=1\}},  \notag
\end{align}
where $b(t,s)$ is defined by \eqref{eq:b(t,s)},
$A^{k-1}=A^{k-1}(\la_0,u_i,h_i,0\le i \le k-1)$ is given by
\begin{align}\label{eq:equation A^{k-1}}
 A^{k-1}
 =&\lambda_0(t){\bar\sigma} \, \partial_\rho u_{k-1}
 \\
 &+\sum_{l=2}^k  b_l(t,s) \sum_{\substack{ i_1,\cdots,i_l\geq 1,j\geq-1, \\ i_1+\cdots+i_l+j=k-1} }
  \tilde{h}_{i_1} \cdots \tilde{h}_{i_l} \, \partial_\rho u_j 
 -\partial_t^{\Ga}u_{k-2}+\Delta^{\Ga}u_{k-2}
 \nonumber\\
 &
 +\sum_{i=1}^{k-1}\bigl(\partial_t^{\Ga}h_i-\Delta^{\Ga}h_i-b(t,s) \tilde h_i\bigr)
\partial_\rho u_{k-1-i}
 \nonumber\\
 &+ \sum_{\substack{i,j\geq1, i+j\leq k+1, \\
(i,j) \not= (k,1), (1,k)}}\nabla^{\Ga}h_i \cdot \nabla^{\Ga}h_j \partial_{\rho}^2u_{k-(i+j)}
 -2\sum_{i=1}^{k-1}\nabla^{\Ga}h_i \cdot \nabla^{\Ga}\partial_{\rho} u_{k-1-i}
 \nonumber\\
 &+\sum_{l=2}^{k+1} \frac{1}{l!}f^{(l)}(m) \sum_{\substack{ i_1, \ldots,i_l \ge 0, \\ i_1+\cdots+i_l=k+1-l}} u_{i_1}\cdots u_{i_l},  
 \nonumber
\end{align}
$\tilde{h}_i(t,s):=h_i(t,s)+\rho\,1_{\{i=1\}}$, $1\leq i \leq K$, $u_{-1}=m$ and $b_l(t,s)$ are defined by \eqref{eq:Delta d}.  In particular, 
$A^0(\la_0,u_0)=\lambda_0(t)\bar\si \, \partial_\rho u_0 +\frac12f^{\prime\prime}(m)u_0^2$,
since $\partial_t^\Ga m = \De^\Ga m=0$.  We have used \eqref{eq:Delta d} and \eqref{eq:V + K} 
to have the expansion of the term $-V+\De d$.

From the solvability condition $\int\mathcal{L}u_k \, m' d\rho=0$ for \eqref{eq:A_1}
and noting that $\int_\R m'' m' \, d\rho=0$,
it follows that $\lambda_k$ and $h_k$ should satisfy
\begin{align}\label{eq:lamdba(k) odd}
  \lambda_k(t)=& \frac{1}{\bar\sigma}\left( \partial_t^{\Ga} - \Delta^{\Ga} -b(t,s) \right)h_k
  +\frac{1}{2}\int_\R A^{k-1} m'(\rho)d\rho \\
& + \left\{ v(t) - \frac12 b(t,s)
 \int_\R \rho \left(m'(\rho)\right)^2 \, d\rho\right\} 1_{\{k=1\}}.  \notag
\end{align}

The next task is to determine $u_k(t,\rho,s)$ and $u^{\pm}_k(t)$.
Recall that $k\ge1$ is fixed.
From \eqref{eq:fundamental solution m},
it follows that the linear equation $\mathcal{L}u_k=\widetilde{A}^k$
satisfying $u(t,0,s)=0$ has a unique solution given by
\begin{align}\label{eq:inner u_k (odd)}
u_k(t,\rho,s)=&\Bigl(-\int_0^\rho m'(y)\widetilde{A}^k(t,y,s)dy\Bigr)
m'(\rho)\int_0^\rho\frac{dy}{m'(y)^2} 
\\
&+ \Bigl(\int_0^\rho m'(y) (\int_0^y \frac{dz}{m'(z)^2}) \widetilde{A}^k(t,y,s)dy \Bigr)m'(\rho),
\nonumber
\end{align}
where $\widetilde{A}^k(t,\rho,s)$ denotes  the right hand side of \eqref{eq:A_1}.
Especially, supposing that $u_0^\pm$, $u_1^\pm,\ldots,u_{k-1}^\pm$ are determined inductively
and $\{u_i\}_{i=0}^{k-1}$ satisfy the following \eqref{eq:compat kth} with $i$ instead of $k$,
from Lemma 3 in \cite{CHL} and noting that $h_k$ are independent of $\rho$ 
and $m^\prime$, $\rho m^\prime$, $m^{\prime\prime}$, $\partial_\rho u_i$, $\partial_\rho^2 u_i$, $\Delta^\Ga u_{k-2}$ tend to $0$ as $|\rho|\rightarrow\infty$,
it is easy to check
\begin{align}\label{eq:compat kth}
&  \partial_{\rho}^m \partial_{s}^\bold{n} \partial_{t}^l
   \Bigl(u_k(t,\rho,s)-\frac{g_k^\pm(t)}{f'(\pm1)}\Bigr)=O(e^{-\zeta|\rho|}),
   \quad \text{as } \rho\rightarrow\pm\infty,
\end{align}
for $(t,s)$,
where
\begin{align}\label{eq:g_k(t)}
  g_k^\pm(t)
= - & \left(A^{k-1}|_{\rho=\pm\infty} -\lambda_k(t) +v(t) 1_{\{k=1\}} \right) \\
=-&\sum_{j=2}^{k+1} \frac{1}{j!}
       \sum_{\substack{ 0 \leq i_1 \leq i_2 \leq \cdots \leq i_j,\\ i_1+i_2+\cdots+i_j=k-(j-1)}}f^{(j)}(\pm 1)\frac{u_{i_1}^{\pm}(t)\cdots u_{i_j}^{\pm}(t)}{f'(\pm 1)^j}   \notag \\
       &+\partial_t^\Ga u_{k-2}^{\pm}(t)+\lambda_k(t) -v(t) 1_{\{k=1\}}.
\nonumber
\end{align}
This suggests to define 
\begin{align}\label{eq:uk outer}
u_k^{\pm}(t):=\frac{g_k^\pm(t)}{f'(\pm1)}.
\end{align}

We will determine $\la_k(t)$ in such a way that the term of order $O(\e^k)$ of the integral
\begin{align}\label{eq:m 01}
\int_D \Bigl(\partial_t u^\e_k(t,x) - v(t)\Bigr) dx
\end{align}
vanishes, where $u^\e_k$ is an approximate solution defined as \eqref{eq:definition u_k} later.
In fact, this results in
\begin{align}\label{eq:lambda_k representation}
\lambda_k(t)&=
\frac{1}{\bar\sigma|\ga_t^v|}\int_\mathcal{S}
\Bigl((\Delta d(t,X_0(t,s)) - \bar\sigma \lambda_0(t)) \Delta d(t,X_0(t,s)) h_k(t,s) - b(t,s)h_k(t,s) 
\\
&+ \frac{\bar\sigma}{2}\int_\R A^{k-1} m'(\rho)d\rho \Bigr)J^0(t,s)ds
\nonumber
\\
&+\Bigl(v(t)-\frac{1}{2|\ga_t^v|}\int_\R \rho (m^\prime(\rho))^2d\rho\int_\mathcal{S}b(t,s)J^0(t,s)ds\Bigr)1_{\{k=1\}}
\nonumber
\\
&+\frac{1}{2\bar\sigma|\ga_t^v|}B_{k-1},
\nonumber
\end{align}
where $B_{k-1}=B_{k-1}(u_i,u_i^\pm,h_i,\,0\leq i\leq k-1)$ is a term which is determined from $u_i$, $u_i^\pm$ and $h_i$, $0\leq i \leq k-1$;
see Section 3.2, in particular, \eqref{eq:definition of B_{k-1}} for $B_{k-1}$
and \eqref{eq:mass cons 01} below.
From \eqref{eq:lamdba(k) odd} and \eqref{eq:lambda_k representation},
we define $h_k$ as the solution of the following equation:
\begin{align}\label{eq:equation for h_k}
&\left( \partial_t^{\Ga}-\Delta^{\Ga}-b(t,s)\right)h_k(t,s)
\\
 & -\frac{1}{|\ga_t^v|}\int_\mathcal{S} \Bigl(\bigl(\Delta d(t,X_0(t,s')) - \bar\si\lambda_0(t)\bigr)\Delta d(t,X_0(t,s'))-b(t,s')\Bigr)h_k(t,s')J^0(t,s')ds'
\nonumber
\\
 =&-\frac{\bar\sigma}{2}\int_{\R} A^{k-1}\,m'(\rho)d\rho
   +\frac{\bar\sigma}{2|\ga_t^v|}\int\int_{\mathcal{S}\times \R}A^{k-1}\,m'(\rho)
J^0(t,s')ds'd\rho
\nonumber
\\
 &+\Bigl(\frac{\bar\sigma}{2}\bigl(b(t,s)-\frac{1}{|\ga_t^v|}\int_\mathcal{S} b(t,s')J^0(t,s')ds'\bigr)\int_\R \rho (m^\prime(\rho))^2d\rho\Bigr)1_{\{k=1\}}
+\frac{1}{2|\ga_t^v|}B_{k-1}.
\nonumber
\end{align}
This is a linear equation for $h_k$, $k\geq1$;
note that the right hand side is determined from $\lambda_0$, $u_i$, $u_i^\pm$ and $h_i$ with $0\leq i \leq k-1$.
When $k=1$, nonlinear term $|\nabla^\Ga h_1|^2$ appears in \eqref{eq:A_1},
however, this disappears by the relation $\int_\R m^\prime(\rho) m^{\prime \prime }(\rho)d\rho=0$.
We will study the equation \eqref{eq:equation for h_k} in Section \ref{sec:3.3.3}.
\section{Approximate solutions and their estimates}
\subsection{Approximate solutions}

Once all $\{\nu_i\}_{i=0}^{K}$ are determined, for $k=0,1,\cdots,K,$ we can define
{\it approximations} 
$u_{k,\e}^{\text{in}}$, $u_{k,\e,\pm}^{\text{out}}$, $h_k^\e$, $\la_k^\e$ and $\rho_{k,\e}$
of $u$ both in inner and outer's senses, $h_\e$, $\la_\e$ and $\rho_\e$, respectively,
by cutting the expansions \eqref{eq:expansion-i} and \eqref{eq:expansion-o}
 after the $k+1$th or $k$th terms:
\begin{equation} \label{eq:approximate-sol}
\begin{aligned}
  &u_{k,\e}^{\text{in}}(t,x):=m(\rho)+\sum_{i=0}^{k}\e^{i+1} u_i(t,\rho,\bold{S}(t,x)),
  \\
  &u_{k,\e,\pm}^{\text{out}}(t):=\pm1+\sum_{i=0}^{k}\e^{i+1} u_i^{\pm}(t),
  \\
  &\e h_k^\e(t,s)=\sum_{i=1}^k \e^{i}h_i(t,s),
  \\
  &\la_k^\e(t):=\sum_{i=0}^{k}\e^i\la_i(t),
  \\
  &\rho=\rho_{k,\e}(t,x):=\frac{1}{\e}\bigl(d(t,x)-\e h_{k}^\e(t,\bold{S}(t,x))\bigr).
\end{aligned}
\end{equation}

Let $\eta \in C^{\infty}(\mathbb{R})$ be a function satisfying the conditions:
$\eta(s)=1 \text{ for } |s|\leq \delta,
  \,\,\eta(s)=0 \text{ for } |s| \geq 2\delta\, \text{  and } 
\,0\leq s\eta^{'}(s)\leq 4 \text{ for } \delta \leq |s| \leq 2\delta$.
Then, define the approximate solutions $u_k^{\e}(t,x)$ as follows by connecting the inner and outer approximate solutions:
\begin{align}\label{eq:definition u_k}
u_k^\e(t,x):=\eta(d(t,x))u_{k,\e}^{in}(t,x)+\bigl(1-\eta(d(t,x))\bigr)
\bigl(u_{k,\e,+}^{\text{out}}(t)1_{\{d>0\}}+u_{k,\e,-}^{\text{out}}(t)1_{\{d<0\}}\bigr).
\end{align}
%
\subsection{Derivation of \eqref{eq:lambda_k representation}}

Now we consider the integral \eqref{eq:m 01},
which is the left hand side of \eqref{eq:mass cons 3} times $|D|$
with $u^\e$ and $\a\dot{w}^\e$ replaced by $u_k^\e$ and $v$, respectively, 
and expand it in $\e$ to obtain \eqref{eq:lambda_k representation}.
Let us decompose the time derivative of the mass of $u^\e_k$ over $D$ into
\begin{align*}
\int_D \partial_t u^\e_k(t,x)dx &= \int_{ |d(t,x)|\geq 3\delta} \partial_t u^\e_k(t,x)dx 
\nonumber
\\
&+\int_{|\rho|\geq\frac{\delta}{\e}} \partial_t u^\e_k(t,x)dx+\int_{|\rho|<\frac{\delta}{\e}} \partial_t u^\e_k(t,x)dx
\nonumber
\\
&=I+II+III,
\end{align*}
where $\delta>0$ is chosen in Section \ref{sec:2.2}, 
$\{|\rho|\geq\frac{\delta}{\e}\}=\{ x \in V^{t}_{3\delta}\,:\,|\rho(t,x)|\geq\frac{\delta}{\e}\}$,
$\{|\rho|<\frac{\delta}{\e}\}=\{ x \in V^{t}_{3\delta}\,:\,|\rho(t,x)|<\frac{\delta}{\e}\}$ 
and $V^{t}_{3\delta}:=\{ x \in D\,:\, |d(t,x)|<3\delta\}$.
From now on, we choose a sufficiently small $\e_0^*=\e_0^*(\delta,K,v)>0$ such that
\begin{align}\label{eq: e times h-e}
\sup_{s \in \mathcal{S}, t\in[0,T]}|\e h_k^\e(t,s)|\leq\frac{\delta}{2}
\end{align}
holds for all $0<\e\leq\e_0^*$. This is possible for each $v$.
Thus $|d(t,x)|\geq \frac{\delta}{2}$ follows from $|\rho(t,x)|\geq \frac{\delta}{\e}$
and for any point $(t,x)$ where either $|d(t,x)|\geq 3\delta$ or $|\rho(t,x)|\geq \frac{\delta}{\e}$,
$|d(t,x)|\geq \frac{\delta}{2}$ holds. 
Noting that each inner solution $u_i$ converges to its associated outer solution $u^{\pm}_i$ with an error $O(e^{-\zeta|\rho|})$ as $|\rho|\rightarrow\infty$,
and setting 
\begin{align*}
&D^{+}_{\e}(t)=\bigl\{ x\in D \,: \, d(t,x)\geq3\delta\bigr\}\cup\bigl\{ x\in D\,: \, |d(t,x)|<3\delta,\,\rho(t,x)>0\bigr\},
\nonumber
\\
&D^{-}_{\e}(t)=D\backslash \bar{D}^{+}_{\e}(t),
\nonumber
\end{align*}
we have
\begin{align*}
  I+II&=(u_{k,\e,+}^{\text{out}})'(t)|D^{+}_{\e}(t)|+(u_{k,\e,-}^{\text{out}})'(t)|D^{-}_{\e}(t)|
  \\
  &-(u_{k,\e,+}^{\text{out}})'(t)\int_{|\rho|<\frac{\delta}{\e}}1_{\{\rho>0\}}dx-(u_{k,\e,-}^{\text{out}})'(t)\int_{|\rho|<\frac{\delta}{\e}}1_{\{\rho<0\}}dx+O(e^{-\zeta\delta/\e}),
\end{align*}
as $\e\downarrow 0$, 
where $'$ means the derivative in $t$ and we have used \eqref{eq:compat kth} 
which gives the error term $O(e^{-\zeta\delta/\e})$.
Concerning the term $III$, change of variables and (81) of \cite{CHL} lead us to
\begin{align*}
 III=&\int_\mathcal{S}\int_{|\rho|<\frac{\delta}{\e}} \bigl((V(t,s)\e^{-1}
 -\partial_t^{\Ga}h_k^\e)\partial_\rho u^\e_k\bigr) \e J^\e(t,\rho,s)d\rho ds
 \\
 &+\int_\mathcal{S}\int_{|\rho|<\frac{\delta}{\e}} \partial_t^{\Ga} u^\e_k \e J^\e(t,\rho,s)d\rho ds.
\end{align*}
As a result, we obtain
\begin{align}\label{eq:int partial_t u dx}
  \int_D \partial_t u^\e_k(t,x)dx 
 =&(A)+(B)+(C)+O(e^{-\zeta\delta/\e}),
\end{align}
where
\begin{align*}
 (A)=&(u_{k,\e,+}^{\text{out}})'(t)|D^{+}_{\e}(t)|+(u_{k,\e,-}^{\text{out}})'(t)|D^{-}_{\e}(t)|,
 \\
 (B)=&-\int_{|\rho|<\frac{\delta}{\e}}\Bigl((u_{k,\e,+}^{\text{out}})'(t)1_{\{\rho>0\}}+(u_{k,\e,-}^{\text{out}})'(t)1_{\{\rho<0\}}\Bigr)dx
 \\
 &+\int_\mathcal{S}\int_{|\rho|<\frac{\delta}{\e}} \partial_t^{\Ga} u^\e_k(t,\rho,s) \e J^\e(t,\rho,s)d\rho ds,
 \\
 (C)=&\int_\mathcal{S}\int_{|\rho|<\frac{\delta}{\e}} (V(t,s)\e^{-1}
 -\partial_t^{\Ga}h_k^\e(t,s))\partial_\rho u^\e_k(t,\rho,s) \e J^\e(t,\rho,s)d\rho ds.
\end{align*}
%
%

In the following, we rewrite these three terms $(A)$, $(B)$ and $(C)$.
First, $(A)$ can be rewritten as follows:
Let us denote by $J(t,r,s)$ the Jacobian of the map \eqref{eq:parameterize}. 
By the equation (44) and p.\ 538 of \cite{CHL},
we have 
\begin{align}\label{eq:Jacobi x}
  J(t,r,s)= J^0(t,s) \left(1+\Delta d(t,X_0(t,s))r+\sum_{i=2}^{n-1}r^ij_i(t,s)\right),
\end{align}
for some functions $j_i(t,s)$; see also \eqref{eq:Jacobian e}.
By \eqref{eq:Jacobi x}, we get
\begin{align*}
|D^{-}_{\e}(t)|=&|D_t|+\int_\mathcal{S} \int_0^{\e h_k^\e(t,s)}J(t,r,s)drds
=|D_t|+\int_\mathcal{S} \e h_k^\e(t,s) J^0(t,s)ds + (A_1),
\nonumber
\\
|D^{+}_{\e}(t)|=&|D|-|D^{-}_{\e}(t)|
=|D|-|D_t|-\int_\mathcal{S} \e h_k^\e(t,s) J^0(t,s)ds - (A_1),
\end{align*}
where $D_t$ denotes the inside surrounded by the hypersurface $\ga_t^v$ and 
\begin{align*}
&(A_1)=\frac12 \int_\mathcal{S} \Delta d(t,X_0(t,s))(\e h_k^\e(t,s))^2 J^0(t,s)ds + \sum_{i=2}^{n-1}\frac{1}{i+1}\int_\mathcal{S} j_i(t,s)(\e  h_k^\e(t,s))^{i+1}J^0(t,s)ds.
\end{align*}
Recalling the expansions of $u_{k,\e,\pm}^{\text{out}}$ and $\e h^\e_k$ given in \eqref{eq:approximate-sol},
we can decompose $(A)$ into the sum of the following two parts, namely,
\begin{align}\label{eq:Comp A}
(A) 
=&\sum_{i=1}^k \e^i B_{i-1}^A + \sum_{i=k+1}^\infty \e^i \Bar{B}_{k,i}^A,
\end{align}
where $B_{i-1}^A$ is the term of order $O(\e^i)$ determined from $d$, $u_j^\pm$ and $h_j$, $0\leq j\leq i-1$
and $\Bar{B}_{k,i}^A$ denotes the term of order $O(\e^i)$ determined from $d$, $u_j^\pm$, $h_j$, $0\leq j\leq k$ for $i\geq k+1$.

For $(B)$, using \eqref{eq:Jacobian e} and 
the expansions of $u_k^\e=u_{k,\e}^{in}$ for $|\rho|<\frac{\delta}{\e}$ and $u_{k,\e,\pm}^{\text{out}}$ given in \eqref{eq:approximate-sol},
we obtain
\begin{align}\label{eq:Comp B-00}
  (B)=&
  \int_\mathcal{S}\int_{|\rho|<\frac{\delta}{\e}} \Bigl(\partial_t^{\Ga} u^\e_k 
  -(u_{k,\e,+}^{\text{out}})'(t)1_{\{\rho>0\}}-(u_{k,\e,-}^{\text{out}})'(t)1_{\{\rho<0\}}\Bigr)\e J^\e(t,\rho,s)d\rho ds
\\
=&
  \sum_{i=0}^k\int_\mathcal{S}\int_{|\rho|<\frac{\delta}{\e}} \e^{i+2}\Bigl(\partial_t^{\Ga} u_i(t,\rho,s) 
  -(u_i^+)^\prime(t)1_{\{\rho>0\}}-(u_i^-)^\prime(t)1_{\{\rho<0\}}\Bigr) 
\nonumber\\
 & \qquad \qquad
  \times\Bigl( 1 + \e(\rho+\e h_{k}^\e(t,s))\Delta d + \sum_{l=2}^{n-1}\e^l(\rho+\e h_{k}^\e(t,s))^l j_l(t,s)\Bigr)J^0(t,s)d\rho ds 
\nonumber\\
=& (B_1)-(B_2),  \notag
\end{align}
where $(B_1)$ and $(B_2)$ are defined as the middle of \eqref{eq:Comp B-00}
with the integral region
$\{|\rho|<\frac{\delta}{\e}\}$ replaced by $\R$ and $\{|\rho|\ge\frac{\delta}{\e}\}$,
respectively.  Note that both $(B_1)$ and $(B_2)$ are finite, since we have
from compatibility condition \eqref{eq:compat kth} and \eqref{eq:uk outer}:
\begin{align}\label{eq:Comp B-1}
  \Bigl(\partial_t^{\Ga} u_i(t,\rho,s)
  -(u_i^+)^\prime(t)1_{\{\rho>0\}}-(u_i^-)^\prime(t)1_{\{\rho<0\}}\Bigr)
  =O(e^{-{\zeta|\rho|}}),
\end{align}
as $\rho\rightarrow\pm\infty$ for $(t,s)$.
Then, similarly to $(A)$, we can decompose $(B_1)$ into the sum of two parts 
and we obtain:
\begin{align}\label{eq:Comp B}
(B) 
=&\sum_{i=2}^k\e^i B_{i-2}^B(t) + \sum_{i=k+1}^\infty\e^i \Bar{B}_{k,i}^B(t) - (B_2),
\end{align}
where $B_{i-2}^B$ is the term of order $O(\e^i)$ determined from $d$, $u_j$, $u_j^\pm$ 
and $h_j$, $0\leq j\leq i-2$ and $\Bar{B}_{k,i}^B$ denotes the term of order $O(\e^i)$ determined
from $d$, $u_j$, $u_j^\pm$, $h_j$, $0\leq j\leq k$ for $i\geq k+1$.

For $(C)$,
note that $V(t,s)=\Delta d(t,X_0(t,s)) - \bar\sigma \lambda_0(t)$ from \eqref{eq:curvature} and \eqref{eq:V + K}, and
$\partial_\rho u_k^\e = \partial_\rho u_{k,\e}^{in} =\partial_\rho m + \sum_{i=0}^k \e^{i+1}\partial_\rho u_i$ from \eqref{eq:approximate-sol}. 
Recalling \eqref{eq:Jacobian e} again, we have
\begin{align*}
(C)
=&\int_\mathcal{S}\int_{|\rho|<\frac{\delta}{\e}}
\Bigl\{ (\Delta d(t,X_0(t,s)) - \bar\sigma \lambda_0(t)) -\partial_t^\Ga \sum_{i=1}^k \e^{i}h_i(t,s) \Bigr\}
\Bigl\{ m'(\rho) + \sum_{i=0}^k \e^{i+1}\partial_\rho u_i \Bigr\} \\
& \qquad  \times\Bigl\{ 1+\Delta d(t,s)(\e(\rho+h_k^\e(t,s)))
+\sum_{i=2}^{n-1}(\e(\rho+h_k^\e(t,s)))^ij_i(t,s) \Bigr\}J^0(t,s)d\rho ds \\
=& (C_1)-(C_2), 
\end{align*}
where $(C_1)$ and $(C_2)$ are defined by the above line with the integral region
$\{|\rho|<\frac{\delta}{\e}\}$ replaced by $\R$ and $\{|\rho|\ge\frac{\delta}{\e}\}$,
respectively.  
Then, in the term of order $O(\e^i)$ in the expansion of $(C_1)$ in $\e$, separating terms
containing $h_i$ from those containing lower order functions $\{\nu_j\}_{j\leq i-1}$,
and noting $\int_\R \partial_\rho m d\rho = 2$,
we can decompose $(C_1)$ into the sum of the following four parts:
\begin{align}\label{eq:3-a}
(C_1) =&2\int_\mathcal{S}(\Delta d(t,X_0(t,s)) - \bar\sigma \lambda_0(t))J^0(t,s)ds
\\
+
&2 \sum_{i=1}^k\e^i \int_\mathcal{S} \Bigl((\Delta d(t,X_0(t,s)) - \bar\sigma \lambda_0(t)) \Delta d(t,X_0(t,s)) h_i(t,s) - \partial_t^{\Gamma}h_i(t,s)\Bigr)J^0(t,s)ds
\nonumber
\\
&+\sum_{i=1}^k\e^i B_{i-1}^C(t) + \sum_{i=k+1}^\infty\e^i \Bar{B}_{k,i}^C(t), \notag
\end{align}
where $B_{i-1}^C$ is the term of order $O(\e^i)$ determined from $d$, 
$\la_0$, $u_j$, $u_j^\pm$ and $h_j$, $0\leq j\leq i-1$
and $\Bar{B}_{k,i}^C$ denotes the term of order $O(\e^i)$ determined from $d$, 
$\la_0$, $u_j$, $u_j^\pm$, $h_j$, $0\leq j\leq k$ for $i\geq k+1$.

Setting 
\begin{align} \label{eq:definition of B_{k-1}}
&B_{k-1}=B_{k-1}^A+B_{k-2}^B+B_{k-1}^C,
\end{align}
and
\begin{align} \label{eq:definition of B_{k-1}bar}
&\bar{B}_{k,i}=\bar{B}_{k,i}^A+\bar{B}_{k,i}^B+\bar{B}_{k,i}^C,  
\end{align}
we obtain
\begin{align}\label{eq:eq of mass con 1}
&\int_D \Bigl( \partial_t u^\e_k(t,x) - v(t)\Bigr)dx
\\
=&2\int_{\mathcal{S}} \Bigl(\Delta d(t,X_0(t,s))-\bar\sigma\lambda_0(t)-\frac{|D|}{2|\ga_t^v|}v(t)\Bigr) J^0(t,s)ds
\nonumber
\\
&+2\sum_{i=1}^k\e^i \int_\mathcal{S} \Bigl((\Delta d(t,X_0(t,s)) - \bar\sigma \lambda_0(t)) 
\Delta d(t,X_0(t,s)) h_i(t,s) - \partial_t^{\Gamma}h_i(t,s)\Bigr)J^0(t,s)ds
\nonumber
\\
&+\sum_{i=1}^k\e^i B_{i-1} + 
\sum_{i=k+1}^\infty\e^i {\Bar{B}_{k,i}}-(B_2)-(C_2) + O(e^{-\zeta\de/\e}),
\nonumber
\end{align}
from \eqref{eq:int partial_t u dx}, \eqref{eq:Comp A}, \eqref{eq:Comp B},
\eqref{eq:3-a} and noting $\int_\mathcal{S} J^0(t,s)ds = |\ga_t^v|$.
Recalling that $\lambda_0(t)$ is defined as in \eqref{eq:lambda 0}, 
the first term of the right hand side of \eqref{eq:eq of mass con 1} vanishes so that
we can rewrite as
\begin{align}\label{eq:mass cons 01}
&\int_D \Bigl( \partial_t u^\e_k(t,x) - v(t)\Bigr)dx
\\
=&\sum_{i=1}^k\e^i\biggl[ 2 \int_\mathcal{S} \biggl( (\Delta d(t,X_0(t,s)) - \bar\sigma \lambda_0(t)) \Delta d(t,X_0(t,s)) h_i(t,s)
 - \bar\sigma\lambda_i(t)  - b(t,s)h_i(t,s)
\nonumber
\\
&+ \frac{\bar\sigma}2 \int_\R A^{i-1} m'(\rho)d\rho + \bar\sigma \Bigl(v(t)-\frac{b(t,s)}{2}\int_\R \rho (m^\prime(\rho))^2d\rho\Bigr)1_{\{i=1\}}\biggr)J^0(t,s)ds
+B_{i-1}\biggl]
\nonumber
\\
&+\sum_{i=k+1}^\infty\e^i {\Bar{B}_{k,i}} -(B_2)-(C_2) + O(e^{-\zeta\de/\e}),
\nonumber
\end{align}
by using \eqref{eq:lamdba(k) odd} for $\partial_t^{\Ga}h_i$ and noting 
$\int_\mathcal{S} \Delta^{\Ga}h_i(t,s)J^0(t,s)ds = \int_{\ga_t^v}
\De^\Ga h_i(t,\bar s)d\bar s=0$.  We wrote $\De^\Ga$ for the operator seen in 
the coordinate $s\in \mathcal{S}$, but we also write $\De^\Ga$ for that defined in the
coordinate $\bar{s}$ on $\ga_t$.  However, since $\lambda_i(t)$, $i=1,\ldots,k$ are defined as
in \eqref{eq:lambda_k representation}, the first term of the right hand side of  \eqref{eq:mass cons 01} vanishes.
This explains the reason that we determine
$\la_k(t)$ as in \eqref{eq:lambda_k representation} and we obtain 
\begin{align}\label{eq:MASS CONS U_K^E}
&\int_D \Bigl( \partial_t u^\e_k(t,x) - v(t)\Bigr)dx = \Phi_k^\e(t),
\end{align}
where $\Phi_k^\e=\sum_{i=k+1}^\infty\e^i {\Bar{B}_{k,i}} -(B_2)-(C_2) + O(e^{-\zeta\de/\e})$.
Furthermore, note that each of $(B_2)$ and $(C_2)$ goes to $0$ as $\e\downarrow0$
exponentially fast due to the compatibility condition for $m'$, $\partial_\rho u_i$ and $\partial_t^\Ga u_i$.
%
\subsection{Estimates for $u_k$ and $u_k^\pm$}

\subsubsection{Bounds for spatial derivatives of  $b$, $X_0$, $d$ and $\bold{S}$}
%
We introduce the following norms for $g\in C^\infty(\R_+)$:

\begin{defn}\label{defn:definition01}
For $k\in\Z_+$ and $T>0$,
we define $|g|_k \equiv |g|_{k,T}$  as
\begin{align}
&|g|_{k,T} = \sum_{i=0}^k\sup_{t\in[0,T]}\biggl| \frac{d^ig}{dt^i}(t) \biggr|.
\end{align}
\end{defn}
In the following, we take and fix a class $\mathcal{V}$ of functions
$v \in C^{\infty}(\R_+)$ and $T>0$ satisfying that
\begin{align}  \label{eq:C-mathcal}
C_{\mathcal{V},T} = \max(C_{\mathcal{V},T}^{(1)},C_{\mathcal{V},T}^{(2)})<\infty,
\end{align}
where
\begin{align}  \label{eq:C-mathcal-02}
C_{\mathcal{V},T}^{(1)} := \sup_{\substack{v\in \mathcal{V},s\in \mathcal{S},\\r\in(-3\de,3\de)}}&
 \Bigl\{|\partial^{\bold{m}}d(\cdot,X_0(\cdot,s))|_{0,T},\,
 |\partial^{\bold{m}}S^l(\cdot,X(\cdot,r,s))|_{0,T},\\ 
 & \qquad \qquad  |\partial^{\bold{m}}X_0(\cdot,s)|_{0,T};\,
 1\leq l \leq n-1,\,|\bold{m}|\leq M  \Bigr\} < \infty,   \notag
\end{align}
and
\begin{align}  \label{eq:C-mathcal-01}
C_{\mathcal{V},T}^{(2)} :=\sup_{v\in \mathcal{V}, s\in \mathcal{S}, t\in[0,T]} 
  \Bigl\{ \big( \a_-(t,s)\big)^{-1}, \, |\ga_t^v|^{-1} \Bigr\} < \infty.
\end{align}
Here, in $C_{\mathcal{V},T}^{(1)}$, $M=M(K)\in\N$ denotes the maximal number of 
the degrees of spatial derivatives taken over the terms appearing in $\widetilde{A}^K$ 
which is defined below \eqref{eq:inner u_k (odd)} and
$\delta>0$ is chosen as in Section \ref{sec:2.2}.
Recall Section \ref{sec:2.2} for $d,X_0, S^l$ and $X$ determined depending on $v$,
and $\partial^{\bold{m}}=\partial_{x^1}^{m_1}\cdots\partial_{x^n}^{m_n}$, $|\bold{m}|=\sum_{i=1}^{n}m_i$
for $m=(m_1,\cdots,m_n)\in(\Z_+)^{n}$.
Moreover, in $C_{\mathcal{V},T}^{(2)}$, 
$$
\a_-(t,s) \equiv \a_-^v(t,s) := \inf_{\xi\in \R^{n-1}: |\xi|=1} \big(\a(t,s)\xi,\xi\big),
$$
where $\a(t,s) = (\a_{ij}(t,s))_{1\le i,j\le n-1}$ is the matrix defined by
\eqref{eq:alpha and beta} below, and $(\cdot,\cdot)$ and $|\cdot|$ denote
the inner product and the norm of $\R^{n-1}$, respectively.

Later, for each $\om \in \Om$, we will take
$\mathcal{V} \equiv \mathcal{V}(\om) := \{\a \dot{w}^\e; 0<\e \le \e_0^*\}$.  Then, for a.s.\ $\om$,
by Assumption \ref{ass:2}, one can take $T=T(\om):= \inf_{0<\e\le \e_0^*} \si^\e>0$ 
such that the equation \eqref{eq:2-e} with $v=\a\dot{w}^\e$
has a solution up to the time $T(\om)$ and  $C_{\mathcal{V}(\om),T(\om)}<\infty$ holds. 

Recalling the definition \eqref{eq:b(t,s)} of $b(t,s)$, the constant $C_{\mathcal{V},T}^{(1)}$ gives
bounds for $b$.  We need the bound $C_{\mathcal{V},T}^{(2)}$ for deriving estimates on
the fundamental solution (see \eqref{eq:3.44}, Lemma \ref{lem:Schauder estimate 2 plus a})
and on $\la_0(t)$ (see \eqref{eq:304-003}, Proposition \ref{prop:Lemma order of e}).

\subsubsection{Bounds for time derivatives of $X_0$, $\partial^{\bold{m}}X_0$ and $\bold{S}$}
Our goal is to give estimates for $u_k$ and $u_k^\pm$, see
Proposition \ref{prop:Lemma order of e} and Corollary \ref{lem:Lemma for H} below.
To do this, we will prepare several lemmas.
%
%
The inward normal velocity $V$ is represented as 
$V(t,s)=\partial_t d(t,X_0(t,s))$ from \eqref{eq:2.V} and 
\begin{align}\label{eq:eq of V}
V(t,s) = \k(t,s)-\frac{1}{|\ga_t^v|}\int_\mathcal{S} \k(t,s) J^0(t,s)ds 
+ \frac{|D|}{2|\ga_t^v|} v(t),
\end{align}
holds with $\k(t.s)=\Delta d(t,X_0(t,s))$ from \eqref{eq:lambda 0} and \eqref{eq:V + K}. 
Let us formulate the following assumption on the time derivatives of $\partial^{\bold{m}}X_0$ and
$\bold{S}$ by means of the norm of the forcing term $v$.

\begin{assump}\label{ass:3}
There exist some $N=N(K)\in\N$, $T=T(\mathcal{V})>0$ and
$C_1=C_1(C_{\mathcal{V},T},K,T)>0$ such that
\begin{align}
 &\sup_{1\leq i \leq n}\sup_{s\in \mathcal{S}}|\partial_t^k \partial^{\bold{m}} X_0^i(\cdot,s)|_{0,T}\leq C_1(1+|v|_{N,T})^{N},\label{eq:der X_0}\\
 &\sup_{1\leq i \leq n-1}\sup_{r\in(-3\delta,3\delta),\,s\in \mathcal{S}}|\partial_t^k\partial^{\bold{m}} S^i(\cdot,X(\cdot,r,s))|_{0,T}\leq C_1(1+|v|_{N,T})^{N},  \label{eq:der S}
\end{align}
for $k=0,1,\cdots,K$, $|\bold{m}|\leq M$ and $v\in\mathcal{V}$.
\end{assump}
Under the choice $\mathcal{V} \equiv \mathcal{V}(\om) = \{\a \dot{w}^\e; 0<\e\le \e_0^*\}$,
Assumption 3.1 determines $\t(\om) := T(\mathcal{V}(\om))$, up to which two bounds
\eqref{eq:der X_0} and \eqref{eq:der S} hold.
We will show in Section 5 that Assumption \ref{ass:3} is true for some $T=\t(\om)>0$ 
under a two-dimensional setting as long as the limit curve $\ga_t$ is convex.; see Lemma
\ref{lem:x(0) and Th in 2D case}.

\subsubsection{Schauder estimates for $h_k$}  \label{sec:3.3.3}
Recall that $h_k$ are the solutions of \eqref{eq:equation for h_k},
which is rewritten as 
\begin{align}\label{eq: Eq (h_k)}
  Lh_k+\mathcal{L}h_k=F_{k-1},
\end{align}
where $L=L_t$ is a differential operator defined by
\begin{align}
L =&\partial_t^{\Ga}-\Delta^{\Ga}-b(t,s)\\
=&(\partial_t+\sum_{i=1}^{n-1}S_t^i \partial_{s^i})-\left(\sum_{i=1}^{n-1}\Delta S^i \partial_{s^i}+\sum_{i,j=1}^{n-1}\nabla S^i\cdot\nabla S^j \partial_{s^is^j}^2\right)-b(t,s),
\nonumber
\\
=&\partial_t-\sum_{i,j=1}^{n-1}\a_{ij}(t,s)\partial_{s^is^j}^2-\sum_{i=1}^{n-1}\beta_i(t,s)\partial_{s^i}-b(t,s),
\nonumber
\end{align} 
with
\begin{align}\label{eq:alpha and beta}
&\a_{ij}(t,s)=\nabla S^i(t)\cdot\nabla S^j(t),\quad \beta_i(t,s)=-S_t^i(t) + \Delta S^i(t),
\\
&s=(S^1(t),\cdots,S^{n-1}(t))\in \mathcal{S},
\nonumber
\end{align}
while $\mathcal{L}=\mathcal{L}_t$ is an integral operator acting on a function $u=u(s)$ defined by
\begin{align*}
\mathcal{L}_tu=-\frac{1}{|\ga_t^v|}\int_\mathcal{S} \Bigl(\bigl(\Delta d(t,X_0(t,s')) - \bar\si\lambda_0(t)\bigr)
\Delta d(t,X_0(t,s'))-b(t,s')\Bigr)u(s')J^0(t,s') ds',
\end{align*}
and
$F_{k-1}$ denotes the right hand side of \eqref{eq:equation for h_k}.
Since $v$ and the interface $\gamma_t^v$ are smooth, it follows that \eqref{eq: Eq (h_k)} with an initial condition $h_k(0,s)=0$
has a local unique solution for each $v$ from general argument for second order parabolic partial differential equations. 

Our goal is to obtain estimates for the solutions $h_k$ of \eqref{eq:equation for h_k}. 
To do this, we basically follow the argument given in Friedman \cite{FR64} and 
derive the Schauder estimates.
However, in our setting, the coefficients of the operators $L$ and $\mathcal{L}$ determined from $v=\a\dot{w}^\e$ are not bounded in $\e$,
but controlled by the constant $C_{\mathcal{V},T}$ and the norm $|v|_{1,T}$.
We need to carefully study how the estimates for $h_k$ depend on these diverging factors.
We first treat the contributions of the operator $L$; see Lemmas \ref{lem:lem032} and \ref{lem:Schauder sup norm of h}.
The contribution of the operator $\mathcal{L}$ will be discussed starting from \eqref{eq:Fr-303} below.

For $\tau\in[0,T]$ and $\xi\in \mathcal{S}$, let $Z(t,s;\tau,\xi)$ be the fundamental solution of
\begin{align}\label{eq:z-01}
& \Bigl(\frac{\partial}{\partial t} - \sum_{i,j=1}^{n-1}\a_{ij}(\tau,\xi) \frac{\partial^2}{\partial s^i\partial s^j}\Bigr) u(t,s) = 0,\quad (t,s)\in(\tau,T]\times \mathcal{S},
\end{align}
that is, $Z$ is a Gaussian kernel.
Then, the fundamental solution $\Ga$ of $L_t$ is constructed
by the usual parametrix method. Namely, by regarding $Z$ as the principal part of $\Ga$, 
we can find $\Ga$ in the form
\begin{align}\label{eq:z-03}
& \Ga(t,s;\tau,\xi)=Z(t,s;\tau,\xi)+\int_\tau^t \int_\mathcal{S} Z(t,s;\si,s') \Phi(\si,s';\tau,\xi) d\si ds',
\end{align}
where $\Phi$ is a function satisfying
\begin{align}\label{eq:z-04}
& \Phi(t,s;\tau,\xi) = L_tZ(t,s;\tau,\xi) + \int_\tau^t\int_\mathcal{S} L_tZ(t,s;\si,s') \Phi(\si,s';\tau,\xi) d\si ds'.
\end{align}
Let $\mu\in(\frac12,1)$ and fix it. Then, we have the following estimate on $\Ga$.
\begin{lem}\label{lem:lem 1-1-1}
The fundamental solution $\Ga$ has the estimate:
\begin{align}\label{eq:z-Ga-01}
|\Ga(t,s;\tau,\xi)|\leq
&\,C_8(1+|v|_N)^{\nu_0+\frac{N}{1-\mu}}e^{C_8 (T\vee 1)(1+|v|_N)^{\frac{N}{1-\mu}}}
\frac{1}{(t-\tau)^{\frac{n-1}{2}}} e^{- C_9 \frac{|s-\xi|^2}{(t-\tau)} },
\end{align}
for $0\le\tau<t\leq T$, $s,\,\xi \in \mathcal{S}$ with some $\nu_0\in \N$,
$C_8=C_8(C_{\mathcal{V},T},T,\mu)$ and $C_9= C_9(C_{\mathcal{V},T},T,\mu)>0$.
\end{lem}
\begin{proof}
Let us rewrite $L_t Z$ into 
\begin{align}\label{eq:z-05}
L_tZ(t,s;\tau,\xi)=&\sum_{i,j=1}^{n-1} \Bigl(\a_{ij}(\tau,\xi) - \a_{ij}(t,s) \Bigr)\frac{\partial^2}{\partial s^i\partial s^j}Z(t,s;\tau,\xi)\\
&-\sum_{i=1}^{n-1} \beta_i(t,s) \frac{\partial}{\partial s^i}Z(t,s;\tau,\xi)-b(t,s)Z(t,s;\tau,\xi).
\nonumber
\end{align}
From \eqref{eq:b(t,s)}, \eqref{eq:alpha and beta}, \eqref{eq:C-mathcal} and Assumption \ref{ass:3}, we have
\begin{align}\label{eq:z-07}
&|\a_{ij}(t,s) - \a_{ij}(\tau,\xi)|\leq C_2(1+|v|_N)^N\bigl(|t-\tau|+|s-\xi|\bigr),\\
&|\beta_{i}(t,s)-\beta_i(\tau,\xi)|\leq C_2(1+|v|_N)^N\bigl(|t-\tau|+|s-\xi|\bigr),\label{eq:z-07 beta}\\
&|b(t,s)-b(\tau,\xi)|\leq C_2(1+|v|_N)^N\bigl(|t-\tau|+|s-\xi|\bigr), \label{eq:z-07 b}
\end{align}
for some $C_2=C_2(C_{\mathcal{V},T},T)>0$.
Then, according to \cite{FR64}, Chapter 1, (4.3) with $\a=1$ and $n$ replaced by $n-1$, we have
\begin{align}\label{eq:z-08}
&\Bigl|L_tZ(t,s;\tau,\xi)\Bigr|\leq C_3\frac{(1+|v|_N)^N}{(t-\tau)^\mu}\frac{1}{|s-\xi|^{n-2\mu}},
\end{align}
for some $C_3=C_3(C_{\mathcal{V},T},T)>0$.
Let us define
\begin{align}\label{eq:z-09}
(L_tZ)_1(t,s;\tau,\xi) &= L_tZ(t,s;\tau,\xi),\\
(L_tZ)_{\nu+1}(t,s;\tau,\xi) &= \int_\tau^t\int_\mathcal{S}
 L_tZ(t,s;\si,s')(L_tZ)_{\nu}(\si,s';\tau,\xi)d\si ds'. \label{eq:z-09-001}
\end{align}
As is seen in \cite{FR64} p.\ 14, the singularity of $(L_tZ)_{2}$ is weaker than that of $(L_tZ)_{1}$.
Indeed, we have
\begin{align}\label{eq:z-09-01}
|(L_tZ)_2(t,s;\tau,\xi)| &\leq \frac{C_4(1+|v|_N)^{2N}}{(t-\tau)^{2\mu-1}|s-\xi|^{(n-2\mu)+(1-2\mu)}},
\end{align} 
for some $C_4=C_4(C_{\mathcal{V},T},T)>0$; recall $\mu\in(\frac12,1)$.
By applying similar arguments to $(L_tZ)_{3}$, $(L_tZ)_{4}$ \ldots, we can find an integer $\nu_0>0$ such that
\begin{align}\label{eq:z-10}
&\Bigl|(L_tZ)_{{\nu_0}}(t,s;\tau,\xi)\Bigr|\leq C_5\Bigl( (1+|v|_N)\Bigr)^{\nu_0},
\end{align}
for some $C_5=C_5(C_{\mathcal{V},T},T,\mu)>0$.
In particular, the singularity disappears after $\nu_0$-steps.

In what follows, we will show
\begin{align}\label{eq:z-11}
&\Bigl|(L_tZ)_{{\nu_0}+m}(t,s;\tau,\xi)\Bigr|\leq C_6 \frac{\Bigl( C_7 (1+|v|_N)^N\Bigr)^m}{\Ga((1-\mu)m+1)}\Bigl((1+|v|_N)\Bigr)^{\nu_0}(t-\tau)^{(1-\mu)m},
\end{align}
for some $C_6=C_6(C_{\mathcal{V},T},T,\mu)>0$ and $C_7=C_7(C_{\mathcal{V},T},T,\mu)\ge 1$.
In the case $m=0$, this is clear from \eqref{eq:z-10}.
Suppose that \eqref{eq:z-11} holds for $m$, and let us show that this holds also for $m+1$. 
Indeed, using \eqref{eq:z-09-001} with $\nu=\nu_0+m$ and \eqref{eq:z-08},
\begin{align}\label{eq:z-12}
\Bigl|(L_tZ)_{{\nu_0}+m+1}(t,s;\tau,\xi)\Bigr|
\leq& \,C_3C_6 \frac{\Bigl( C_7 (1+|v|_N)^N\Bigr)^m}{\Ga((1-\mu)m+1)}
\Bigl((1+|v|_N)\Bigr)^{\nu_0}(1+|v|_N)^N\\
&\times\int_\tau^t (t-\si)^{-\mu}(\si-\tau)^{(1-\mu)m}d\si\int_\mathcal{S} |s-s'|^{-n+2\mu}ds',
\nonumber
\end{align}
Changing the variable $\rho=(\si-\tau)(t-\tau)^{-1}$, we obtain
\begin{align}\label{eq:z-13}
&\int_\tau^t (t-\si)^{-\mu}(\si-\tau)^{(1-\mu)m}d\si
= (t-\tau)^{(1-\mu)(m+1)}\frac{\Ga(1-\mu)\Ga((1-\mu)m+1)}{\Ga(2-\mu+(1-\mu)m)},
\end{align}
where we have used
\begin{align}\label{eq:z-14}
&\int_0^1 (1-\rho)^{a-1}\rho^{b-1}d\rho = \frac{\Ga(a)\Ga(b)}{\Ga(a+b)},\quad a,b>0,
\end{align}
and $\Ga(a)=\int_0^\infty e^{-x}x^{a-1}dx$.
Let us take
\begin{align}  \label{eq:3.43}
C_7=\max\Bigl(C_3\Ga(1-\mu) \sup_{s\in \mathcal{S}}
\int_\mathcal{S} |s-s'|^{-n+2\mu}ds',\,1\Bigr) \ge 1.
\end{align}
Note that the integral in \eqref{eq:3.43} is finite because of $\mu\in(\frac12,1)$.
Thus, \eqref{eq:z-11} holds for $m+1$ with this choice of $C_7$.

Noting that
\begin{align*}
 &\sum_{m=1}^\infty\frac{\Bigl(C_7 (1+|v|_N)^{N}\Bigr)^{m}}{\Ga((1-\mu)m+1)}(t-\tau)^{(1-\mu)m}\\
\leq&\sum_{m=1}^\infty\frac{\Bigl( C_7^{\frac1{1-\mu}}(1+|v|_N)^{\frac{N}{1-\mu}}\Bigr)^{(1-\mu)m}}{\Ga((1-\mu)m+1)}T^{(1-\mu)m}\\
\leq&\,\sum_{m=1}^\infty\frac{\Bigl( C_7^{\frac1{1-\mu}} 
(1+|v|_N)^{\frac{N}{1-\mu}}\Bigr)^{[(1-\mu)m]+1}}{{C}_\Ga[(1-\mu)m]!}(T\vee 1)^{[(1-\mu)m]+1}\\
\leq&\,C_{\mu,\Ga}(T\vee 1)(1+|v|_N)^{\frac{N}{1-\mu}}\sum_{m=0}^\infty
\frac{\Bigl( C_7^{\frac1{1-\mu}} (1+|v|_N)^{\frac{N}{1-\mu}}\Bigr)^m (T\vee 1)^{m} }{m!},
\end{align*}
for some $C_{\mu,\Ga}>0$, where we have used $\Ga((1-\mu)m+1)\geq {C}_\Ga[(1-\mu)m]!$ 
with ${C}_\Ga=\inf_{1\leq a < 2}\Ga(a)$ for the second inequality, we obtain 
\begin{align}\label{eq:z-15}
&\sum_{m=1}^\infty \Bigl|(L_tZ)_{{\nu_0}+m}(t,s;\tau,\xi)\Bigr| 
\leq C_8 (1+|v|_N)^{\nu_0+\frac{N}{1-\mu}} e^{ C_8 (T\vee 1)(1+|v|_N)^{\frac{N}{1-\mu}}},
\end{align}
for some $C_8=C_8(C_{\mathcal{V},T},T,\mu)>0$.
By \eqref{eq:z-15}, recalling $C_{\mathcal{V},T}$ gives the lower bound for the 
diffusion matrix $\a(t,s)$ in the sense that $\a(t,s) \ge \big(C_{\mathcal{V},T}^{(2)}\big)^{-1}I$,
and using \cite{FR64}, p.\ 16, we have 
\begin{align}  \label{eq:3.44}
& |\Phi(t,s;\tau,\xi)|\leq C_8 (1+|v|_N)^{\nu_0+\frac{N}{1-\mu}} e^{C_8 (T\vee 1)(1+|v|_N)^{\frac{N}{1-\mu}}} 
\frac{1}{(t-\tau)^{\mu}|s-\xi|^{n-2\mu}} e^{-{C_9}\frac{|s-\xi|^2}{(t-\tau)}},
\end{align}
for some $C_9 = C_9(C_{\mathcal{V},T},T,\mu)>0$.
Thus, by the standard argument (see \cite{FR64}, Chapter 1, Section 6), 
we obtain the estimate \eqref{eq:z-Ga-01} for  the fundamental solution $\Ga$.
\end{proof}
Consider the following equation: 
\begin{align}\label{eq:Fr-301}
\left\{
\begin{array}{ll}
 L_tu(t,s) =f(t,s), &\quad  (t,s)\in (0,T]\times \mathcal{S},\\
 u(0,s)=u_0(s),   &\quad s\in \mathcal{S}, 
\end{array}
\right.
\end{align}
for $f\in \overline{C}^\a$, $\a\in(0,1]$ and $u_0 \in C(\mathcal{S})$,
where 
\begin{align}
& \overline{C}^\a =\{ u(t,s) | \, |\overline{u}|_\a < \infty \},\nonumber\\
& |\overline{u}|_\a \equiv |\overline{u}|_{\a,T}
= |u|_0 + \sup_{\substack{(t,s),(\overline{t},\overline{s})\\ \in [0,T]\times \mathcal{S}}}\frac{|u(t,s)-u(\overline{t},\overline{s})|}{\sqrt{|s-\overline{s}|^2+|t-\overline{t}|}^\a},\nonumber\\
& |u|_0 \equiv |u|_{0,T}
=\sup_{(t,s)\in [0,T]\times \mathcal{S}}|u(t,s)|,\quad |s|=\Big(\sum_{i=1}^{n-1}(s^i)^2\Big)^{\frac12}.\nonumber
\end{align}
Applying the argument of \cite{FR64}, Chapter 3, Section 3, Theorem 7 for our problem,
we obtain the following lemma.
\begin{lem}\label{lem:lem032}
There exists a unique solution $u$ of \eqref{eq:Fr-301}
and $u\in \overline{C}^{2+\a}$,
where 
\begin{align}
& \overline{C}^{2+\a} =\{ u(t,s) | \, |\overline{u}|_{2+\a} < \infty \},\nonumber\\
& |\overline{u}|_{2+\a} = |\overline{u}|_\a
  +\sum_{|\bold{m}|=1}|\overline{ \partial^{\bold{m}}u }|_\a + \sum_{|\bold{m}|=2}
|\overline{ \partial^{\bold{m}}u }|_\a+\sum|\overline{\partial_tu}|_\a.\nonumber
\end{align}
\end{lem}

Furthermore, the following lemma is shown by using \cite{FR64}, Chapter 5, Section 3, Lemma 2.
\begin{lem}\label{lem:Schauder sup norm of h}
Let $u$ be the solution of \eqref{eq:Fr-301}.
Then, we have
\begin{align}\label{eq:Schauder u}
&|u|_0 \leq K_1\Bigl(\sup_{s\in \mathcal{S}}|u_0(s)|+
T \sup_{(t,s)\in[0,T]\times \mathcal{S}}|f(t,s)| \Bigr),
\end{align}
where
\begin{align}\label{eq:value of K2}
K_1 \equiv K_1(v) = C_{11} \,e^{2C_8 (1+|v|_N)^{p}(T\vee 1)},
\end{align}
for some $C_{11}=C_{11}(C_{\mathcal{V},T},T,\mu)>0$ and $p=p(\mu,\nu_0)\in \N$.
\end{lem}

\begin{proof}
Since the solution $u$ is given by
\begin{align}
  u(t,s)=\int_\mathcal{S} \Ga(t,s;0,\xi) u_0(\xi)d\xi 
+ \int_0^t \int_\mathcal{S} \Ga(t,s;\si,\xi) f(\si,\xi)d\si d\xi,
\end{align}
we have
\begin{align}\label{eq: 303-001}
&|u|_0 \leq K_2\Bigl(\sup_{s\in \mathcal{S}}|u_0(s)|+
 \sup_{(t,s)\in[0,T]\times \mathcal{S}} t |f(t,s)| \Bigr),
\end{align}
where
\begin{align}\label{eq: 303-002}
K_2 = C_{10} (1+|v|_N)^{\nu_0+\frac{N}{1-\mu}}(T\vee 1) 
e^{C_8 (T\vee 1)(1+|v|_N)^{\nu_0+\frac{N}{1-\mu}}},
\end{align}
for some $C_{10}=C_{10}(C_{\mathcal{V},T},T,\mu)>0$ by Lemma \ref{lem:lem 1-1-1};
note that we have estimated $(1+|v|_N)^{\frac{N}{1-\mu}}
\le (1+|v|_N)^{\nu_0+ \frac{N}{1-\mu}}$ in the exponential, since
$1+|v|_N \ge 1$.  By choosing suitably large integer $p \geq \nu_0+\frac{N}{1-\mu}$
and using $e^{2y}\geq ye^y$ for every $y\geq0$, 
the proof is complete by taking suitable constant $C_{11}=C_{11}(C_{\mathcal{V},T},T,\mu)>0$.
\end{proof}
We now study the equation \eqref{eq: Eq (h_k)} for $h_k$ taking the contribution
of the operator $\mathcal{L}_t$ into account.  Let us consider 
\begin{align}\label{eq:Fr-303}
\left\{
\begin{array}{ll}
 L_t u(t,s) =g(t,s,u), &\quad  (t,s)\in(0,T]\times \mathcal{S},\\
 u(0,s)=0,        &\quad  s\in \mathcal{S},
\end{array}
\right.
\end{align}
where 
\begin{align}
&g(t,s,u) = -\mathcal{L}_tu + F_{k-1}(t,s).
\end{align}
\begin{lem}\label{lem:lem 034}
There exists a unique solution $u$ of \eqref{eq:Fr-303} satisfying 
\begin{align}\label{eq:Fr-eq-1003}
|u|_{0,T}\equiv 
&\sup_{(t,s)\in [0,T]\times \mathcal{S}}|u(t,s)| \leq \Bigl(C_{13}K_3\Bigr)^{C_{13}K_3}|F_{k-1}|_{0,T},
\end{align}
for some $C_{13}=C_{13}(C_{\mathcal{V},T},T,\mu)>0$ and 
\begin{align*}
K_3\equiv K_3(v) =e^{4 C_8 (1+|v|_N)^{1+p}(T\vee 1)}.
\end{align*}
\end{lem}
\begin{proof}
Let us note that $F_{k-1} \in \overline{C}^{\a}$ for each $v$; recall \eqref{eq: Eq (h_k)}
and \eqref{eq:equation for h_k} for $F_{k-1}$.  By Lemma \ref{lem:lem032}, for given 
continuous functions $w=w(t,s)$ on $[0,T]\times \mathcal{S}$ and $\psi=\psi(s)$ on 
$\mathcal{S}$, the equation
\begin{align}\label{eq:Fr-304}
\left\{
\begin{array}{ll}
 L_t u(t,s) = g(t,s,w), &\quad  (t,s)\in(0,T]\times \mathcal{S},\\
 u(0,s)=\psi(s),        &\quad  s\in \mathcal{S},
\end{array}
\right.
\end{align}
has a unique solution $u$; note that it holds $\mathcal{L}_tw\in \overline{C}^{\a}$,
which depends only on $t$.
Then, by Lemma \ref{lem:Schauder sup norm of h} 
with $T=\de$, we have
\begin{align}\label{eq:304-002}
|u|_{0,\de}\equiv 
  \sup_{(t,s)\in[0,\delta]\times \mathcal{S}}|u(t,s)|\leq K_2 |\psi|_0 
+ K_2 \delta |g(\cdot,\cdot,w)|_{0,\de},
\end{align}
for every $\delta \in (0,T]$.  However, recalling the definitions of $\mathcal{L}_t$,
$b(t,s)$ in \eqref{eq:b(t,s)} and $\la_0(t)$ in \eqref{eq:lambda 0}, and especially
noting $1/|\ga_t^v|\le C_{\mathcal{V},T}^{(2)}$ to estimate the second term
of $\la_0(t)$, we have
\begin{align}\label{eq:304-003}
& |g(\cdot,\cdot,w)|_{0,\de} \leq C_{12}(1+|v|_0)|w|_{0,\de} + |F_{k-1}|_{0,\de},
\end{align}
for some $C_{12}=C_{12}(C_{\mathcal{V},T},T)>0$, and therefore from 
\eqref{eq:304-002},
\begin{align}\label{eq:304-004}
  |u|_{0,\de} \leq K_2 |\psi|_0 + C_{12} K_2 \delta (1+|v|_0)|w|_{0,\de} + K_2 \delta |F_{k-1}|_{0,\de},
\end{align}

Now choose $\delta\in(0,T]$ sufficiently small such that
\begin{align}\label{eq:304-004-001}
\delta= \tfrac12 \Bigl(C_{12} K_2 (1+|v|_0)\Bigr)^{-1} \wedge T \wedge 1,
\end{align}
and fix it.
Set
\begin{align}\label{eq:304-005}
R_1= 2 K_2\Bigl(|\psi|_0 +  |F_{k-1}|_{0,\de} \Bigr).
\end{align}
We denote the map $w \mapsto u$ on $C([0,\delta]\times \mathcal{S})$
determined by solving the equation \eqref{eq:Fr-304} by $\La$, so that $\La w = u$.
Then, for every $w$ satisfying $|w|_{0,\de}\leq R_1$, \eqref{eq:304-004} shows
that $|\La w|_{0,\de} \leq R_1$.  In particular, $\La$ is a map on $C_{R_1}:=
\{w\in C([0,\delta]\times \mathcal{S}) \big| |w|_{0,\de}\leq R_1\}$.  Moreover,
for $w_1, w_2 \in C_{R_1}$, since $u:= \La w_1-\La w_2$ is a solution of 
\eqref{eq:Fr-304} with $w=w_1-w_2$, $\psi=0$ and $F_{k-1}=0$, we have
from \eqref{eq:304-004}
$$
|u|_{0,\de} \le C_{12} K_2 \de (1+|v|_0) |w_1-w_2|_{0,\de}
\le \tfrac12 |w_1-w_2|_{0,\de}.
$$
This implies that $\La$ is a strict contraction map.
Therefore, the solution $u$ of \eqref{eq:Fr-303} satisfies
\begin{align}\label{eq:304-006}
|u|_{0,\de} = &\sup_{(t,s)\in [0,\delta]\times \mathcal{S}}|u(t,s)| \leq R_1,
\end{align}
where $R_1$ is that with $\psi=0$.

Next, applying the above procedure to the equation \eqref{eq:Fr-304} with $t$ 
replaced by $t+\delta$ and $\psi$ replaced by $u(\delta,\cdot)$, 
for the solution $u$ of \eqref{eq:Fr-303}, we have
\begin{align}\label{eq:304-007}
&\sup_{(t,s)\in [0,\delta]\times \mathcal{S}}|u(t+\delta,s)| \leq R_2,
\end{align}
where $R_2$ is given by $R_1$ with $\psi$ replaced by $u(\delta)$ in
\eqref{eq:304-005}.  By repeating this procedure recursively, assuming $2K_2\ge 2$
for simplicity (or we may replace $K_2$ by $K_2\vee 1$ if necessary),
we finally obtain
\begin{align}\label{eq:304-008}
\sup_{(t,s)\in [0,T]\times \mathcal{S}}|u(t,s)| 
&\leq \sum_{i=1}^{[\frac{T}{\delta}]+1}(2K_2)^i|F_{k-1}|_{0,T}
\leq (2K_2)^{[\frac{T}{\delta}]+2}|F_{k-1}|_{0,T}\\
&\leq\Bigl(2K_2\Bigr)^{T\big( 2C_{12}K_2(1+|v|_0)\vee T \vee 1\big)+2}|F_{k-1}|_{0,T},
\nonumber
\end{align}
recall $K_2$ is given by \eqref{eq:value of K2}.  In particular, taking some $C_{13}
=C_{13}(C_{\mathcal{V},T},T,\mu)>0$, the right hand side of \eqref{eq:304-008} is bounded
from above by $(C_{13}K_3)^{C_{13}K_3}|F_{k-1}|_{0,T}$.
The proof is complete.
\end{proof}
As for the regularity of the solution $u$ of \eqref{eq:Fr-303}, we have the following 
lemma. The proof is given based on the arguments in \cite{FR64} or \cite{Lb96}.

\begin{lem}\label{lem:Schauder estimate 2 plus a}
For the solution $u$ of \eqref{eq:Fr-303} satisfying \eqref{eq:Fr-eq-1003}, 
\begin{align}\label{eq:u 2 plus alpha}
  |u|_{2+\a} \leq K_4(|u|_0+|F_{k-1}|_{\a}),
\end{align}
holds for $K_4\equiv K_4(v) =C_{14}(1+|v|_{n_0(K)})^{n_0(K)}$, where 
$C_{14}=C_{14}(C_{\mathcal{V},T},T,\mu)>0$ and $n_0=n_0(K) \in \N$.
\end{lem}
\begin{proof}
We can make use of \cite{FR64}, Chapter 4, Section 1, Theorem 1 and Chapter 4, Section 4
or \cite{Lb96}, Chapter IV, Theorem 4.28, recalling $\a(t,s) \ge 
\big( C_{\mathcal{V},T}^{(2)}\big)^{-1} I$.  
Indeed, we rewrite \eqref{eq:Fr-303} into
\begin{align}
 &\sum_{i,j=1}^{n-1} \a_{ij}(t',s')\partial^2_{s^is^j}u-\partial_t u\\
=& \sum_{i,j=1}^{n-1}\Bigl(\a_{ij}(t',s')-\a_{ij}(t,s)\Bigr)\partial^2_{s^is^j}u-\sum_{i=1}^{n-1} \beta_i(t,s)\partial_{s^i} u -b(t,s)u + \mathcal{L}_t u_t - F_{k-1},
\nonumber
\end{align}
with a fixed $(t',s')\in [0,T]\times \mathcal{S}$.  Note that, as we saw for
\eqref{eq:304-003}, we have
\begin{align}\label{eq:Est of mathcal(L)}
|\mathcal{L}_t u|_0 \leq C_{12}(1+|v|_0)|u|_0.
\end{align}
Using \eqref{eq:Est of mathcal(L)}, \eqref{eq:z-07}, \eqref{eq:z-07 beta} and \eqref{eq:z-07 b},  
by proceeding similarly to the argument in \cite{FR64} Chapter 4, Section 4,
it follows that the upper bound $K_4$ can be chosen as a polynomial of the norms of $v$.
Therefore, we obtain \eqref{eq:u 2 plus alpha}. 
\end{proof}

\subsubsection{Estimates for $u_k$ and $u_k^\pm$}
One can apply Lemma \ref{lem:Schauder estimate 2 plus a} for the solutions
$h_k$ of \eqref{eq: Eq (h_k)} to obtain estimates for them.  
Based on this, we have the following estimates for $u_k$ and $u_k^\pm$.
\begin{prop}\label{prop:Lemma order of e}
For every $k=0,1,\ldots,K$,
\begin{align}\label{eq:Est of u and u-pm}
\sup_{(t,\rho,s)\in  [0,T]\times \R\times \mathcal{S}}
\bigl\{
|u_k(t,\rho,s)|,|u_k^{\pm}(t)|\bigr\}\leq 
(C_{15} K_5)^{C_{15}K_5},
\end{align}
holds for some $C_{15}=C_{15}(C_{\mathcal{V},T},T,\mu)>0$ and 
\begin{align*}
K_5\equiv K_5(v) =e^{n_1(K)(1+|v|_{n_1(K)})^{n_1(K)}(T\vee 1)},
\end{align*}
with some $n_1=n_1(K)\in\N$.
\end{prop}

\begin{proof}
First note that the estimates for $|\partial^{\textbf{m}}X_0(\cdot,s)|_0$ 
and $|\partial^{\textbf{m}}S^j(\cdot,x)|_0$  for every $|\textbf{m}|\leq M$
are given by Assumption \ref{ass:3}.
Recalling that $u_0$ and $u_0^\pm$ are determined by
\eqref{eq:inner u_0} and \eqref{eq:u0 outer}, respectively,
since we have $|\la_0(t)| \le C\cdot C_{\mathcal{V},T}(1+|v|_{0,T})$
from \eqref{eq:lambda 0} and $1/|\ga_t^v| \le C_{\mathcal{V},T}^{(2)}$
as we saw for \eqref{eq:304-003}, \eqref{eq:Est of u and u-pm} is easily
shown for $u_0$ and $u_0^\pm$.

From this, one can derive the estimate for $F_0$ defined by the right hand
side of \eqref{eq:equation for h_k} with $k=1$, since $F_0$ contains $u_0$, 
$u_0^\pm$ and $b$. Thus, the estimate for $|h_1|_0$ is obtained by 
Lemma \ref{lem:lem 034}.  In addition, the estimates for 
$|\partial_t h_1|_0$, $|\partial_{s^i} h_1|_0$ and $|\partial_{s^is^j}^2 h_1|_0$ are obtained
by Lemma \ref{lem:Schauder estimate 2 plus a}.
Similarly, the estimate for $|\partial^{\textbf{m}} \partial_t^lF_0|_0$, 
$|\textbf{m}|\leq M$,  $l\leq K$ is obtained.
Differentiating \eqref{eq: Eq (h_k)} and applying the above argument recursively, 
the estimates for $|\partial_t^{l} h_1|_0$ and $|\partial^{\textbf{m}} h_1|_0$, 
$|\textbf{m}|\leq M$, $l\leq K$ are also obtained.

Recall that $u_1$ is represented by using $\widetilde{A}^1$, which contains
$\lambda_0$, $u_0$, $h_1$.  Therefore, the estimate for $|u_1|_0$ is obtained from
those for $\lambda_0$, $u_0$, $h_1$.  Once we have the estimate for $|u_1|_0$, 
similarly as above, the estimate for $|F_1|_0$ follows.
From this, we have the estimate for $|h_2|_0$ by Assumption \ref{ass:3}. 

In this way, the estimates for $u_k$, $u_k^\pm$, $k\leq K$ are obtained recursively
by deriving the estimate for each term in $\widetilde{A}^{k}$.
Since $K$ is finite and $\widetilde{A}^{k}, k\le K$ contain only finitely many terms,
it is easy to see that there exists a suitable $n_1(K)$ such that 
\eqref{eq:Est of u and u-pm} holds.
The proof is complete.
\end{proof}

Finally in this section, we consider the random case taking $v(t) = \a \dot{w}^\e(t)$,
where $w^\e(t)$ satisfies Assumption 1.1, and apply Proposition \ref{prop:Lemma order of e}
for this case.

\begin{cor}\label{lem:Lemma for H}
We assume Assumptions 1.1, 1.2 and 3.1, and define $G_\e\ge e^e, 0<\e\le 1,$ 
from $H_\e$ appearing in Assumption 1.1 by the relation
\begin{align}\label{eq:def of He(K) 01}
 \log\log G_{\e} &=H_\e^{2 n_1(K)},
\end{align}
where $n_1(K)\in\N$ is the number determined by Proposition \ref{prop:Lemma order of e}.
Then, we have
\begin{align}
\lim_{\e\downarrow0}G_\e=\infty, \quad
\lim_{\e\downarrow0}\frac{G_\e}{|\log\e|}=0.
\label{eq:v2}
\end{align}
Furthermore, $u_k$ and $u_k^\pm$ determined from 
$v(t)=\a\dot{w}^\e(t)$ as above satisfy
\begin{align}\label{eq:Bounded case 01}
\sup_{(t,\rho,s)\in  [0,T(\om)]\times \R\times \mathcal{S}}
\bigl\{
|u_k(t,\rho,s)|,|u_k^{\pm}(t)|\bigr\}\leq G_\e,\quad 0\leq k \leq K,
\end{align}
for every sufficiently small $\e>0$ and every $\om\in \Om$, where $T(\om)>0$
is  the minimum of that determined at the end of Section 3.3.1 and $\t(\om)$
given below Assumption 3.1 in Section 3.3.2.
\end{cor}

\begin{proof}
The first one in \eqref{eq:v2} is clear from \eqref{eq:def of He(K) 01} and
$\lim_{\e\downarrow 0} H_\e=0$.  For the second one, use twice that 
$\lim_{\e\downarrow 0}\frac{\log a_\e}{\log b_\e}=0$
and $\lim_{\e\downarrow 0} b_\e=\infty$ imply
$\lim_{\e\downarrow 0}\frac{a_\e}{b_\e}=0$.  To show \eqref{eq:Bounded case 01},
we see from Assumption 1.1 that $v=\a\dot{w}^\e$ satisfies
\begin{align}  \label{eq:4.7}
  & |v|_{n_1(K),T}\leq (n_1(K)+1) \a H_\e.
\end{align}
From this, one can verify
\begin{align}  \label{eq:4.9}
 & (C_{15} K_5)^{C_{15}K_5} \leq G_\e,
\end{align}
for every sufficiently small $\e>0$, where the left hand side is the bound 
obtained in Proposition \ref{prop:Lemma order of e}; note that $K_5= K_5(v)$
with $v=\a \dot{w}^\e$.  Indeed, 
\begin{align}  \label{eq:4.10}
\log \log \left\{(C_{15} K_5)^{C_{15}K_5}\right\}
= \log C_{15} + \log K_5 + \log\log (C_{15}K_5),
\end{align}
whose leading term for small $\e>0$  is $\log K_5$, and it has a bound
\begin{align*}
\log K_5 & = n_1(K) (1+ |v|_{n_1(K)})^{n_1(K)} (T\vee 1)
\le C\, H_\e^{n_1(K)} < H_\e^{2n_1(K)},
\end{align*}
with some $C>0$, by the definition of $K_5=K_5(v)$ and \eqref{eq:4.7},
for every $\e>0$ small enough, since $\lim_{\e\downarrow 0} H_\e=\infty$.
Note that, recalling the remark $C_{\mathcal{V}(\om),T(\om)}<\infty$
for $\mathcal{V}(\om) = \{\a\dot{w}^\e; 0 < \e \le \e_0^*\}$
made at the end of Section 3.3.1, $C_{15}= C_{15}(C_{\mathcal{V}(\om),T(\om)},
T(\om),\mu)$ is bounded in $\e$ so that the other terms in \eqref{eq:4.10}
are much smaller than $\log K_5$.  Thus, we obtain \eqref{eq:4.9} for every $\e>0$
small enough. Proposition \ref{prop:Lemma order of e} completes the proof.
\end{proof}

\section{Proof of Theorem \ref{thm:Th1}}
%
%
%
%
\subsection{The stochastic term $w^\e(t)$}
This subsection gives two examples of $w^\e(t)$, which satisfy the condition
\eqref{eq:definition 2 of g} and Assumption \ref{ass:1}.  The integer $K$ is
fixed throughout the paper.

The first example is a mollification of the Brownian motion with an extremely
slow convergence speed.  Let $w=w(t)$ be the one-dimensional standard Brownian motion
and set
\begin{align}\label{eq:def of psi(e)}
\psi(\e) = \bigl(\log\,\log\,\log|\log\e|\bigr)^{\tilde{\beta}},
\end{align}
with $\tilde{\beta}>0$.  Let $W_\e(t), \e>0$ be the stopped process of $w$, 
that is, $W_\e(t)=w(t\wedge \tau(\e))$,
where $\tau(\e)$ is the first exit time of $w(t)$ from the interval 
$I_\e=(-\psi(\e),\psi(\e))$, that is,
$\tau(\e) = \inf\{t>0,\, w(t) \notin I_\e\}$.
We define $w^\e(t)$ by
\begin{align}  \label{eq:4-1st-example}
{w}^\e(t)=\int_0^\infty \eta_{\psi(\e)}(t-s)W_\e(s)ds, 
\end{align}
where 
\begin{align*}
&\eta_{\psi(\e)}(s)= \psi(\e)\eta\big(\psi(\e) s\big)
\end{align*}
and $\eta$ is a non-negative $C^{\infty}$-function on $\R$, whose support is contained 
in $(0,1)$, satisfying $\int_{\R} \eta(u)du=1$.  

\begin{lem}  \label{lem:4.1}
For $w^\e(t)$ defined by \eqref{eq:4-1st-example}, we have
\begin{align}  \label{eq:v1}
  |\dot{w}^\e|_{k,T} \leq k |\eta|_{k+2} \psi(\e)^{k+2}, \quad k\in \Z_+.
\end{align}
In particular, Assumption 1.1 holds for this $w^\e(t)$ with the choice
of $H_\e= n_1(K) |\eta|_{n_1(K)+2} \psi(\e)^{n_1(K)+2}$.  The condition 
\eqref{eq:definition 2 of g} holds obviously.
\end{lem}

\begin{proof}
A simple computation shows \eqref{eq:v1}.  Once \eqref{eq:v1} is shown,
the rest is obvious.
\end{proof}

\begin{rem}
If we choose $\tilde\b>0$ as $2\tilde\b n_1(K)(n_1(K)+2)=1$,
we have $H_\e^{2n_1(K)}= (n_1(K) |\eta|_{n_1(K)+2})^{2n_1(K)} \log\log\log |\log\e|$.  Since the
condition  \eqref{eq:def of He(K) 01} can be relaxed as
$\log\log G_\e$ $= C\, H_\e^{2n_1(K)}$ with any constant $C>0$, one can choose
$G_\e= \log |\log\e|$ for every small enough $\e>0$.
\end{rem}

Let us give another example of $w^\e(t)$.
Let $\xi=\xi(t)$, $t\ge 0$ be a stochastic process satisfying the following conditions,
see \cite{F99}:
\begin{enumerate}
 \item $\xi$ is a stationary and strongly mixing stochastic process defined on a probability space $(\Omega,\mathcal{F},P)$, that is, $\int_0^\infty \rho(t)^{\frac1p}dt<\infty$ for some $p>\frac32$, where $\rho$ is given by
\begin{align*}
\rho(t)=\sup_{s\geq0}\sup_{A\in\mathcal{F}_{s+t,\infty},\,B\in\mathcal{F}_{0,s}}\frac{|P(A\cap B)-P(A)P(B)|}{P(B)},\quad t\geq0,
\end{align*}
where $\mathcal{F}_{s,t}=\sigma(\xi(u), u\in[s,t])$.
 \item $\xi(\cdot)\in C^\infty(\R_+)$, a.s. and $|\xi^{(k)}(t)|\leq M$, a.s., $k=0,1,2,\cdots,K$ for some non-random $M>0$.
 \item $E\bigl[\xi(t)\bigr]=0$.
\end{enumerate}
Let us define $\dot{w}^\e(t)$ as 
\begin{align}  \label{eq:4A}
\dot{w}^\e(t)=A^{-1} \psi(\e) \xi\bigl(\psi(\e)^{2} t \bigr),
\end{align}
where $\psi(\e)$ is taken as in \eqref{eq:def of psi(e)} and
$A=\large\{2\int_0^\infty E[\xi(0)\xi(t)]dt\large\}^{1/2}$.
Then, by a simple computation, we obtain the following lemma.
 
\begin{lem}\label{lem:Example of wepsilon2}
The process $\dot{w}^\e(t)$ defined by \eqref{eq:4A} satisfies
\begin{align*}
  & |\dot{w}^\e|_{k,T}\leq \frac{Mk}{A}{\psi(\e)}^{2k+1}, \quad k \in \Z_+.
\end{align*}
In particular, Assumption 1.1 holds with the choice of
$H_\e = M n_1(K) \psi(\e)^{2n_1(K)+1}/A$.
\end{lem}
\begin{rem}
The process ${w}^\e(t)$ converges to the Brownian motion 
${w(t)}$ as $\e\downarrow0$ in law,
see \cite{F99}.  Therefore, for the condition \eqref{eq:definition 2 of g} in
{\rm a.s.}-sense, we need to apply Skorohod's representation theorem.
\end{rem}
%
\subsection{Error estimate}
In this subsection, coming back to the situation we discussed 
in Corollary \ref{lem:Lemma for H}, 
we estimate the error term $\delta_k^\e$ defined by
\begin{align}\label{eq:error delta_k}
\delta_k^\e(t,x)
:=\frac{\partial u^\e_k}{\partial t}-\Delta u^\e_k
 -\frac{1}{\e^2}(f(u^\e_k)-\e \la^\e_k)+v(t),
\end{align}
with $v(t) = \a \dot{w}^\e(t)$.  The corresponding quantity is introduced in
\cite{CHL}, (9), but in our case, its bound involves an extra slowly diverging
factor $G_\e$.

More precisely, in Sections 4.2--4.4, we assume a (non-random) $v(t)$
is given, and $u_k$ and $u_k^\pm$ determined from this $v(t)$ satisfy
the bound \eqref{eq:Bounded case 01} in Corollary \ref{lem:Lemma for H}
with $T(\om)$ replaced by $T$ and for simplicity for every $\e>0$, i.e.,
\begin{align}\label{eq:4.Cor3.7}
\sup_{(t,\rho,s)\in  [0,T]\times \R\times \mathcal{S}}
\bigl\{|u_k(t,\rho,s)|,|u_k^{\pm}(t)|\bigr\}\leq G_\e,\quad 0\leq k \leq K,
\end{align}
holds for every $\e>0$.  Later, we will apply the results obtained 
in these three sections for $v=\a\dot{w}^\e$ and $T=T(\om)$.

\begin{lem}\label{lem:Lemma 2-1}
There exists $C_{err}=C_{err}(K,T)>0$ such that
\begin{align}
\sup_{(t,x)\in[0,T]\times D}|\delta_k^\e(t,x)|\leq \e^k{C_{err}}G_\e,
\end{align}
holds for every $k=1,2, \ldots,K$ and $\e>0$.
\end{lem}
\begin{proof}
Note that the terms of order $O(\e^i)$, $i=0,1,\cdots,k+1$ do not appear in
$\de_k^\e$ by the method of the construction of approximate solutions.
First, let us consider for $(t,x)$ satisfying $|d(t,x)|\le\delta$, recall \eqref{eq:definition u_k}. 
Using the expansion of $u_k^\e$ and $\lambda_k^\e$, defined by
\eqref{eq:approximate-sol} and \eqref{eq:definition u_k}, combined with 
\eqref{eq:change variable 01}, \eqref{eq:change variable 02}, and from
the condition \eqref{eq:4.Cor3.7}, we obtain 
\begin{align*}
\sup_{(t,x)\in[0,T]\times D: |d(t,x)|\le \de}|\delta_k^\e(t,x)|\leq \e^k C_{err,1}G_\e,
\end{align*}
for some constant $C_{err,1}=C_{err,1}(K,T)$.

On the other hand, for $(t,x)$ satisfying $|d(t,x)|\geq2\delta$, 
proceeding similarly to the case of inner solutions, we obtain 
\begin{align*}
\sup_{(t,x)\in[0,T]\times D: |d(t,x)|\ge 2\de}|\delta_k^\e(t,x)|\leq \e^k C_{err,2}G_\e,
\end{align*}
for some constant $C_{err,2}=C_{err,2}(K,T)$.

Finally, at the points satisfying $\delta\leq|d(t,x)|\leq2\delta$, we have
$|u^\e_k(t,x) - u_{k,\e,\pm}^{\text{out}}(t,x)|=O(e^{-\frac{\zeta \delta}{4\e}})$
by  \eqref{eq:compat kth} and \eqref{eq:uk outer} (see also \cite{CHL}, p.\ 547).
By setting $C_{err}= C_{err,1}\vee C_{err,2}$, the proof is complete.
\end{proof}
We prepare another lemma; recall that $\Phi_k^\e(t)$ is defined 
just below \eqref{eq:MASS CONS U_K^E}.

\begin{lem}\label{lem:Lemma PHI_E^K}
There exists $C_{err}^0=C_{err}^0(K,T)>0$ such that
\begin{align}
|\Phi_k^\e(t)|_{0,T}\leq \e^k C_{err}^0 G_\e,
\end{align}
for $k=0,1,\ldots,K$.
\end{lem}
\begin{proof}
Recall that ${\overline{B}_{k,i}}$ has a prefactor $\e^{k+1}$ and $(B_2)$, $(C_2)$
decrease to $0$ exponentially fast as $\e\downarrow0$. 
Thus, the estimate for $|\Phi_k^\e(t)|_{0,T}$ is obtained by the condition
\eqref{eq:4.Cor3.7}.
\end{proof}

We may assume $C_{err}\ge C_{err}^0$ by taking $C_{err} \vee C_{err}^0$ for
$C_{err}$.

\subsection{The Allen-Cahn operator}
The goal of this subsection is to show Lemma \ref{lem:Lemma 2.3}, i.e.,
the lower bound of the spectrum of the Allen-Cahn operator $-\e\De-\e^{-1}f'(v_k^\e)$, 
which is a linearization of the nonlinear Allen-Cahn
equation around $v_k^\e$ defined by
\begin{align}
&v_k^\e(t,x)=u_k^\e(t,x)- \frac{1}{|D|}\int_0^t\Phi_k^\e(s)ds, \quad 0\leq k \leq K, \label{eq:v_K-1}
\end{align}
for $x\in V_{3\de}^{t}$, where $\Phi_k^\e$ is the function defined as in \eqref{eq:MASS CONS U_K^E}.

An estimate similar to that in  Lemma \ref{lem:Lemma 2.3} is stated in \cite{CHL}, (10)
(the condition $\int_D\phi dx=0$ is unnecessary), in which they consider the linearization
around $u_k^\e$.  In our case, we need to take $v_k^\e$ instead of $u_k^\e$, since
the vanishing condition $\int_D(u^\e(t,x)-v_k^\e(t,x))dx=0$ holds under this choice
because of the effect of $v(t)$, as we will see in Lemma \ref{lem:mass of R = 0}
below.  This vanishing condition is needed for Lemma \ref{lem:CHL lemma 2},
which is shown by applying Poincar\'e's inequality.

The argument to show Lemma \ref{lem:Lemma 2.3} relies on \cite{C94}, Section 2.
Note that the correction term of $v_k^\e$ from $u_k^\e$ (i.e., the second term of 
\eqref{eq:v_K-1}) is small, but it involves a slowly diverging factor $G_\e$ as is seen
in Lemma \ref{lem:Lemma PHI_E^K}.  This gives the factor $G_\e$ in 
Lemma \ref{lem:Lemma 2.3}, differently from \cite{C94}, Theorem 2.3.

We denote by $v_k^\e(t,\rho,s)$ the inner solution of $v_k^\e$ viewed under the 
coordinate $(t,\rho,s)$ defined by \eqref{eq:related variables}.
By Lemma \ref{lem:Lemma PHI_E^K} and the condition \eqref{eq:4.Cor3.7}, 
it follows that $v_k^\e$, $2\leq k \leq K$ is rewritten into
\begin{align}\label{eq:representation of uk}
v_k^\e(t,\rho,s)=m(\rho)-\e \la_0(t)\theta_1(\rho)+\e^2q^\e(t,\rho,s),
\end{align}
where $\th_1(\rho)$ is the function given below \eqref{eq:fundamental solution m}
and $q^\e(t,\rho,s)$ is a  function satisfying
\begin{align}\label{eq:estimate for q}
\sup_{\rho\in\R}\sup_{s\in \mathcal{S}}\sup_{\e\in(0,1]}|q^\e(\cdot,\rho,s)|_0\leq \CerrH.
\end{align}
%
%

For measurable and integrable functions $\Psi(t,z,s)$, $\Phi(t,z,s)$, $\psi(t,r,s)$ and $\phi(t,r,s)$,
$z\in I_\e \subset \R, r \in I_1, s \in \mathcal{S}, t \in [0,T]$, we define
\begin{align}
& L^s\langle \Psi,\Phi\rangle=\int_{I_\e}\bigl(\Psi_z\Phi_z-f^\prime(v^\e_k(t,\e z,s)\Psi\Phi\bigr)J(t,\e z,s)dz,
\\
& L^s( \psi,\phi )=\int_{I_1}\bigl(\e\psi_r\phi_r-\e^{-1}f^\prime(v^\e_k(t,r,s)\psi\phi\bigr)J(t,r,s)dr,
\\
& L^0\langle\Psi,\Phi\rangle=\int_{I_\e}\bigl(\Psi_z\Phi_z-f^\prime(m(z))\Psi\Phi \bigr)dz,
\\
&\langle \Psi,\Phi \rangle=\int_{I_\e}\Psi\Phi dz,\quad \langle \Psi,\Phi \rangle_s=\int_{I_\e}\Psi\Phi J(t,\e z,s) dz,
\\
&\|\Psi\|=\langle \Psi,\Psi \rangle^\frac12,\quad\|\Psi\|_s=\langle \Psi,\Psi \rangle_s^\frac12,
\\
&(\psi,\phi)_s=\int_{I_1}\psi\phi J(t,r,s)dr,\quad |\psi|_s=(\psi,\psi)_s^\frac12.
\end{align}
where $\Psi_z(t,z,s)$, $\psi_r(t,r,s)$ represents $\frac{\partial \Psi}{\partial z}(t,z,s)$,
$\frac{\partial \psi}{\partial r}(t,r,s)$, respectively, $z=\frac{r}{\e}$,
$J(t,r,s)$ is given by \eqref{eq:Jacobi x},  and $I_\e=(-\frac1\e,\frac1\e)$. 
Under these settings, the following lemma holds.
%
\begin{lem}\label{lem:Lemma 2.2}
Let us set $\overline{\lambda}_1(t,s)=\inf_{\|\Psi\|_s=1}L^s\lan \Psi,\Psi\ran$, 
$s\in \mathcal{S}, t\in [0,T]$.
Then there exist constants $c_1$, $c_2>0$ depending on $K$, $T$ 
and $\e_{0,1}=\e_{0,1}(K,T)>0$ such that
\begin{align}
-c_1\e^2 \CerrH \leq \overline{\lambda}_1(t,s) \leq c_2\e^2 \CerrH,
\end{align}
holds for every $\e\in(0,\e_{0,1}]$ and $t\in[0,T]$.
\end{lem}
\begin{proof}
This lemma is obtained in a similar manner to \cite{C94}, Lemmas 2.1 and 2.2.
We make use of notations appearing in their proofs.  Under the transformation
$\Psi=J^{-\frac12}\hat{\Psi}$, we have
\begin{align}\label{eq:Ls}
L^s\langle\Psi,\Psi\rangle
&=L^0\langle\hat{\Psi},\hat{\Psi}\rangle
-\e \la_0(t)\int_{I_\e}f^{\prime\prime}(m)\theta_1\hat{\Psi}^2dz
\\
& \qquad
+\e^2\int_{I_\e}\tilde{q}^\e\hat{\Psi}^2dz-\e\Bigl(\hat{\Psi}^2\frac{J_r(t,\e z,s)}{2J(t,\e z,s)}\Bigr)
\bigg\vert _{z=-\e^{-1}}^{\e^{-1}},\quad t\in[0,T],\nonumber
\end{align}
where
\begin{align*}
\tilde{q}^\e(t,z,s)= \e^{-2}\Bigl[- f^\prime(v_k^\e(t,x))+f^\prime(m(z))
+\e f^{\prime\prime}(m(z))\la_0(t)\theta_1(z)\Bigr] +\frac14\Bigl[2J^{-1}J_{rr}-J^{-2}J_{r}^2\Bigr].
\end{align*}
From the condition \eqref{eq:4.Cor3.7}, Lemma \ref{lem:Lemma PHI_E^K} and 
\eqref{eq:estimate for q}, we have
\begin{align}\label{eq:q(z,s)-2}
|\e^2\tilde{q}^\e(t,z,s)|\leq  \widetilde{C} \e^2 \CerrH,
\end{align}
for some $\widetilde{C}(T,K)>0$ independently of $\e$ and $s$.
Let $\overline{\lambda}_1^0$ be the principal eigenvalue of 
$\mathcal{L}=-\partial_\rho^2 - f^\prime(m(\rho))$ in $I_\e$ with Neumann 
boundary condition and let $\Psi_1^0(\rho)$ be the normalized eigenfunction
corresponding to $\overline{\lambda}_1^0$. Set
\begin{align}\label{eq:eigenvalue 0}
\overline{\lambda}_2^0(t,s)=\inf_{\Psi \bot \Psi_1^0 \,\|\Psi\|_s=1}L^0\lan\Psi,\Psi\ran.
\end{align}

From \eqref{eq:Ls} and \eqref{eq:q(z,s)-2} and noting that 
$\|\tilde{r} J^{-\frac12}m^\prime\|_s=\|\tilde{r} m^{\prime}\|=1$, 
where $\tilde{r}=\|m^\prime\|^{-1}$,
we obtain the upper bound:
\begin{align}\label{eq:upper}
\overline{\lambda}_1(t,s) \leq& L^s\langle\tilde{r}J^{-\frac12}m^\prime,\tilde{r}J^{-\frac12}m^\prime\rangle
\\
\leq& \tilde{r}^2L^0\langle m^\prime,m^\prime\rangle - \tilde{r}^2\e \la_0(t)\int_{I_\e}f^{\prime\prime}(m)\theta_1({m^\prime})^2dz
\nonumber
\\
&+\tilde{r}^2\e^2\int_{I_\e}\tilde{q}^\e(m^\prime)^2dz + \e\tilde{r}^2\Bigl(({m^\prime})^2\frac{J_r(t,\e z,s)}{2J(t,\e z,s)}\Bigr) \bigg\vert _{z=-\e^{-1}}^{\e^{-1}}
\nonumber
\\
\leq &c_2\e^2 \CerrH,
\nonumber
\end{align}
for some $c_2=c_2(T,K)>0$ independently of $\e$ and $s$.
Here, for the last inequality, we have used \eqref{eq:decay of m}, the fact that 
$\overline{\lambda}_1^0$ does not depend on $v$ and 
is of order $O(e^{-\frac{3\zeta}{2\e}})$ for some $\zeta=\zeta(f)>0$ 
(see \cite{C94} Lemma 2.1) and furthermore
\begin{align*}
\int_\R f^{\prime\prime}(m)\theta_1({m^\prime})^2dz=0.
\end{align*}

From \eqref{eq:upper} and \eqref{eq:v2} of Corollary \ref{lem:Lemma for H}, 
we can choose some $\e_{0,1}\in(0,\e_0^*]$, (recall that $\e_0^*$ is defined in 
Section 3.3.1) such that $\overline{\la}_1(t,s)\leq\frac{\zeta}{4}$ holds for every 
$s\in \mathcal{S}$, $t\in [0,T]$ and $\e\in(0,\e_{0,1}]$.
Let $\Psi_1(t,z,s)$ be the normalized eigenfunction corresponding to 
the principle eigenvalue $\overline{\la}_1$ of 
\begin{align}
\mathcal{L}_s=-J^{-1}\frac{d}{dz}(J\frac{d}{dz})-f^\prime(v_k^\e(t,\e z,s)),
\end{align}
where $v_k^\e(t,\e z,s)$ is the function $v_k^\e(t,x)$ viewed under the coordinate
$(t,\rho,s)$ defined by \eqref{eq:related variables}.  Then, similarly to the proof of
\cite{C94} Lemma 2.1, we obtain that 
\begin{align}\label{eq:estimate for Psi}
|\Psi_1(t,z,s)|\leq C_{\Psi_1}e^{-\frac{\zeta}{2}|z|},
\end{align}
for $|z|\geq a_0+1$, where $C_{\Psi_1}>0$ is a constant independent of $\e$ and $v$, 
and $a_0>0$ is a number satisfying $\inf_{|z|\geq a_0}(-f^\prime(m(z)))\geq\frac{3\zeta}{4}$.

To show the lower bound for $\overline{\la}_1(t,s)$, for each $s\in \mathcal{S}$, $t\in [0,T]$ 
and $z\in I_\e$, set $\hat{\Psi}_1=J^{\frac12}\Psi_1$ and let us write
$\hat{\Psi}_1=a\Psi_1^0+(\Psi_1^0)^\bot$ for some $a\in\R$.
Then, by \cite{C94} Lemma 2.2, p.\ 1381, we obtain
\begin{align}\label{eq:I_e Lemma 2.2}
\int_{I_\e}f^{\prime\prime}(m)\theta_1(\hat{\Psi})^2dz=O(\|(\Psi_1^0)^\bot \|),
\end{align}
where $(\Psi_1^0)^\bot$ is the orthogonal complement of $\Psi_1^0$.  In particular,
there exist some constants $\a_i(a)>0$, $i=1,2$ such that
\begin{align}
-\a_1\|(\Psi_1^0)^\bot \|\leq\int_{I_\e}f^{\prime\prime}(m)\theta_1(\hat{\Psi})^2dz\leq\a_2\|(\Psi_1^0)^\bot \|,
\end{align}
for every $s\in \mathcal{S}$ and $t\in [0,T]$.
Using \eqref{eq:Ls} and \eqref{eq:I_e Lemma 2.2}, we  have
\begin{align}\label{eq:lower}
& \hskip -3.5mm
 \overline{\la}_1(t,s)=L^s\langle\Psi_1,\Psi_1\rangle\\
=&a^2L^0\langle\Psi^0_1,\Psi^0_1\rangle+L^0\langle{(\Psi^0_1)}^\bot,{(\Psi^0_1)}^\bot\rangle
-\e \la_0(t)\int_{I_\e}f^{\prime\prime}(m)\theta_1 (\hat{\Psi}_1)^2 dz
\nonumber\\
&+ \e^2\int_{I_\e}\tilde{q}^\e (\hat{\Psi}_1(t,z,s))^2 dz 
 -\e\Bigl((\hat{\Psi}_1(t,z,s))^2 \frac{J_r(t,\e z,s)}{2J(t,\e z,s)}\Bigr)
\bigg\vert _{z=-\e^{-1}}^{\e^{-1}}
\nonumber\\
\geq&a^2\overline{\lambda}_1^0+\overline{\la}_2^0 \|(\Psi_1^0)^\bot\|^2(t,s)-C_{\mathcal{V},T}\a_1\e(1+|v|_0)\|(\Psi_1^0)^\bot\| 
\nonumber\\
  &-\tilde{C}\e^2 \CerrH -C_{\Psi_1}e^{-\frac{\zeta}{2\e}}
\nonumber\\
\geq&a^2\overline{\la}_1^0+\overline{\la}_2^0 \|(\Psi_1^0)^\bot\|^2 -C_{\mathcal{V},T}\a_1\e(1+|v|_0)\|(\Psi_1^0)^\bot\| 
\nonumber\\
&-\tilde{C}\e^2 \CerrH -C_{\Psi_1}e^{-\frac{\zeta}{2\e}},
\nonumber
\end{align}
for all $\e\in(0,\e_{0,1}]$,
where we have used \eqref{eq:q(z,s)-2}, \eqref{eq:eigenvalue 0}, \eqref{eq:estimate for Psi} 
and \eqref{eq:I_e Lemma 2.2} for the first inequality and \eqref{eq:v1} for the last one.
Note that $\overline{\la}_1^0$ is of order $O(e^{-\frac{3\zeta}{2\e}})$, (see \cite{C94} Lemma 2.1).
Then, boundedness of $\|(\Psi_1^0)^\bot \| $ is obtained from \eqref{eq:lower}.
Indeed, \eqref{eq:lower} is a quadratic inequality for $\|(\Psi_1^0)^\bot\| $ with coefficient
$\overline{\la}_2^0(t,s)>0$.
Since \eqref{eq:upper} holds and $\lim_{\e\downarrow0}\e G_\e=0$ ($\e|v|_0$ 
also converges to $0$), boundedness of $\|(\Psi_1^0)^\bot \| $ is shown.
Thus, we obtain $\overline{\la}_1(t,s)\geq -c_1 \e^2 \CerrH$  for all  $\e\in(0,\e_{0,1}]$
with some constant $c_1>0$ independent of $\e$, $t$ and $s$.
Thus we have the lower bound of $\overline{\la}_1(t,s)$.
The proof is complete.
\end{proof}
From Lemma \ref{lem:Lemma 2.2} and \cite{C94} Theorem 2.3, we obtain the following lemma.
\begin{lem}\label{lem:Lemma 2.3}
There exists a constant $C_A>0$ independent of $\e$ and $t$ such that for every 
$\e\in(0,\e_{0,1}]$ and $\psi=\psi(t,\cdot)\in H^1(D)$, $t\in [0,T]$,
\begin{align*}
& \int_D \bigl(\e|\nabla\psi(t,x)|^2-\e^{-1}f^\prime(v_k^\e(t,x))\psi(t,x)^2\bigr)dx\\
& \qquad\qquad  \geq -C_A\e \CerrH \int_D\psi(t,x)^2dx, \quad 0\leq t \leq T,\nonumber
\end{align*}
holds.
\end{lem}
\begin{proof}
The left hand side of the above inequality is bounded from below by
$\e^{-1} \min_{s\in\mathcal{S}, t\in [0,T]} \overline{\lambda}_1(t,s)$.
Indeed, for $\psi(t,x)$ satisfying $\int_D \psi^2(t,x)dx=1$, if we define
$\Psi(t,z,s) = \e^{1/2} \psi(t,\e z,s)$, then we have $\|\Psi\|_s=1$
and easily see that $\overline{\lambda}_1(t,s) \le L^s\lan\Psi,\Psi\ran
= \e \times$ the left hand side of the above inequality.
Thus, from Lemma \ref{lem:Lemma 2.2}, the assertion is obtained.
\end{proof}
\subsection{Estimate for the difference between $u^\e_K$ and $u^\e$}
%
Similarly to \cite{CHL}, (12), we take initial data $g^\e=g^\e(x)$ of
\eqref{eq:1} or \eqref{eq:1-v} satisfying the following three conditions:
\begin{align}
&  g^\e(x)= u_K^\e(0,x)+\phi^\e(x), \label{eq:TH assump 1}
\\
&  \|\phi^\e\|_{L^2(D)} \leq C_1^{-\frac1p} \e^K,  
\label{eq:TH assump 2} 
\\
&  \int_D \phi^\e(x) dx = 0,  
\label{eq:TH assump 3} %
\end{align}
for sufficiently small $\e>0$, where $C_1>0$ is the constant appearing in 
\eqref{eq:assump eq 1 in TH} in the proof of Theorem \ref{thm:err} below and
$p=\min\{\frac4n,1\}$ is the number 
which will be given in Lemma \ref{lem:CHL lemma 2}.
Recall that $u_K^\e(0,x)$ is defined by \eqref{eq:definition u_k} with
$t=0$ and $K>\max{(n+2,6)}$ is fixed throughout
the paper.  Then the following lemma is shown by an elementary computation.
\begin{lem}\label{lem:mass of R = 0}
Let $u^\e$ be the solution of \eqref{eq:1-v} with initial data $g^\e$ satisfying
the conditions \eqref{eq:TH assump 1} and \eqref{eq:TH assump 3}.
Then, $\int_D R(t,x)dx=0$ holds, where $R(t,x)=u^\e(t,x)-v_K^\e(t,x)$. 
\end{lem}

\begin{proof}
First we decompose the integral $\int_D R(t,x) dx$ as
\begin{align}\label{eq:27-05}
  \int_D R(t,x) dx = \int_D \Bigl(u^\e(t,x)-\bar v(t) \Bigr) dx 
-  \int_D \Bigl(v_K^\e(t,x)- \bar v(t) \Bigr) dx,
\end{align}
where $\bar v(t) := \int_0^t v(s)ds$.  Then, from \eqref{eq:mass cons 3}
with $\a \dot{w}^\e(t)$ replaced by $v(t)$, we see that the first integral in
the right hand side of \eqref{eq:27-05} is independent of $t$, and thus
it is equal to $\int_D g^\e(x)dx$ by noting $\bar v(0)=0$.
On the other hand, recalling the definition \eqref{eq:v_K-1} of $v_K^\e$,
the second integral in the right hand side of \eqref{eq:27-05} is
rewritten as
\begin{align}
\int_D \Bigl(u_K^\e(t,x)-\bar v(t)\Bigr)dx - \int_0^t\Phi_K^\e(s)ds.
\nonumber
\end{align}
However, integrating \eqref{eq:MASS CONS U_K^E} from  $0$ to $t$,
the first integral is equal to
\begin{align*}
\int_D u_K^\e(0,x)dx+\int_0^t\Phi_K^\e(s)ds,
\end{align*}
so that we see that the second integral in \eqref{eq:27-05} is equal to
$\int_D u_K^\e(0,x)dx$.  Summarizing these, we have
\begin{align*}
  \int_D R(t,x) dx = \int_D \Bigl(g^\e(x) - u_K^\e(0,x)\Bigr)dx
= \int_D \phi^\e(x)dx =0,
\end{align*}
by \eqref{eq:TH assump 1} and \eqref{eq:TH assump 3}.  This completes the proof.
\end{proof}

The next lemma is taken from \cite{CHL}, p.\ 530; see p.\ 547 and middle of
p.\ 530 for the proof. 

\begin{lem}\label{lem:CHL lemma 2}
Let $D\subset\R^n$ be a bounded domain. Let $p=\min\{\frac4n,1\}$.
Then there exists $C_n(D)>0$ such that for every $R\in H^1(D)$ with $\int_D R(x) dx=0$,
\begin{align}\label{eq:CHL - Lemma1}
  &\|R\|_{L^{2+p}(D)}^{2+p}\leq C_{n}(D) \|R\|_{L^{2}(D)}^p \|\nabla R\|_{L^{2}(D)}^2,
\end{align}
holds. Furthermore, there exists a constant $C^{\prime}>0$ such that
\begin{align}\label{eq:CHL - Lemma1(2)}
  R\bigl(f(u+R)-f(u)-f^\prime(u)R\bigr)\leq C^{\prime}|R|^{2+p},
\end{align}
for every $|u|\leq 2$ and $R\in\R$.
\end{lem}

The following theorem extends \cite{CHL} Lemma 2 in our setting.
Note that the difference between $v_K^\e$ and $u_K^\e$ is small; recall
\eqref{eq:v_K-1} and Lemma \ref{lem:Lemma PHI_E^K}.

\begin{thm}\label{thm:err}
Let $u^\e$ be the solution considered in Lemma \ref{lem:mass of R = 0} and 
assume \eqref{eq:TH assump 1}--\eqref{eq:TH assump 3} for the
initial data $g^\e$.
Then, for sufficiently small $\e\in(0,\e_1]$, where $\e_1=\min(\e_{0,1},e^{-1})$,
\begin{align}\label{eq:error estimate 001}
&\sup_{t \in [0,T]}\|v_K^\e(t)-u^\e(t)\|_{L^2(D)}
\leq 
C_2 
\e^{K-1} |\log{\e}|,
\end{align}
holds for some constant $C_2=C_2(D,C_A)>0$.
\end{thm}
\begin{proof}
From \eqref{eq:v2}, we can assume
\begin{equation}\label{eq:assump eq 1 in TH}
C_1^{-\frac1p}\leq T|D|^{\frac12}\CerrH \quad \text{and} \quad \frac{2C_A\CerrH  T}{|\log\e|}\leq1,
\end{equation}
for sufficiently small $\e>0$,
where the constant $C_1=C_n(D)C^{\prime}>0$ is determined from
the two constants given in Lemma \ref{lem:CHL lemma 2}.
We can also assume
\begin{equation}\label{eq:Th eq 3}
\e|\log\e||D|^{\frac12}C_A^{-1}C_1^{\frac1p} \leq 1,
\end{equation}
for sufficiently small $\e>0$.

Similarly to the proof of Lemma 2 in \cite{CHL}, 
using \eqref{eq:CHL - Lemma1(2)} with 
$R=u^\e-v_K^\e$ and $u=v_K^\e$, we obtain
\begin{align}\label{eq:Th eq 001}
&\frac{1}{2}\frac{d}{dt}\|R(t)\|_{L^2(D)}^2 + \int_D \bigl( |\nabla R(t,x)|^2 - \e^{-2}f^{'}(v^\e_K)R^2(t,x) \bigr) dx
\\
\leq &\int_D\bigl(C^{\prime}\e^{-2}|R(t,x)|^{2+p} + |R(t,x)\delta_K^\e(t,x)|\bigr)dx
\nonumber
\\
\leq &\int_D C^{\prime}\e^{-2}|R(t,x)|^{2+p}dx + \|R(t)\|_{L^2(D)}\|\delta_K^\e(t)\|_{L^2(D)},
\quad t \in [0,T].\nonumber
\end{align}
Note that $|v_K^\e|\le 2$ holds for sufficiently small $\e>0$ because of 
the definition \eqref{eq:v_K-1} of $v_K^\e$, the construction of $u_K^\e$,
the condition \eqref{eq:4.Cor3.7} and Lemma \ref{lem:Lemma PHI_E^K} 
for $\Phi_K^\e(t)$.  However, the second term in the left hand side of
\eqref{eq:Th eq 001} can be decomposed
and be bounded from below as follows for all $\e\in(0,\e_1]$:
\begin{align*} 
 &\e^2\int_D \bigl( |\nabla R(t,x)|^2 - \e^{-2}f^{'}(v^\e_K)R^2(t,x)\bigr)dx
\\
&\quad +(1-\e^2)\int_D \bigl( |\nabla R(t,x)|^2 - \e^{-2}f^{'}(v^\e_K)R^2(t,x)\bigr)dx
\nonumber
\\
&\; \ge \e^2\|\nabla R(t)\|^2_{L^2(D)}- \bar c_1 \|R(t)\|_{L^2(D)}^2
-C_A \, \CerrH\|R(t)\|^2_{L^2(D)}.\nonumber
\end{align*}
Here, for the first term, we have used the assumption (iii) for $f$:
$f'(v_K^\e(t,x))\leq\bar c_1$, while for the second term, we have applied
Lemma \ref{lem:Lemma 2.3} with $\psi(t,x)
= R(t,x)$ and omitted the factor $(1-\e^2)$.
On the other hand, for the first term in the right hand side of \eqref{eq:Th eq 001},
we can apply the interpolation inequality \eqref{eq:CHL - Lemma1} and
finally obtain
\begin{align}\label{eq:Th eq 004}
\frac{1}{2}\frac{d}{dt}\|R(t)\|_{L^2(D)}^2\leq&
\|\delta_K^\e(t)\|_{L^2(D)}\|R(t)\|_{L^2(D)}+\big(\bar c_1+C_A \, \CerrH \big)
    \|R(t)\|_{L^2(D)}^2
\\&- \e^2 \|\nabla R(t)\|_{L^2(D)}^2\bigl(1-C_1\e^{-4}\|R(t)\|_{L^2(D)}^p\bigr),
\quad t \in [0,T],
\nonumber
\end{align}
where $C_1>0$ is defined below \eqref{eq:assump eq 1 in TH}.

Now, consider the time $T_\e \ge 0$ defined by
\begin{align}
&T_\e:=\inf \Bigl\{ t \geq 0; \; 
\|R(t)\|_{L^2(D)} \geq C_1^{-\frac{1}{p}} \e^{\frac{4}{p}}\Bigr\} \wedge T.
\nonumber
\end{align}
If the above set $\{\cdot\}$ is empty, we define $T_\e=T$.
The goal is to show $T_\e = T$ and the conclusion \eqref{eq:error estimate 001} based on this.
From \eqref{eq:TH assump 1} and \eqref{eq:TH assump 2}, we have
$R(0) = \phi^\e$ and
\begin{align}\label{eq:R_0}
\|R(0)\|_{L^2(D)}  \le  C_1^{-\frac{1}{p}} \e^K,
\end{align}
for all $0<\e\leq\e_1$, which implies that $T_\e>0$
since $K>\frac4p+2$. Note that $K>\frac4p+2$ follows from $K>\max{(n+2,6)}$ 
and the choice of $p$: $p= \min \{ \frac4n,1\}$.
On the other hand, from \eqref{eq:Th eq 004} and the definition of $T_\e$,
which guarantees the non-positivity of the last term in \eqref{eq:Th eq 004}
for $t\le T_\e$, we obtain that for every $t\in[0,T_\e]$,
\begin{align*}
\frac{d}{dt}\|R(t)\|_{L^2(D)} \leq 2C_A \, \CerrH \|R(t)\|_{L^2(D)} + \|\delta_K^\e(t)\|_{L^2(D)},
\end{align*}
where we have estimated $\bar c_1$ as $\bar c_1 \le C_A \CerrH$ 
for sufficiently small $\e>0$ from \eqref{eq:v2}.
Then, Gronwall's lemma shows that
\begin{align}\label{eq:TH estimate 001 Gronwall}
\sup_{t \in [0,T_\e]}\|R(t)\|_{L^2(D)}
\leq &e^{2C_A \, \CerrH T_\e}\Bigl(\|R(0)\|_{L^2(D)} + \int_0^{T_\e}\|\delta_K^\e(t)\|_{L^2(D)}dt\Bigr).
\end{align}
However, Lemma \ref{lem:Lemma 2-1} implies
\begin{align*}
\int_0^{T_\e}\|\delta_K^\e(t)\|_{L^2(D)}dt \leq& T_\e \e^{K}|D|^{\frac12}\CerrH,
\end{align*}
and therefore, noting $e^{2C_A \, \CerrH T_\e}\leq e^{|\log\e|}=\e^{-1}$ from 
\eqref{eq:assump eq 1 in TH}, we obtain
\begin{align*}
\sup_{t \in [0,T_\e]}\|R(t)\|_{L^2(D)} \leq& 2\e^{K-1}T|D|^{\frac12}\CerrH\\
\leq& \e^{K-1}|\log\e||D|^{\frac12}C_A^{-1} <  \e^{\frac{4}{p}}C_1^{-\frac{1}{p}},
\nonumber
\end{align*}
where we have used \eqref{eq:R_0}, \eqref{eq:TH estimate 001 Gronwall},
\eqref{eq:assump eq 1 in TH} and $T_\e \le T$ for the first inequality, 
\eqref{eq:v2} for the second for sufficiently small $\e>0$, 
and then $K>\frac{4}{p}+2$ and \eqref{eq:Th eq 3}
for the last one.  Thus, $T_\e = T$ holds, which implies that we obtain 
\eqref{eq:error estimate 001} by setting $C_2=|D|^{\frac12}C_A^{-1}$. The proof 
is completed.
\end{proof}
%
%
%
\subsection{Proof of Theorem \ref{thm:Th1}}
Let $\gamma_0$ be a smooth hypersurface in $D$ satisfying the conditions
in Theorem \ref{thm:Th1}.

\begin{defn}
The solution of \eqref{eq:2} with an initial hypersurface $\ga_0$ means 
the hypersurface $\Ga=\bigcup_{0\leq t<\si}(\ga_t\times\{t\})$ and 
$(d(t\wedge\si),w(t\wedge\si))$ $\in C^2(\mathcal{O})\times \R$, $t\ge 0,$ with
a stopping time $\si$ defined on a probability space $(\Omega,\mathcal{F},P)$ 
equipped with the filtration $(\mathcal{F}_t)_{t\geq0}$ and $\mathcal{O}$
being an open neighborhood of $\ga_0$ (cf.\ \cite{We}) such that\\
\noindent{\rm(i)}  $d(t,x)$ is an $(\mathcal{F}_t)$-adapted signed distance
of $x\in \mathcal{O}$ to $\ga_t$ for $t\in [0,\si)$ and satisfies $|\na d|=1$.\\
\noindent{\rm(ii)} $w(t)$ is an $(\mathcal{F}_t)$-Brownian motion.\\
\noindent{\rm(iii)} The following stochastic integral equation holds in
Stratonovich sense:
\begin{align}
d(t\wedge\si,x) = & d(0,x) + \int_0^{t\wedge\si}g(D^2d(u,x),d(u,x))du\\
&- \int_0^{t\wedge\si} \frac{du}{|\ga_u|} \int_{\ga_u} \Delta d(u,\bar s) d\bar s
+ \int_0^{t\wedge\si} \frac{\a|D|}{2|\ga_u|}\circ dw(u), \quad x \in \mathcal{O}.
\nonumber
\end{align}
Here $D^2d$ denotes the Hessian of $d$ and $g(A,q)=\text{tr}(A(I-qA)^{-1})$ for a symmetric matrix $A$ and $q\in\R$.
\end{defn}

It is not difficult to show that the solution $(\ga_t^\e,d^\e) \equiv
(\ga_t^{\a\dot{w}^\e}, d^{\a\dot{w}^\e})$ of \eqref{eq:2-e} with $v=\a\dot{w}^\e$
exists uniquely for each $0<\e\le 1$ by employing the known results in deterministic case
since $\dot{w}^\e$ is $C^\infty$ in $t$.
By Assumption \ref{ass:2}, a unique solution 
$\Gamma=\bigcup_{0\leq t < \sigma}(\gamma_t\times\{t\})$ and
$(d(t\wedge\si),w(t\wedge\si))$ of \eqref{eq:2} with
an initial hypersurface $\gamma_0$ also exists.

We are now in the position to give the proof of the main theorem of this paper.
%
%
\begin{proof}[Proof of Theorem \ref{thm:Th1}]
We assume Assumptions 1.1, 1.2 and 3.1, take initial data $g^\e$ of \eqref{eq:1}
satisfying the three conditions \eqref{eq:TH assump 1}--\eqref{eq:TH assump 3}
and fix $T>0$ as in Assumption 3.1.  Note that the solution of \eqref{eq:1} 
exists uniquely for $t\in [0,T]$ a.s., since $\dot{w}^\e$ is $C^\infty$ in $t$.  
By Corollary \ref{lem:Lemma for H}, 
the condition \eqref{eq:4.Cor3.7} holds for $v=\a\dot{w}^\e$ and 
$T=T(\om) \wedge \t(\om)$;
recall that $T(\om)$ determined from $\si^\e$ in Section 3.3.1 satisfies
$T(\om) \equiv T^{\e_0^*}(\om) \uparrow \si(\om)$ as $\e_0^* \downarrow 0$
a.s. and $\t(\om)$ is given below Assumption 3.1. 
Therefore, by \eqref{eq:v_K-1}, Lemma \ref{lem:Lemma PHI_E^K} and recalling 
the construction of $u_K^\e$, we see that
\begin{align}\label{eq:skorohod 01}
& \lim_{\e\downarrow 0} \sup_{t\in [0,T(\om)\wedge\t(\om)\wedge T]} \big\|v_{K}^\e(t, \cdot) 
- m\big(d^{\e}(t,\cdot)/\e\big) \big\|_{L^\infty(D)} = 0, \quad \rm{a.s.}
\end{align}
Furthermore, it follows from Assumption \ref{ass:2} that
\begin{align}\label{eq:skorohod 01-1}
&  \lim_{\e\downarrow 0} \sup_{t\in [0,T]} 
\big\|m\big(d^\e(t\wedge{\si^\e},\cdot)/\e\big)
-\chi_{\ga_{t\wedge\si}}(\cdot)\big\|_{L^2(D)} = 0, \quad \rm{a.s.}
\end{align}
Moreover, noting Corollary \ref{lem:Lemma for H} again, 
Theorem \ref{thm:err} implies 
\begin{align}\label{eq:skorohod 02}
&\sup_{t\in[0,T(\om)\wedge\t(\om)\wedge T]}\big\|u^\e(t,\cdot) - {v}_{K}^\e(t,\cdot)\big\|_{L^2(D)}
\leq C_2 \e^{K-1} |\log{\e}|,\quad a.s.
\end{align}
From \eqref{eq:skorohod 01}--\eqref{eq:skorohod 02}, we obtain
\begin{align*}
\lim_{\e\downarrow 0} \sup_{t\in \R_+} 
 \big\|u^\e(t\wedge\si^\e\wedge\t\wedge T,\cdot) 
- \chi_{\ga_{t\wedge \si\wedge\t\wedge T}}(\cdot) 
\big\|_{L^2(D)} = 0, \quad \rm{a.s.}
\end{align*}
The proof is completed, since $T=T^{\e_0^*}(\om) \uparrow \si(\om)$ as 
$\e_0^*\downarrow 0$.
\end{proof}
%

%
\section{Local existence and uniqueness for the limit dynamics \eqref{eq:2}}
In this section we consider the stochastically perturbed volume preserving mean curvature
flow \eqref{eq:2}, which appears in the limit.  We write $c= |D|/2$ for simplicity.
As explained in Section 1, the stochastic term destroys the volume conservation law.
Here, we restrict ourselves in two-dimension and discuss under the situation
that the closed curve $\ga_t$ stays strictly convex.  We prove the
local existence and uniqueness of the stochastic evolution
governed by \eqref{eq:2} by extending the method employed in \cite{F99},
and show that Assumption 3.1 holds up to some stopping time $T=\t(\om)>0$
assuming Assumption 1.2 in a.s.-sense and then prove Assumption 1.2 in law sense
for $\dot{w}^\e$ given by \eqref{eq:4A}.  The difference between a.s.-sense
and law sense can be filled by applying Skorohod's theorem and changing
the probability space as we pointed out in Section 1.

A strictly convex closed plane curve  $\ga$  can be parameterized 
by  $\th\in S := [0,2\pi)$  in terms of the Gauss map, that is,
the position  $X_0$  on  $\ga$  is denoted by  $X_0(\th)$  if the 
angle between one fixed direction  $\bold e := (1,0)$  in the 
plane  $\R^2$  and the outward normal  $\Vec{n}(X_0)$  at  $X_0$  to  
$\ga$  is  $\th$.  The set $S$ or a unit circle in $\R^2$ plays a role
of the reference manifold $\mathcal{S}$.  
We further denote by  $\k = \k(\th) >0$  the 
curvature of  $\ga$  at  $X_0 = X_0(\th)$.  Under these notation, 
the dynamics \eqref{eq:2} is rewritten into
the stochastic integro-differential equation for $\k =\k(t,\th)$:
\begin{equation} \label{eq:kappa}
\frac{\partial \k}{\partial t} = \k^2
\frac{\partial^2 \k}{\partial \th^2} + \k^3 - \k^2\cdot \bar\k
+ \frac{c\a \k^2}{|\ga|} \circ \dot{w}(t),
\end{equation}
where $\bar\k$ denotes the average of $\k$ over the curve $\ga=\ga_t$
and $|\ga|$ stands for the length of $\ga$; heuristically \eqref{eq:kappa}
is derived by applying \cite{Gu} p.\ 17, (2.20) with $V=\k-\bar\k
+ \frac{c\a}{|\ga|} \dot{w}(t)$ and see also \cite{F99}.
The volume (length) element
$d\bar\th$ on $\ga$ is given by $d\bar\th =  |\partial_\th X_0(\th)| d\th$. 
Since $X_0(\th)\in \R^2 \cong \mathbb{C}$ is written as
$$
X_0(\th) = X_0(0) - \sqrt{-1}\int_0^\th \frac{e^{\sqrt{-1}\th'}}{\k(\th')}d\th',
$$
we see that $|\partial_\th X_0(\th)| = 1/\k(\th)$. Therefore, $\bar\k$ and $|\ga|$
are given by
\begin{align}
& |\ga| := \int_S |\partial_\th X_0(\th)| d\th= \int_S \frac{d\th}{\k(\th)},
    \label{eq:2.5.1} \\
& \bar\k := \frac{1}{|\ga|} \int_S\k(\th)|\partial_\th X_0(\th)| d\th
= \frac{2\pi}{|\ga|}       \label{eq:2.5.2}
=2\pi\Bigl(\int_S \frac{d\th}{\k(\th)}\Bigr)^{-1}, 
\end{align}
respectively,  which are functionals of $\k=\{\k(\th); \th\in S\}$.

As in \cite{F99}, we introduce a cut-off because of the singularity
in \eqref{eq:kappa}.  
For $L\in \N$, we define a cut-off function $\chi_L\in C_b^\infty(\R)$
(in particular, $\chi_L, \chi_L', \ldots$ are all bounded)
such that $\chi_L(x)=x$ for $x\in [1/L,L]$ and $1/2L\le \chi_L(x)\le 2L$ 
for all $x\in \R$, and set
\begin{align}
& |\ga|_L \equiv |\ga|_L(\k(\cdot)) = \int_S \frac{d\th}{\chi_L(\k(\th))},
 \notag \\
& a_L(\k) = \chi_L^2(\k), \label{eq:5.L}  \\
& b_L(\th,\k(\cdot)) = \chi_L^3(\k(\th)) - 2\pi\chi_L^2(\k(\th))|\ga|_L^{-1},
 \notag \\
& h_L(\th,\k(\cdot)) = c\a\chi_L^2(\k(\th))|\ga|_L^{-1}. \notag
\end{align}
Fixing $L$ for a while and denoting $a_L, b_L, h_L$ by $a,b,h$
for simplicity, we consider the stochastic integro-differential equation 
with cut-off for $\k =\k(t,\th)$:
\begin{equation} \label{eq:kappa-cut}
\frac{\partial \k}{\partial t} =  a(\k)\frac{\partial^2 \k}{\partial \th^2} 
   + b(\cdot,\k) + h(\cdot, \k) \circ \dot{w}(t), \quad t>0, \, \th \in S.
\end{equation}

We also consider the dynamics \eqref{eq:2-e} with $v(t)= \a \dot{w}^\e(t)$,
which is described by the integro-differential equation:
\begin{equation} \label{eq:kappa-cut-smear}
\frac{\partial \k}{\partial t} =  a(\k)\frac{\partial^2 \k}{\partial \th^2} 
   + b(\cdot,\k) + h(\cdot, \k) \dot{w}^\e(t), \quad t>0, \, \th \in S,  
\end{equation}
by replacing $\dot{w}(t)$ in \eqref{eq:kappa-cut} with $\dot{w}^\e(t)$.
We gave two examples of smooth noises $\dot{w}^\e(t)$ in Section 4.1.
Since $\dot{w}^\e(t)$ is smooth in $t$, \eqref{eq:kappa-cut-smear}
has a unique solution for every $\om$.  

We now show that Assumption \ref{ass:3} holds with the choice
$\mathcal{V}= \{\a\dot{w}^\e; 0<\e\le \e_0^*\}$ 
in the setting of this section and assuming Assumption 1.2 in a.s.-sense.
The dynamics \eqref{eq:2-e} is rewritten into the integro-differential equation for $\k =\k(t,\th)$:
\begin{equation} \label{eq:kappa-v}
\frac{\partial \k}{\partial t} = \k^2
\frac{\partial^2 \k}{\partial \th^2} + \k^3 - \k^2\cdot \bar\k
+ \frac{c\k^2}{|\ga^v|}v,
\end{equation}
where the averaged curvature $\bar\k$ and the length $|\ga^v|$ are determined 
from $\k=\k(t)$ by \eqref{eq:2.5.2} and \eqref{eq:2.5.1}, respectively.  Note that
the curve  $\ga_t^v =\{X_0(t,\th) \in \R^2; \th \in S\}$ is recovered from
$\{\k(s); s\le t\}$ and $X_0(0,0)$ as
\begin{align}  \label{eq:2.5}
X_0(t,\th) & = X_0(0,0) + \left( \int_0^\th \frac{\sin \th'}{\k(t,\th')}d\th', 
   - \int_0^\th \frac{\cos \th'}{\k(t,\th')}d\th' \right)  \\
&\quad +\left(\int_0^t (\k(s,0)+\bar{\k}(s)-\frac{c}{|\ga_s|}v(s)) ds,
   \int_0^t \frac{\partial}{\partial\th}\k(s,0)ds\right).  \notag
\end{align}
Once $\ga_t^v$ is determined, one can define the signed distance function
$d^v(t,x)$.

We consider the time 
\begin{align*} 
&\t_L^v=\inf \left\{ t>0; m(\k_t)>L \text{ or dist}(\ga_t,\partial D)<\frac1L \text{ or }
\inf_{x \in V_{3\de}^t}(I_2+d(t,x)\nabla ^2 d(t,x))<\frac1L \right\},
\end{align*}
where $I_2$ is the unit matrix, $m(\k)=\max_{\th\in S} \left\{\k(\th), \k(\th)^{-1},
 |\k^{(1)}(\th)|\right\}$,
$\k^{(n)} = \partial^n \k/ \partial \th^n$, $\k_t=\k(t)\equiv \k^v(t)$ is the solution 
of \eqref{eq:kappa-v}, and $\ga_t=\ga_t^v$ and $d(t,x)=d^v(t,x)$ are determined as above.
We define the stopping time $\t_L^\e:= \t_L^v$ with $v=\a\dot{w}^\e$; in other words, $\t_L^\e$
is the stopping time for the solution of \eqref{eq:kappa-cut-smear}.
Note that the uniqueness of solutions implies that the solution of 
\eqref{eq:kappa-cut-smear} coincides with that of \eqref{eq:kappa-v} with 
$v=\a\dot{w}^\e$ for $0\le t\leq\t^\e_L$.  The reason that we care 
$I_2+d\na^2 d$ will be clear in \eqref{eq:ax01-04}. 

We are now in the position to show Assumption 3.1 assuming Assumption 1.2
in a.s.-sense.

\begin{lem}\label{lem:x(0) and Th in 2D case}
For every $L\in\N$, there exist $N=N(K)\in\N$, a stopping time $T=\t(\om)>0$ and
$C=C(C_{\mathcal{V},T},K,T,L)>0$ such that
\begin{align}
 &\sup_{1\leq i \leq 2}\sup_{\th\in S}|\partial_t^k \partial_\th^m X_0^i(\cdot,\th)|_{0,T}\leq C(1+|v|_{N,T})^{N},\label{eq:der x(0)}\\
 &\sup_{x\in V^t_{3\delta}}|\partial_t^k \partial^{\bold{m}}\th(\cdot,x)|_{0,T}\leq C(1+|v|_{N,T})^{N}, \label{eq:der Th} 
\end{align}
for every $v \in \mathcal{V} = \{\a\dot{w}^\e; 0<\e\le \e_0^*\}$,
$k=0,1,\cdots,K$ and $m, |\bold{m}|\leq M$, where $\th(t,x) = \bold{S}(t,x)\in S$ is the
inverse function of $x= X_0(t,\th) + r \Vec{n}(X_0(t,\th))$ defined below
\eqref{eq:parameterize}.
\end{lem}

\begin{proof}
Using \eqref{eq:2.5.1} and \eqref{eq:2.5.2}, \eqref{eq:kappa-v} and
\eqref{eq:2.5} are rewritten into
\begin{equation} \label{eq:kappa-2}
\frac{\partial \k}{\partial t} = \k^2
\frac{\partial^2 \k}{\partial \th^2} + \k^3 - 2\pi\k^2\Bigl(\int_S \frac{d\th'}{\k(\th')}\Bigr)^{-1}
- c\k^2 \Bigl(\int_S \frac{d\th'}{\k(\th')}\Bigr)^{-1} v,
\end{equation}
and
\begin{align}  \label{eq:2.5-2}
& \quad X_0(t,\th)  = X_0(0,0)+ \left( \int_0^\th \frac{\sin \th'}{\k(t,\th')}d\th', 
   - \int_0^\th \frac{\cos \th'}{\k(t,\th')}d\th' \right)  \\
& + \left(\int_0^t \Bigl[\k(s,0)+2\pi\Bigl(\int_S \frac{d\th'}{\k(s,\th')}\Bigr)^{-1}-c\Bigl(\int_S \frac{d\th'}{\k(s,\th')}\Bigr)^{-1}v(s)\Bigr] ds,
   \int_0^t \frac{\partial}{\partial\th}\k(s,0)ds \right),  \notag
\end{align}
respectively.  Differentiating the both sides of \eqref{eq:2.5-2} in $t$, $\partial_t\k(t,\th)$ 
appears from the derivative of the right hand side
and $\partial_t\k(t,\th)$ is represented by the right hand side of \eqref{eq:kappa-2}.
In addition, $\dot{v}$ appears 
from the derivative of the right hand side.

On the other hand, differentiating the both sides of \eqref{eq:2.5-2} in $\th$, $\partial_\th\k(t,\th)$ 
appears from the derivative of the right hand side.
Since $\k$, $\bar{\k}$, $\partial_\th\k$ and $\partial_\th^2\k$ are bounded 
up to $\t_L^\e$ due to the definition of $\t_L^\e$, 
we have 
\begin{align}\label{eq:ax01-06}
 \sup_{1\leq i \leq 2}\sup_{\th\in S}|\partial_t\partial_\th X_0^i(\cdot,\th)|_{0,T}\leq C\cdot C_{\mathcal{V},T}(1+|v|_{1,T}),
\end{align}
with some $C>0$ and $T=\t(\om):=\inf_{0<\e\le\e_0^*} \t_L^\e(\om)>0$.
Note that $\t(\om)>0$ a.s.\ due to Assumption 1.2 in a.s.-sense.
 Estimates for higher order derivatives of $X_0(t,\th)$ in $t$ and $\th$ are obtained 
recursively by direct computations using
\eqref{eq:2.5-2} and \eqref{eq:kappa-2}.
Therefore, we can conclude that the bound \eqref{eq:der x(0)} holds with  some $N= N_1
\equiv N_1(K)\in\N$.

Next, we show \eqref{eq:der Th}.
From \eqref{eq:eq of V}, noting again that $\ga_t^v$ is parametrized by $\th\in S$, we have
\begin{align}\label{eq:ad01}
 \partial_t d(t,X_0(t,\th)) = \Delta d(t,X_0(t,\th)) - \frac{1}{2\pi} \int_0^{2\pi} \Delta d(t,X_0(t,\th'))
 J^0(t,\th')d\th'+ \frac{c}{|\ga_t^v|} v(t),
\end{align}
where $J^0(t,\th) = |\partial_\th X_0(t,\th)| \big(=1/\k(t,\th)\big)$ defined by
\eqref{eq:2-J0} noting that $\bold{n}(t,\th)$ and $\partial_\th X_0(t,\th)$ are
perpendicular, and, from \eqref{eq:boldn}, \eqref{eq:parameterize} and \eqref{eq:nablad}, we have
\begin{align}\label{eq:ax01}
 x=X_0(t,\th)+d(t,x)\nabla d(t,X_0(t,\th)),\quad |d(t,x)|<3\de.
\end{align}
From now on, we omit the variables $(t,\th)$ in short. Differentiating \eqref{eq:ax01} in $t$, we have
\begin{align}\label{eq:ax01-01}
 \R^2 \ni 0=&\partial_tX_0+\partial_\th X_0 \partial_t \th + \partial_t d \na d + d\na \partial_t d\\
   &+d\na^2 d \partial_t X_0 + d \na^2 d \,\partial_\th X_0 \partial_t \th,
   \nonumber
\end{align}
where $\th=\th(t,x)$ is the inverse function of \eqref{eq:ax01}.
Substituting \eqref{eq:ad01} into \eqref{eq:ax01-01} and noting $v$ does not depend on $x$, we have
\begin{align}\label{eq:ax01-02}
 0=&\partial_tX_0+{\partial_\th X_0} \partial_t \th + \Bigl(\De d - \frac{1}{2\pi}\int_0^{2\pi} \De d
  |\partial_\th X_0|  + \frac{c}{|\ga^v|} v\Bigr)\na d\\
   &+(\na\De d)d +d\na^2 d \partial_t X_0 + d \na^2 d \,{\partial_\th X_0} \partial_t \th.
   \nonumber
\end{align}
Thus, \eqref{eq:ax01-02} is rewritten into
\begin{align*} 
 &(I_2+d\na^2d){\partial_\th X_0} \partial_t \th \\ 
 =&-\partial_tX_0 - d\na^2 d \,\partial_t X_0 - (\na\De d)d - \Bigl(\De d - \frac{1}{2\pi}\int_0^{2\pi} \De d
|\partial_\th X_0| + \frac{c}{|\ga^v|} v\Bigr)\na d.
 \nonumber
\end{align*}\label{eq:ax01-03}
Since $(I_2+d\na^2d)$ is equal to the Jacobian $J$ of the diffeomorphism map from $x$ to $(r,s)$, 
it is invertible. Thus,
\begin{align}\label{eq:ax01-04}
 &\partial_\th X_0 \partial_t \th = -(I_2+d\na^2d)^{-1}\\
 &\times\Bigl(
 \partial_tX_0 + d\na^2 d \partial_t X_0 + (\na\De d)d + \Bigl(\De d - \frac{1}{2\pi}\int_0^{2\pi} \De d
|\partial_\th X_0| + \frac{c}{|\ga^v|} v\Bigr)\na d 
 \Bigr).
 \nonumber
\end{align}
Note that $|\partial_\th X_0|\ge \frac1L$, $|I_2+d\na^2d|\ge \frac1L$, 
$|\ga_t^v|\ge \frac{2\pi}L$ hold for $0\le t \le \t_L^\e$.
Indeed, as we saw above \eqref{eq:2.5.1}, $|\partial_\th X_0|=1/\k
\geq\frac1L$ up to the time $\t_L^\e$.
The second is clear by the definition of $\t_L^\e$, while 
$|\ga_t^v|\ge \frac{2\pi}L$ follows from \eqref{eq:2.5.1}.
Therefore, we obtain from \eqref{eq:ax01-04} and \eqref{eq:ax01-06}
\begin{align}\label{eq:ax01-07}
\sup_{t\in [0,T]}  \sup_{x\in V^t_{3\delta}}
|\partial_t \th(t,x)| \leq C\cdot C_{\mathcal{V},T}(1+|v|_{1,T}),
\end{align}
with some $C>0$ and $T= \t(\om)$. Thus, we have obtained the estimates for $\partial_t \th$.
Estimates for higher order derivatives of $\th(t,x)$ in $t$ are obtained by 
differentiating \eqref{eq:ax01-04} recursively
and combined with \eqref{eq:der x(0)}.

Differentiating  the both sides of \eqref{eq:ax01-04} in $x_j$, $j=1,2$, 
the estimate for $\partial_t\partial_{x_j}\th(t,x)$ is
obtained by that for $\eqref{eq:der x(0)}$ with $k=1$ and $|\bold{m}|=1$.
Similarly, estimates for derivatives of $\partial^{\bold{m}}\th(t,x)$ in $t$ are shown
by recursively differentiating \eqref{eq:ax01-04} in spatial variables.
Therefore, we can derive the upper bound of each of them of the form
$C\cdot C_{\mathcal{V},T}(1+|v|_{N_2,T})^{N_2}$, 
with some $N_2=N_2(K)\in\N$. 
In particular, \eqref{eq:der x(0)} and \eqref{eq:der Th} also hold for $N=\max(N_1,N_2)$.
Thus, the proof is complete.
\end{proof}

The next task is to show that Assumption 1.2 holds in law sense in our setting.  We prepare the
following theorem which gives the construction of the solution of \eqref{eq:kappa-cut}
and therefore the local solution of \eqref{eq:kappa} in law sense, and shows the
convergence in law of the solution of \eqref{eq:kappa-v} with $v(t) = \a \dot{w}^\e(t)$
to that of \eqref{eq:kappa} locally in time.  We use the usual martingale
method.  Since this is similar to \cite{F99}, Section 5, the details are omitted;
see Proposition \ref{prop:5.15} and Theorem \ref{thm:5.6} stated below.  
The noise $\dot{w}^\e(t)$ is taken same
as the second example given in Section 4.1, i.e., as in \eqref{eq:4A}.

\begin{thm} \label{thm:tight}
For each $m \in \Bbb N$, $L>0$  and  $T>0$,  let  $P^\e$ be the distribution
of the solution $\k^\e(t,\cdot)$  of \eqref{eq:kappa-cut-smear}
on  $C([0,T], C^m(S))$.  Then,  $\{P^\e\}_{0<\e<1}$  is tight.
\end{thm}

To prove this theorem, we need to show the following two
propositions; see Theorems 6.1 and 4.2 in \cite{F99}
for the first proposition and Proposition 4.1 in \cite{F99}
for the second.  Note that the assertion of Proposition \ref{prop:1.3.0}
is a little weaker than \cite{F99}, because an additional term appears
in the bound given in Lemma \ref{lem:F lem 6-7}.

\begin{prop}\label{prop:1.3.0}
There exist stopping times $\tilde\si^\e \equiv \tilde\si_L^\e, 0<\e<1,$ such that
$\tilde\si^\e>0$ a.s., 
\begin{equation} \label{eq:5.tilde-si}
\lim_{\e\downarrow 0} P(\tilde\si^\e>T)=1,
\end{equation}
for every $T>0$ and 
\begin{equation}  \label{eq:uniform-moment-estimate}
\sup_{0 < \e < 1} E\left[ \sup_{0\le t \le T}
  \|\k^{(n)}(t\wedge\tilde\si^\e)\|_{L^p(S)}^p \right] < \infty,
\end{equation}
for every  $n \in \Bbb Z_+ =\{0,1,2,\ldots\}$ and  $p \ge 1$,
where  $\k(t) = \k^\e(t)$  is the solution of \eqref{eq:kappa-cut-smear}.
\end{prop}

\begin{prop} \label{prop:1.3}
For every  $\fa\in C^\infty(S)$  and $p\ge 0$,
$$
E[|\lan \k^\e(t_2) - \k^\e(t_1), \fa \ran |^p]
  \le C(t_2 - t_1)^{p/2}, \qquad  0 \le t_1 < t_2 \le T,
$$
where $\lan\k,\fa\ran = \int_S\k(\th)\fa(\th) d\th$.
\end{prop}

Once we have the uniform moment estimates \eqref{eq:uniform-moment-estimate},
relying on the criterion due to Holley and Stroock (see \cite{F92}),
the conclusion of Theorem \ref{thm:tight} follows from
the weak tightness of  $\{\k^\e(t)\}_{0<\e<1}$ shown in Proposition 
\ref{prop:1.3}.  Indeed, \eqref{eq:uniform-moment-estimate} implies the
tightness of  $\{\k^\e(t\wedge\tilde\si^\e)\}_{0<\e<1}$ and this shows
the conclusion from \eqref{eq:5.tilde-si}.

\begin{proof}[Proof of Proposition \ref{prop:1.3}]
We follow the proof of Proposition 4.1 in \cite{F99} and only sketch 
the proof of the proposition.  The difference from \cite{F99} is that 
we need to replace $h'(\k(t,\theta))$ 
by  $Dh(\theta,\k(t))$, the Fr\'echet derivative in $\k$.
Indeed, the integral form of our equation \eqref{eq:kappa-cut-smear} is
\begin{align*}
 \k(t_2,\cdot)-\k(t_1,\cdot)=\int_{t_1}^{t_2} a(\k(s,\cdot))
\frac{\partial ^2\k(s,\cdot)}{\partial \theta^2}ds
 +\int_{t_1}^{t_2} b(\cdot,\k(s))ds+\int_{t_1}^{t_2} h(\cdot,\k(s)) \dot{w}^\e(s)ds.
\end{align*}
To complete the proof, it suffices to show 
\begin{align*}
 E\left[|X(t)|^p\right]\leq C(t-t_1)^{\frac{p}{2}},\quad t\geq t_1, 
\quad p\in2\N,
\end{align*}
for $X(t)=\int_{t_1}^{t} \Phi(\k(s))\dot{w}^\e(s)ds, t\geq t_1$,
where $\Phi(\k) = \int_S h(\theta,\k)\fa(\theta)d\theta$.
Since the Fr\'echet derivatives toward
$\psi \in L^2(S)$ of $h(\th,\k)$ and $|\ga|_L(\k)$ are computed as
\begin{align}  \label{eq:D-h}
 &D h(\theta,\k)(\psi)= \frac{2c\a \chi_L^{\prime}(\k(\theta))\chi_L(\k(\theta))}{|\ga|_L(\k)}\psi(\theta)
 +\frac{c\a \chi_L(\k(\theta))^2}{|\ga|_L(\k)^2}
\Big\langle \frac{\chi_L^{\prime}(\k)}{\chi_L(\k)^2},\psi \Big\rangle,\quad \\ 
& D |\ga|_L (\k)(\psi) = - \Big\langle\frac{\chi_L'(\k)}{\chi_L^2(\k)},\psi \Big\rangle,  \label{eq:D-ga}
\end{align}
respectively, we have
\begin{align*}
  &X(t)^p=p\int_{t_1}^{t}\dot{w}^\e(s)ds\Bigl[ \int_{t_1}^s \Psi_1(r) dr 
+ \int_{t_1}^s \Psi_2(r)\dot{w}^\e(r)dr\Bigr],
\end{align*}
where
\begin{align*}
 \Psi_1(r)=&X(r)^{p-1}\Bigl[\int_S \frac{2c\chi_L^{\prime}(\k(r,\theta))\chi_L(\k(r,\theta))}{|\ga|_L(\k(r))}
 \\& \times \Bigl( a(\k(r,\theta))\frac{\partial ^2\k(r,\theta)}{\partial \theta^2}+b(\theta,\k(r))+h(\theta,\k(r))\dot{w}^\e(r)\Bigr) \fa(\th) d\theta\\
 +\int_S&\frac{c\chi_L(\k(r,\theta))^2}{|\ga|_L(\k(r))^2} \Big\langle \frac{\chi_L^{\prime}(\k(r))}{\chi_L(\k(r))^2}, a(\k(r,\cdot))\frac{\partial ^2\k(r,\cdot)}{\partial \theta^2}+b(\cdot,\k(r))+h(\cdot,\k(r))\dot{w}^\e(r) \Big\rangle \fa(\th)d\theta
 \Bigr],\\
\\
 \Psi_2(r)=&(p-1)X(r)^{p-2}(\Phi(\k(r))^2\\
 &+X(r)^{p-1}\int_S \Bigl[\frac{2c\chi_L^{\prime}(\k(r,\theta))\chi_L(\k(r,\theta))}{|\ga|_L(\k(r))}h(\theta,\k(r))\\
 &+\frac{c\chi_L(\k(r,\theta))^2}{|\ga|_L(\k(r))^2} \Big\langle \frac{\chi_L^{\prime}(\k(r))}{\chi_L(\k(r))^2}, h(\cdot,\k(r))\dot{w}^\e(r) \Big\rangle \Bigr] \fa(\th)d\theta.
\end{align*}
Noting that $a$, $b$, $h$, $\chi_L$ and its derivative $\chi_L^{\prime}$ are bounded,
similarly to the proof of Proposition 4.1 in \cite{F99}, the conclusion is shown.  
\end{proof}

The proof of Proposition \ref{prop:1.3.0} can be completed similarly to those of
Theorem 6.1 and Lemmas 6.1 to 6.16 in \cite{F99}.  The difference between
\cite{F99} and ours is that the coefficients $b=b(\th,\k)$ and $h=h(\th,\k)$
of the SPDE \eqref{eq:kappa-cut-smear} are functionals of $\k=\{\k(\th);\th\in S\}$
 in our case.
This requires some more careful computations, 
although most of the proof is similar to that in \cite{F99}.  Another difference is
that, in \cite{F99}, the noise is taken as $\dot{w}^\e(t)=\e^{-\ga}\xi(\e^{-2\ga} t)$,
$\ga>0$ with $\xi$ introduced in Section 4.1 and hence $|\dot{w}^\e(t)|\le M\e^{-\ga}$ holds. 
In our case, $\dot{w}^\e$ is given by \eqref{eq:4A} and satisfies $|\dot{w}^\e(t)|\leq \frac{M}{A}\psi(\e)$,
where $\psi(\e)$ is the function of $\e$ defined as \eqref{eq:def of psi(e)}.
Therefore, we need to replace $\e^{-\ga}$ appearing in lemmas of \cite{F99} by $\psi(\e)$, 
but we can obtain similar results to that of \cite{F99}.
Here we indicate only the different points in the proof from that in \cite{F99}.

Before giving the proof of Proposition \ref{prop:1.3.0}, we prepare several lemmas
parallel to \cite{F99}.
In the rest of this section, we assume $p\in 2\N$.
By directly computing from \eqref{eq:kappa-cut-smear}, we have
\begin{align}\label{eq:F 6-6}
&\frac{d}{dt}||\k_t^{(n)}||_{L^p}^p=p\Bigl(\Phi_1^{(n)}(\k_t)+\Phi_2^{(n)}(\k_t)+\Phi_3^{(n)}(\k_t)\dot{w}^\e(t)\Bigr),
\end{align}
where $L^p=L^p(S)$ and
\begin{align*}
&\Phi_1^{(n)}(\k)=\int_S \{\k^{(n)}(\th)\}^{p-1}\{a(\k(\th))\k^{(2)}(\th)\}^{(n)}d\th,\\
&\Phi_2^{(n)}(\k)=\int_S \{\k^{(n)}(\th)\}^{p-1}\{b(\th,\k)\}^{(n)}d\th,\\
&\Phi_3^{(n)}(\k)=\int_S \{\k^{(n)}(\th)\}^{p-1}\{h(\th,\k)\}^{(n)}d\th.
\end{align*}
Then, $\frac{d}{dt}\Phi_3^{(n)}(\k_t)$ can be decomposed into 
\begin{align}
&\frac{d}{dt}\Phi_3^{(n)}(\k_t)=\Psi_1^{(n)}(\k_t)-\Psi_2^{(n)}(\k_t)\dot{w}^\e(t),
\end{align}
where
\begin{align*}
\Psi_1^{(n)}(\k)=&(p-1)\int_S \{\k^{(n)}(\th)\}^{p-2}\{h(\th,\k)\}^{(n)}\{a(\k(\th))\k^{(2)}(\th)+b(\th,\k)\}^{(n)}d\th\\
&+\int_S \{\k^{(n)}(\th)\}^{p-1}\Bigl[Dh(\th,\k)\Big(a(\k(\cdot))\k^{(2)}+b(\cdot,\k)\Big)\Bigr]^{(n)}d\th,\\
\Psi_2^{(n)}(\k)=&(p-1)\int_S \{\k^{(n)}(\th)\}^{p-2}\Bigl[\{h(\th,\k)\}^{(n)}\Bigr]^2d\th\\
&+\int_S \{\k^{(n)}(\th)\}^{p-1}\{Dh(\th,\k)(h(\cdot,\k))\}^{(n)}d\th.
\end{align*}
Set 
\begin{align}\label{eq:F 6-7}
  & \psi_p^{(n)}(\k)=\int_S\{\k^{(n)}(\th)\}^{p-2}\{\k^{(n+1)}(\th)\}^2d\th,\quad p\geq2,\,n\geq1.
\end{align}
We denote by $\mathcal{P}_n$ the family of polynomials of the forms 
$P(y_1,\cdots,y_{n-1},z;\k)=\sum_\a g_{1,\a}(\k)g_{2.\a}(z)y^\a$, 
$y_i\in \R, 1\le i \le n-1$, $\k\in\R$, $z\in L^2(S)$
with $g_{1,\a}\in C_b^{\infty}(\R)$, $g_{2,\a}\in C_b^{\infty}(L^2(S))$, 
$\a=(\a_1,\ldots,\a_{n-1})\in\Z_+^{n-1}$,
$y^\a=y_1^{\a_1}\ldots y_{n-1}^{\a_{n-1}}$ and the sum $\sum_{\a}$ finite, where
$C_b^{\infty}(L^2(S))$ stands for the family of infinitely Fr\'echet differentiable
functions on $L^2(S)$ having bounded derivatives.

The term $\Phi_1^{(n)}$ coincides with that appearing in \cite{F99}, Lemma 6.1,
so that we have the following lemma.
\begin{lem}(Estimate for $\Phi_1^{(n)}$)\,\label{lem:F lem 6-1}
For $n\geq1$, there exist constants $c=c(n,p,L)$ and $C=C(n,p,L)>0$  
such that
\begin{align}\label{eq:F 6-8}
  \Phi_1^{(n)}(\k)&\leq -c\psi_p^{(n)}(\k)+C\Bigl\{||\k^{(n)}||_{L^p}^p+\sum_{i=1}^3||P_i||_{L^p}^p\Bigr\},
\end{align}
for some $P_i=P_i(\k^{(1)}(\th),\cdots,\k^{(n-1)}(\th),\k;\k(\th))\in\mathcal{P}_n$ with
$g_{2,\a,i}\equiv1$, $i=1,2,3$, that is,
$P_i$ of the forms $P_i(y_1,\cdots,y_{n-1},z;\k)=\sum_\a g_{1,\a,i}(\k)y^\a$ with 
$g_{1,\a,i}\in C_b^{\infty}(\R)$, $i=1,2,3$. 
In particular, in the case $n=1$, \eqref{eq:F 6-8} holds with 
$C=0$ and $P_1=P_2=P_3=0$.
\end{lem}
As for the terms $\Phi_2^{(n)}$ and $\Phi_3^{(n)}$, we have the upper bounds for them as 
in the next lemma.

\begin{lem}(Estimates for $\Phi_2^{(n)}$ and $\Phi_3^{(n)}$)\,\label{lem:F lem 6-2}
For $n\geq1$, there exists a constant $C=C(n,p,L)>0$ such that
\begin{align*}
  |\Phi_2^{(n)}(\k)|&\leq C\Bigl\{||\k^{(n)}||_{L^p}^p+||P_4||_{L^p}^p\Bigr\}, \\ 
  |\Phi_3^{(n)}(\k)|&\leq C\Bigl\{||\k^{(n)}||_{L^p}^p+||P_5||_{L^p}^p\Bigr\}, 
\end{align*}
for some $P_i=P_i(\k^{(1)}(\th),\cdots,\k^{(n-1)}(\th),\k;\k(\th))\in\mathcal{P}_n$, $i=4,5$.
In particular, in the case $n=1$, these estimates hold with $P_4=P_5=0$.
\end{lem}
\begin{proof}
We give only the sketch of the proof. By induction in $n$, we get 
\begin{align*} 
 &\{h(\th,\k)\}^{(n)}=\partial_{\k(\th)}h(\th,\k)\k^{(n)}(\th)
+P(\k^{(1)}(\th),\cdots,\k^{(n-1)}(\th),\k;\k(\th)),
\end{align*}
for some $P$. Here $\partial_{\k(\th)}h(\th,\k)$ is defined as follows. As in
\eqref{eq:5.L}, $h(\th,\k)$ is the product of a function of $\k(\th)$ and
a functional of $\k$: $h(\th,\k) = h_1(\k(\th)) \cdot h_2(\k)$ with
$h_1(x) = c\a \chi_L^2(x)$ and $h_2(\k) = |\ga|_L^{-1}$. Then
$\partial_{\k(\th)}h(\th,\k) := \frac{\partial h_1}{\partial x}(\k(\th)) \cdot h_2(\k)$.
Recalling the definition of the cutoff functions $b$ and $h$ given in \eqref{eq:5.L},
the derivatives $\partial_{\k(\th)}^lh(\th,\k)$ and $\partial_{\k(\th)}^lb(\th,\k)$, 
$l \in \Z_+$, are all bounded. The rest of the proof is similar to that of \cite{F99}, Lemma 6.2.
\end{proof}
We have another estimate for $\Phi_3^{(n)}(\k)$.
In what follows, $\psi(\e)$ denotes the function given in \eqref{eq:def of psi(e)}.
\begin{lem}\label{lem:F lem 6-3}
For every $n\geq1$ and $\de>0$, 
\begin{align*}  
 &|\Phi_3^{(n)}(\k)|\leq\psi(\e)^{-1}(p-1)\Bigl[\de\psi_p^{(n)}(\k)+\de^{-1}||\k^{(n)}||_{L^p}^p+\de^{-1}\psi(\e)^p||P_6||_{L^p}^p\Bigr],
\end{align*}
for some $P_6=P_6(\k^{(1)}(\th),\cdots,\k^{(n-1)}(\th),\k;\k(\th))\in\mathcal{P}_n$.
\end{lem}
\begin{proof}
The proof is given by a similar method to that of \cite{F99}, Lemma 6.3.
Take $P_6=\{h(\th,\k)\}^{(n-1)}$.
Then, using $ab\leq (a^2+b^2)/2$ for $a=\k^{(n-1)}(\th)\de^{\frac12}\psi(\e)^{-\frac12}$ and $b=P_6\de^{-\frac12}\psi(\e)^{\frac12}$,
we easily get the assertion.
\end{proof}
Next, we give a bound for $\Psi_2^{(n)}$ for $n\ge1$.
\begin{lem}(Estimate for $\Psi_2^{(n)}$)\,\label{lem:F lem 6-4}
For $n\geq1$, there exists a constant $C=C(n,p,L)>0$
such that
\begin{align*}  
 &|\Psi_2^{(n)}(\k)|\leq C\Bigl[||\k^{(n)}||_{L^p}^p+\sum_{i=7}^8||P_i||_{L^p}^p\Bigr],
\end{align*}
for some $P_i=P_i(\k^{(1)}(\th),\cdots,\k^{(n-1)}(\th),\k;\k(\th))\in\mathcal{P}_n$, $i=7,8$.
\end{lem}
\begin{proof}
We use the idea of \cite{F99}, Lemma 6.4.
The assertion is shown by noting that 
\begin{align*}
  \{Dh(\th,\k)(h(\cdot,\k))\}^{(n)}&=\{\partial_{\k(\th)}\Bigl(Dh(\th,\k)(h(\cdot,\k))\Bigr)\}\k^{(n)}(\th)\\
  &+P_8(\k^{(1)}(\th),\cdots,\k^{(n-1)}(\th),\k;\k(\th)).
\end{align*}
We omit the details, since the other parts are the same as \cite{F99}, Lemma 6.4.
\end{proof}
Let us give an estimate for $\Psi_1^{(n)}$. We start with the case $n=1$.

\begin{lem}\label{lem:F lem 6-5}
For every $\de>0$, we have
\begin{align*}  
 &|\Psi_1^{(1)}(\k)|\leq C\Bigl[1+||\k^{(1)}||_{L^p}^p+\psi(\e)^{-1}\de^{-1}||\k^{(1)}||_{L^{p+2}}^{p+2}+(1+\psi(\e)\de)\psi_p^{(1)}(\k)\Bigr],
\end{align*}
for some $C=C(1,p,L)>0$.
\begin{proof}
The proof is based on that of \cite{F99}, Lemma 6.5.
Using the integration by parts formula, we have
\begin{align*}
  &\Psi_1^{(1)}(\k)=I+II+III,
\end{align*}
where, writing $\k'=\k^{(1)}$ and $\k''=\k^{(2)}$ for simplicity,
\begin{align*}
  &I=-(p-1)^2\int_S \{\k'(\th)\}^{p-2}\{\k''(\th)\}\partial_{\k(\th)}h(\th,\k)A(\th,\k)d\th,\\
  &II=-(p-1)\int_S \{\k'(\th)\}^{p}\partial_{\k(\th)}^2h(\th,\k)A(\th,\k)d\th,\\
  &III=-(p-1)\int_S \{\k'(\th)\}^{p-2}\{\k''(\th)\}Dh(\th,\k)(A(\cdot,\k))d\th,\\
  &A(\th,\k)\equiv a(\k(\th))\k''(\th)+b(\th,\k).
\end{align*}
Estimates on $I$ and $II$ are similar to those in
the proof of \cite{F99}, Lemma 6.5.
It therefore suffices to give the estimate on $III$.
From \eqref{eq:D-h}, we have
\begin{align*}
D h(\theta,\k)(A(\cdot,\k))= \frac{2c\a \chi_L^{\prime}(\k(\theta))\chi_L(\k(\theta))}{|\ga|_L(\k)}
 A(\th,\k)  +\frac{c\a \chi_L(\k(\theta))^2}{|\ga|_L(\k)^2}
\Big\langle \frac{\chi_L^{\prime}(\k)}{\chi_L(\k)^2},A(\cdot,\k) \Big\rangle.
\end{align*}
The contribution to $III$ of the first term of $D h(\theta,\k)(A(\cdot,\k))$
can be estimated similarly to \cite{F99}, Lemma 6.5.  For the second term,
since we have by the integration by parts
\begin{equation}  \label{eq:5.Int-By-Parts}
\Big| \Big\langle \frac{\chi_L^{\prime}(\k)}{\chi_L(\k)^2},A(\cdot,\k) 
\Big\rangle\Big| \le C \|\k'\|_{L^2}^2,
\end{equation}
the contribution to $III$ of the second term is estimated by
\begin{equation}  \label{eq:5.second}
C(p-1) \|\k'\|_{L^2}^2 \int_S \big| (\k')^{p-2} \k'' \big| d\th.
\end{equation}
However, by Schwarz's inequality, the above integral is bounded by
$$
\int_S \big|  (\k')^{\frac{p-2}2} \cdot (\k')^{\frac{p-2}2} \k'' \big| d\th
\le \big(\int_S |\k'|^{p-2}d\th \big)^{1/2} (\psi_p^{(1)}(\k))^{1/2}
= \|\k'\|_{L^{p-2}}^{\frac{p-2}2} (\psi_p^{(1)}(\k))^{1/2},
$$
if $p\ge 3$ and this is obvious in the case of $p=2$ by regarding
$\|\k'\|_{L^{p-2}}^{\frac{p-2}2} = (2\pi)^{1/2}$.  Therefore,
\eqref{eq:5.second} dropped the constant $C(p-1)$ is bounded by
\begin{align*}
C_p \|\k'\|_{L^{p+2}}^{\frac{p+2}2} (\psi_p^{(1)}(\k))^{1/2} 
\le \frac{C_p}2 \big\{\psi(\e)^{-1}\de^{-1} \|\k'\|_{L^{p+2}}^{p+2}
 + \psi(\e)\de \, \psi_p^{(1)}(\k)\big\},
\end{align*}
where we have used $\|\k'\|_{L^2} \le c_p \|\k'\|_{L^{p+2}}$ and 
$\|\k'\|_{L^{p-2}} \le c_p \|\k'\|_{L^{p+2}}$ with some $c_p>0$,
and set $C_p = c_p^{\frac{p+2}2}$.
This completes the proof.
\end{proof}

\end{lem}
Next we consider the the case $n\ge2$. 
For $n\ge2$, the integration by parts formula for $\Psi_1^{(n)}$, it can be decomposed into
\begin{align*}  
  &\Psi_1^{(n)}(\k)=-(p-1)(p-2)\Psi_a^{(n)}(\k)-(p-1)\Psi_b^{(n)}(\k)-(p-1)\Psi_c^{(n)}(\k),
\end{align*}
where
\begin{align*}
  &\Psi_a^{(n)}(\k)=\int_S \{\k^{(n)}(\th)\}^{p-3}\k^{(n+1)}(\th)\{h(\th,\k)\}^{(n)}\{A(\th,\k)\}^{(n-1)}d\th,\\
  &\Psi_b^{(n)}(\k)=\int_S \{\k^{(n)}(\th)\}^{p-2}\{h(\th,\k)\}^{(n+1)}\{A(\th,\k)\}^{(n-1)}d\th,\\
  &\Psi_c^{(n)}(\k)=\int_S \{\k^{(n)}(\th)\}^{p-2}\k^{(n+1)}(\th)\Bigl[Dh(\th,\k)(A(\cdot,\k))\Bigr]^{(n-1)}d\th.
\end{align*}
Then, we have the following three lemmas for $\Psi_b^{(n)}$, $\Psi_c^{(n)}$ and $\Psi_a^{(n)}$.
\begin{lem}(Estimate on $\Psi_b^{(n)}$)\label{lem:F lem 6-6}
For $n\ge2$, there exists a constant $C=C(n,p,L)>0$ such that
\begin{align*}  
 &|\Psi_b^{(n)}(\k)|\leq C\Bigl[\psi_p^{(n)}(\k)+||\k^{(n)}||_{L^p}^p+\sum_{i=9}^{12}||P_i||_{L^p}^p\Bigr],
\end{align*}
for some $P_i=P_i(\k^{(1)}(\th),\cdots,\k^{(n-1)}(\th),\k;\k(\th))\in\mathcal{P}_n$, $i=9,\ldots,12$.
\begin{proof}
The proof is similar to \cite{F99}, Lemma 6.6
by noting that
\begin{align*}  
\{h(\th,\k)\}^{(n+1)}
=& \partial_{\k(\th)}h(\th,\k)\k^{(n+1)}(\th)
 +(n+1)\partial_{\k(\th)}^2h(\th,\k)\k^{(1)}(\th)\k^{(n)}(\th)\nonumber\\
 &+Q_1(\k^{(1)}(\th),\cdots,\k^{(n-1)}(\th),\k;\k(\th))\nonumber
\end{align*}
and 
\begin{align*} 
\{A(\th,\k)\}^{(n-1)}
=& a(\k(\th))\k^{(n+1)}(\th)+(n-1) a'(\k(\th)) \k^{(1)}(\th)\k^{(n)}(\th)\\
 &+Q_2(\k^{(1)}(\th),\cdots,\k^{(n-1)}(\th),\k;\k(\th)),\nonumber
\end{align*}
where
$Q_i(\k^{(1)}(\th),\cdots,\k^{(n-1)}(\th),\k;\k(\th))\in\mathcal{P}_n$, $i=1,2$;
recall that $\partial_{\k(\th)}^lh(\th,\k)$ and $\partial_{\k(\th)}^lb(\th,\k)$,
$l \in \Z_+$, are all bounded.
The rest of the proof is the same as \cite{F99}, Lemma 6.6.
\end{proof}
\end{lem}

The following estimate is similar to \cite{F99}, Lemma 6.7,
but we need the second term additionally for this estimate.

\begin{lem}(Estimate on $\Psi_c^{(n)}$)\label{lem:F lem 6-7}
For $n\ge2$, there exists a constant $C=C(n,p,L)>0$ such that
\begin{align*} 
 |\Psi_c^{(n)}(\k)|&\leq C\Bigl[\psi_p^{(n)}(\k)+||\k^{(n)}||_{L^p}^p+\sum_{i=13}^{15}||P_i||_{L^p}^p\Bigr]
 +C ||\k^{(1)}||_{L^2}^4\Bigl(||\k^{(n)}||_{L^p}^p+||{Q}||_{L^p}^p\Bigr),
\end{align*}
for some $P_i=P_i(\k^{(1)}(\th),\cdots,\k^{(n-1)}(\th),\k;\k(\th))\in\mathcal{P}_n$, $i=13, 14, 15$
and ${Q}=$ \linebreak
${Q}(\k^{(1)}(\th),\cdots,\k^{(n-1)}(\th),\k;\k(\th))\in\mathcal{P}_n$ with $g_{2,\a}\equiv1$,
namely, ${Q}$ does not contain any functional of $\k$.
\end{lem}

\begin{proof}
As in the proof of Lemma \ref{lem:F lem 6-5},
$D h(\theta,\k)(A(\cdot,\k))$ is the sum of two terms.  
The contribution to $\Psi_c^{(n)}$ of $(n-1)$th derivative
of its first term, i.e.,   
$\frac{2c\a}{|\ga|_L(\k)} 
\Bigl\{\chi_L'(\k(\theta))\chi_L(\k(\theta))A(\th,\k)\Bigr\}^{(n-1)}$,
can be estimated similarly to \cite{F99}, Lemma 6.7
by noting again $\partial_{\k(\th)}^j h(\th,\k)$, $j=1,2$ and 
$a$, $b$, $h$ are all bounded, and we see that it is estimated 
by the first term in the right hand side of the bound stated
in the claim of this lemma.
On the other hand, the contribution to $\Psi_c^{(n)}$ of its
second term is given by
\begin{align*}
J = \frac{c\a}{|\ga|_L(\k)^2} 
\big\lan \frac{\chi_L'(\k)}{\chi_L(\k)^2},A(\cdot,\k)\big\ran
\int_S \{\k^{(n)}(\th)\}^{p-2}\k^{(n+1)}(\th)Q \, d\th,
\end{align*}
where $Q\equiv Q(\th) =\{\chi_L(\k(\th))^2\}^{(n-1)}$.  Then, $J$ is estimated as
\begin{align*}
|J| & \le C \|\k'\|_{L^2}^2
 \Big| \int_S \{\k^{(n)}(\th)\}^{p-2}\k^{(n+1)}(\th)Q \, d\th\Big| \\
& \le C \|\k'\|_{L^2}^2 \, \psi_p^{(n)}(\k)^{1/2}
\Big( \int_S |\k^{(n)}(\th)|^{p-2} Q^2 \, d\th\Big)^{1/2} \\
& \le \psi_p^{(n)}(\k) + \frac{C^2}4 \|\k'\|_{L^2}^4
\Bigl(\|\k^{(n)}\|_{L^p}^p + \|Q \|_{L^p}^p\Bigr),
\end{align*}
where we have used \eqref{eq:5.Int-By-Parts} noting $1/|\ga|_L(\k)$ is
bounded for the first line, Schwarz's inequality for the second and
finally a trivial inequality $a^{p-2}b^2 \le a^p+b^p$ for $a,b\ge 0$
for the third.  Thus, the proof is complete.
\end{proof}

Recalling again that $\partial_{\k(\th)}^lh(\th,\k)$ and 
$\partial_{\k(\th)}^lb(\th,\k)$,
$l \in \Z_+$, are all bounded and following the proof of \cite{F99}, Lemma 6.9, 
we have the next lemma.
\begin{lem}(Estimate on $\Psi_a^{(n)}$)\label{lem:F lem 6-9}
For every $n\ge2$ and $\b'\in(0,1)$, there exist constants $C,N,q>0$ such that
\begin{align*} 
 |\Psi_a^{(n)}(\k)|\leq& C\Bigl[\psi_p^{(n)}(\k)+||\k^{(n)}||_{L^p}^p+\sum_{i=16}^{19}||P_i||_{L^p}^p\Bigr]\\
 &+C\Bigl[1+\sum_{i=1}^{n-1}||\k^{(i)}||_{L^p}^N+||\k^{(n)}||_{L^q}^{\b'}\Bigr]\psi_2^{(n)}(\k),
\end{align*}
for some $P_i=P_i(\k^{(1)}(\th),\cdots,\k^{(n-1)}(\th),\k;\k(\th))\in\mathcal{P}_n$, $i=16,\ldots,19$.
\end{lem}
Now we estimate the conditional expectation of the integral of
the third term on the right hand side of \eqref{eq:F 6-6}.
Let us recall $\mathcal{F}_{s,t}=\sigma(\xi(u), u\in [s,t])$, where $\xi$ is the 
stochastic process given in Section 4.1.
Set $\mathcal{F}_{s,t}^\e=\mathcal{F}_{\psi(\e)^2s,\psi(\e)^2t}$.
Recall also that in our case $\dot{w}^\e$ is given by \eqref{eq:4A} and 
$|\dot{w}^\e(t)|\leq \frac{M}{A}\psi(\e)$ holds.
By replacing the divergent factor $\e^{-\ga}$ in \cite{F99}, Lemmas 6.10 and 6.11 
by $\psi(\e)$, we obtain the following two lemmas.  The proofs are same
and therefore omitted.

\begin{lem}\label{lem:F lem 6-10}
For $0\leq s\leq t \leq T$, we have
\begin{align*} 
 \Biggl|E\Bigl[&\int_s^t\Phi_3^{(n)}(\k_r)\dot{w}^\e(r)dr  \Bigl| \mathcal{F}_{0,s}^\e\Bigr]\Biggr|\\
 \leq& C\psi(\e)^{-1}|\Phi_3^{(n)}(\k_s)|
 +C\int_s^t E\Bigl[\psi(\e)^{-1} |\Psi_1^{(n)}(\k_r)| + |\Psi_2^{(n)}(\k_r)| \Bigl| \mathcal{F}_{0,s}^\e\Bigr]dr,
\end{align*}
for some $C=C(M,A)>0$.
\end{lem}

\begin{lem}\label{lem:F lem 6-11}
(Rough estimates uniform in $\om$)
Assume $\k_0=\k(0)\in C^{\infty}(S)$ is independent of $\e$. Then, there exist constants
$C=C(T,n,p)$ and $N=N(n,p)>0$ such that
\begin{align}
&||\k_t^{(n)}||_{L^p}\leq C\psi(\e)^N,\quad 0\leq t \leq T, \label{eq:F 6-46}\\
&\int_0^T \psi_p^{(n)}(\k_t) dt \leq C\psi(\e)^N,  \label{eq:F 6-47}\\
&||\k_t^{(n)}||_{L^p}^p\leq \Bigl[||\k_s^{(n)}||_{L^p}^p+C\psi(\e)\int_s^t
\sum_{i=1}^5||P_i(\k_r)||_{L^p}^pdr\Bigr]e^{C\psi(\e)(t-s)},  
\notag   
\end{align}
for $0\le s \le t \le T$.
In particular, if $n=1$, we can take $N=1$ in \eqref{eq:F 6-46} and $N=p$ in \eqref{eq:F 6-47}, that is,
\begin{align*}
&||\k_t^{(1)}||_{L^p}^p\leq C\psi(\e),\quad 0\leq t \leq T, \\   
&\int_0^T \psi_p^{(1)}(\k_t) dt \leq C\psi(\e)^p.  
\end{align*}
\end{lem}

Under the preparation of Lemmas \ref{lem:F lem 6-1} to \ref{lem:F lem 6-11}, 
one can complete the proof of Proposition \ref{prop:1.3.0}.
\begin{proof}[Proof of Proposition \ref{prop:1.3.0}]
The proof of \eqref{eq:uniform-moment-estimate} is completed by induction
in the following three steps as in \cite{F99} Section 6.7.
Step 1: we prove \eqref{eq:uniform-moment-estimate} for $n=1$.
Step 2: we prove
\eqref{eq:uniform-moment-estimate} for $n$ assuming that it holds for $i=1,\cdots,n-1$.
Step 3: we prove 
\eqref{eq:uniform-moment-estimate} for $n=0$.

As for Step 1, using the estimates on $\Phi_1^{(1)}$, $\Phi_2^{(1)}$, $\Phi_3^{(1)}$, 
$\Psi_1^{(1)}$ and $\Psi_2^{(1)}$, a similar method to \cite{F99} shows
\begin{align}  \label{eq:5-Step1}
\sup_{0<\e<1}E\Bigl[\sup_{t\in[0,T]}||\k_t^{(1)}||_{L^p}^p\Bigr]<\infty.
\end{align}
In Step 2, integrating the both sides of \eqref{eq:F 6-6} from $s$ to $t$ and 
taking the conditional expectation, and then, using Lemmas \ref{lem:F lem 6-1} to \ref{lem:F lem 6-11} 
and noting that $|\dot{w}^\e(t)|$ satisfies $|\dot{w}^\e(t)|\leq \frac{M}{A}\psi(\e)$, we have
\begin{align}  \label{eq:F 6-51}
&E\Bigl[||\k_t^{(n)}||_{L^p}^p \Bigl| \mathcal{F}_{0,s}^\e\Bigr]+c^{''}\int_s^tE\Bigl[\psi_p^{(n)}(\k_r)\Bigl| \mathcal{F}_{0,s}^\e \Bigr]dr\\
 &\; \leq (1+C\psi(\e)^{-1})||\k_s^{(n)}||_{L^p}^p+C\psi(\e)^{-1}||P_5(\k_s)||_{L^p}^p\nonumber\\
 &\quad +C\int_s^t E\Bigl[||\k_r^{(n)}||_{L^p}^p \Bigl| \mathcal{F}_{0,s}^\e\Bigr]dr + C\int_s^tE\Bigl[\mathcal{P}(\k_r) \Bigl| \mathcal{F}_{0,s}^\e\Bigr]dr\nonumber \\
 &\quad +C\psi(\e)^{-1}\int_s^t E\Bigl[||\k_r^{(1)}||_{L^2}^4||\k_r^{(n)}||_{L^p}^p \Bigl| \mathcal{F}_{0,s}^\e\Bigr]dr + C\psi(\e)^{-1}\int_s^tE\Bigl[||\k_r^{(1)}||_{L^2}^4||Q(\k_r)||_{L^p}^p\Bigl| \mathcal{F}_{0,s}^\e\Bigr]dr\nonumber \\
 &\quad +C\psi(\e)^{-1}(p-2)\int_s^t E\Bigl[\Bigl\{ 1+ \psi(\e)^\de+\sum_{i=1}^{n-1}||\k_r^{(i)}||_{L^q}^N \Bigr\}\psi_2^{(n)}(\k_r)\Bigl| \mathcal{F}_{0,s}^\e\Bigr]dr, \nonumber
\end{align}
for every $\e\in(0,\e_0)$ with a sufficiently small $\e_0$ in such a way that 
$\frac{C}{c}\psi(\e_0)^{-1}\ll1$, (hence, there exists some $c''>0$),
for every $\de>0$ and some $N=N(\de)$, $q=q(\de)\ge1$,
where $\mathcal{P}(\k_t)=\sum_{1\le i\le 19,i\neq5,6}||P_i(\k_t)||_{L^p}^p$.
Let us define $\tilde{\si}_\e=\inf\{t>s;\,||\k_t^{(1)}||_{L^2}>\psi(\e)^{\frac14}\}$ 
and $\tilde{\si}_\e=\infty$ if the set $\{\cdot\}$ is empty.
Then, the estimate \eqref{eq:5-Step1} implies \eqref{eq:5.tilde-si}.
Note that \eqref{eq:F 6-51} is also true with $t$ replaced by $t\wedge\tilde{\si}_\e$.
Thus, by similar manner to \cite{F99}, Section 6.7
with $\e^{-\ga}$ replaced by $\psi(\e)$, we have
\begin{align*}
\sup_{0<\e<1}E\Bigl[\sup_{t\in[0,T]}||\k_{t\wedge\tilde{\si}_\e}^{(n)}||_{L^p}^p\Bigr]<\infty,
\end{align*}
 see \cite{F99}, Lemmas 6.12 to 6.16.
Finally, Step 3 can be completed by using again Lemmas \ref{lem:F lem 6-1} to
\ref{lem:F lem 6-11} for the estimates on $\Phi_1^{(0)}$, $\Phi_2^{(0)}$, $\Phi_3^{(0)}$, 
$\Psi_1^{(0)}$ and $\Psi_2^{(0)}$.  Indeed, we obtain
\begin{align*}
\sup_{0<\e<1}E\Bigl[\sup_{t\in[0,T]}||\k_{t}||_{L^p}^p\Bigr]<\infty.
\end{align*}
 As a result, we get \eqref{eq:uniform-moment-estimate} for every $n\in \Z_+$
and $p\in 2\N$, so that for every $p\ge 1$.
\end{proof}

We now show the pathwise uniqueness for \eqref{eq:kappa-cut}.
For this, we apply the energy
inequality, cf. Lemma 5.2 in \cite{F99}.  To do this, we need to rewrite 
\eqref{eq:kappa-cut}
into It\^o's form.  The Fr\'echet derivatives toward $\psi \in L^2(S)$ of $h(\th,\k)$
and $|\ga|_L(\k)$ are computed as in \eqref{eq:D-h} and \eqref{eq:D-ga},
respectively.  Thus, we have that
\begin{align*}
h(\th,\k(t))\circ dw_t & = h(\th,\k(t)) dw_t + \frac12 d\{h(\th,\k(t))\}dw_t \\
& = h(\th,\k(t)) dw_t + \frac12 \left(\int_S Dh(\th,\k(t))(\th_1) h(\th_1,\k(t))d\th_1\right) dt\\
& = h(\th,\k(t)) dw_t + g(\th,\k(t))dt,
\end{align*}
where
\begin{align*}
g(\th,\k) = & - \frac{c\a}{|\ga|_L(\k)} \chi_L(\k(\th))\chi_L'(\k(\th))h(\th,\k) \\
& + \frac{c\a}{2|\ga|_L(\k)^2} \chi_L^2(\k(\th)) \int_S 
\frac{\chi_L'(\k(\th_1))h(\th_1,\k)}{\chi_L^2(\k(\th_1))}d\th_1.
\end{align*}
Therefore, \eqref{eq:kappa-cut} can be rewritten into It\^o's form:
\begin{equation} \label{eq:kappa-cut-Ito}
\frac{\partial \k}{\partial t} =  a(\k)\frac{\partial^2 \k}{\partial \th^2} 
   + \tilde b(\cdot,\k) + h(\cdot, \k) \dot{w}_t, \quad t>0, \, \th \in S,
\end{equation}
where
$$
\tilde b(\cdot,\k) = b(\cdot,\k)+ g(\cdot,\k).
$$

\begin{prop}  \label{prop:5.15}
(cf. Lemma 5.2 in \cite{F99})
Let  $\k_i(t), i=1,2,$  be two solutions of \eqref{eq:kappa-cut}
or equivalently \eqref{eq:kappa-cut-Ito}
satisfying  $\k_i(t) \in C([0,T],C^2(S))$ a.s. and
having the same initial data:  $\k_1(0) = \k_2(0) \in C^2(S)$.  
Then, we have
$P( \k_1(t) = \k_2(t) \, \text{ for all } \, t \in [0,T]) = 1.$
\end{prop}

\begin{proof}  By applying It\^o's formula,
$$
\|\k_1(t) - \k_2(t)\|_{L^2}^2 
  = \int_0^t \{I_1 + I_2\} \, ds + \int_0^t I_3 \, dw_s,
$$
where
\begin{align*}
& I_1 = 2 \int (\k_1 - \k_2) \{a(\k_1) \k_1^{(2)} - a(\k_2) \k_2^{(2)} \}, \\
& I_2 = 2 \int (\k_1 - \k_2) \{\tilde{b}(\cdot,\k_1) - \tilde{b}(\cdot,\k_2) \}
    + \|h(\cdot,\k_1) - h(\cdot,\k_2)\|_{L^2}^2,  \\
& I_3 = 2 \int (\k_1 - \k_2) \{h(\cdot,\k_1) - h(\cdot,\k_2) \},
\end{align*}
and the norm  $\|\, \cdot \, \|_{L^p(S)}$  will be simply denoted
by  $\|\, \cdot \, \|_{L^p}$  for  $p \in [1,\infty]$.
Noting that $a(\k)\ge a_0=1/2N >0$, the estimate on $I_1$ is
exactly the same as in \cite{F99}:
\begin{align*}
I_1 & \le -2 a_0 \|\k_1' - \k_2'\|_{L^2}^2 
  + 2 \|a'\|_{L^\infty} \|\k_1'\|_{L^\infty}   
    \{\de^{-1} \|\k_1 - \k_2\|_{L^2}^2  + \de \|\k_1' - \k_2'\|_{L^2}^2 \} \\
& \qquad\qquad + 2 \|a'\|_{L^\infty} \|\k_2^{(2)}\|_{L^\infty}
    \|\k_1 - \k_2\|_{L^2}^2,  
\end{align*}
for every  $\de >0$.  Since $D\{1/|\ga|_L\}(\th_1,\k)= - \frac1{|\ga|_L^2}
D|\ga|_L(\th_1,\k)$, \eqref{eq:D-ga} shows that $\|D\{1/|\ga|_L\}(\cdot,\k)\|_{L^2}$
is bounded in $\k$ and thus the map  $\k\in L^2 \mapsto 1/|\ga|_L(\k)\in \R$
is Lipschitz continuous.  This implies that
$$
I_2 \le C \|\k_1 - \k_2\|_{L^2}^2.
$$
Now introduce a stopping time
$$
\t_M = \inf \left\{ t>0; \max\{ \|\k_1\|_{L^\infty}, \|\k_2\|_{L^\infty}, 
  \|\k_1^{(1)}\|_{L^\infty}, \|\k_2^{(2)}\|_{L^\infty}  \} > M \right\},
    \quad  M>0.
$$
Then, by choosing  $\de >0$  sufficiently small, we have
$$
\|\k_1(t\wedge \t_M) - \k_2(t\wedge \t_M)\|_{L^2}^2 
  \le C_M  \int_0^t \|\k_1(s\wedge \t_M)- \k_2(s\wedge \t_M)\|_{L^2}^2 \,ds
  + \int_0^{t\wedge \t_M}  I_3 \, dw_s, \quad  t>0,
$$
for some  $C_M > 0$.
Therefore, taking the expectation of both sides and applying the
Gronwall's lemma, we obtain
$$
E[ \|\k_1(t\wedge \t_M)-\k_2(t\wedge \t_M)\|_{L^2}^2 ] = 0
$$
for all  $M>0$.  This completes the proof of the proposition.
\end{proof}

The pathwise uniqueness combined with the existence of the solution
in law sense implies the existence of a strong solution.  This is
the well-known Yamada-Watanabe's theorem which is extended
into an infinite-dimensional setting by \cite{Ond}, see also \cite{Ta}.
As a result, we obtain the following theorem.

\begin{thm}  \label{thm:5.6}
Let $D$ be a two-dimensional bounded domain and $\gamma_0$ be a closed convex curve given
such that $\gamma\Subset D$. Let $\k_0(\theta)$, $\theta \in S$ be the curvature of $\gamma_0$.
Then, there exists a unique local solution $\k=\k(t,\theta)$ defined for $0\leq t < \si$
and $\theta \in S$ of the SPDE \eqref{eq:kappa}, where $\si>0$ (a.s.) is a stopping
time.  In particular, the dynamics \eqref{eq:2} has a unique solution for $0\leq t < \si$.
\end{thm}

Theorem \ref{thm:tight} together with the remarks made before it
combined with Theorem \ref{thm:5.6}, which implies
the uniqueness in law, shows that Assumption 1.2 holds in law sense
in the setting of this section.

\vskip 5mm
\noindent
{\it Acknowledgement}.  
The authors thank Danielle Hilhorst for leading them to the problem
discussed in this article and stimulating discussions.
%

\end{document}